\documentclass[a4paper,12pt,twoside]{article}
\usepackage{mathrsfs,amsmath,amsfonts,amssymb,amsthm,color}
\usepackage[pdftex]{graphicx,xcolor}
\usepackage{tikz}

\numberwithin{equation}{section}
     \newtheorem{thm}{Theorem}[section]
     \newtheorem{cor}[thm]{Corollary}
     
     \newtheorem{lem}[thm]{Lemma}
\theoremstyle{definition}

     \newtheorem{rem}{Remark}[section]

\newcommand{\N}{\mathbb{N}}
\newcommand{\Q}{\mathbb{Q}}
\newcommand{\R}{\mathbb{R}}
\newcommand{\T}{\mathbb{T}}
\newcommand{\Z}{\mathbb{Z}}
\newcommand{\cD}{\mathcal{D}}

\newcommand{\cK}{\mathcal{K}}

\newcommand{\cP}{\mathcal{P}}

\newcommand{\ve}{\varepsilon}
\newcommand{\vp}{\varphi}

\newcommand{\sign}{\operatorname{sign}}

\newcommand{\tE}{\tilde{E}}
\newcommand{\tG}{\tilde{G}}

\newcommand{\dsharp}{d_{\sharp}}

\newcommand{\pt}{\phantom{***}}

\newcommand{\uU}{U}

\newcommand{\msckw}{%
\footnotetext{\hspace{-4mm} 
{\it 2010 Mathematics Subject Classification}. 
Primary 42B05
\endgraf
{\it Key words and phrases}. 
multiple Fourier series, 
lattice point problem,
Fourier transform,
Hardy's identity, 
Gibbs-Wilbraham phenomenon, Pinsky phenomenon, 
spherical partial sum, 
radial function.
\endgraf
This work was supported by JSPS KAKENHI Grant Numbers JP22540166, JP24540159, JP17K18731.
}}

\title{Multiple Fourier series and 
lattice point problems \msckw}
\author{Shigehiko Kuratsubo and Eiichi Nakai}
\date{}

\begin{document}

\baselineskip=18pt

\maketitle

\begin{abstract}
For the multiple Fourier series of the periodization of some radial functions on $\R^d$,
we investigate the behavior of the spherical partial sum.
We show the Gibbs-Wilbraham phenomenon, the Pinsky phenomenon
and the third phenomenon for the multiple Fourier series,
involving the convergence properties of them.
The third phenomenon is closely related to the lattice point problems,
which is a classical theme of the analytic number theory. 
We also prove that, for the case of two or three dimension,
the convergence problem on the Fourier series 
is equivalent to the lattice point problems in a sense.
In particular, the convergence problem at the origin in two dimension
is equivalent to Hardy's conjecture on Gauss's circle problem.
\end{abstract}

\section{Introduction}\label{sec:intro}

It is well known as the Gibbs-Wilbraham phenomenon that,
for the Fourier series of piecewise continuous functions,
in the neighborhood of each jump, 
the partial sums overshoot the jump by approx 9\% of the jump.
This phenomenon can be seen not only in one dimension 
but also in higher dimensions
(see for example \cite{Colzani-Vignati1995, Kuratsubo-Nakai-Ootsubo2010,Taylor2002}). 

In one dimension, it is also well known as the localization property that,
if the function is zero on an interval,
then the Fourier series converges to zero there.
However, in higher dimensions this property is no longer valid.
In 1993, Pinsky, Stanton and Trapa \cite{Pinsky-Stanton-Trapa1993}
showed that, 
for the Fourier series of the indicator function 
of a $d$-dimensional ball with $d\ge3$, 
the spherical partial sum diverges at the center of the ball.
That is, the Fourier series diverges 
at a smooth point - even a point of
local constancy - of the function, resulting from
global rather than local properties. 
This phenomenon is called the Pinsky phenomenon.

In 2010, 
the third phenomenon was discovered in \cite{Kuratsubo1996, Kuratsubo2010}.
Namely, for the Fourier series of the indicator function 
of a $d$-dimensional ball with $d\ge5$, 
the spherical partial sum diverges at all rational points,
while it converges almost everywhere.
This third phenomenon was proved by using 
results on the lattice point problems.

The study of lattice point problems is 
a classical theme of analytic number theory 
which is concerned with the number of integer points. 
It has a long history and deep accumulations 
since G.~Vorono\"\i, G.~H.~Hardy, E.~Landau, 
J.~G.~Van der Corput, V.~Jarn\'ik and A.~Walfisz, 
see \cite{Fricker1982, Kratzel1988, Kratzel2000}. 
For example,
in $\R^2=\{(x_1,x_2):\text{$x_1$ and $x_2 $ are real numbers}\}$,
by $A(s)$ we denote the number of lattice points,
where are points with integral co-ordinates, 
inside the circle
\begin{equation*}
 x_1^2+x_2^2=s,
\end{equation*}
see Figure~\ref{fig:lattice1}.
Then $A(s)$ is the same as the area of the polygon in Figure~\ref{fig:lattice2},
since the polygon is the union of the unit squares whose centers are the lattice points inside the circle.
Let $P(s)=A(s)-\pi s$, where $\pi s$ is the area of this circle.
Gauss showed that 
\begin{equation*}
 P(s)=O(s^{1/2}) \ \text{as $s\to\infty$}.
\end{equation*}
\begin{figure}[htbp]
\begin{minipage}{.48\linewidth}
\begin{center}
\includegraphics[width=.80\linewidth]{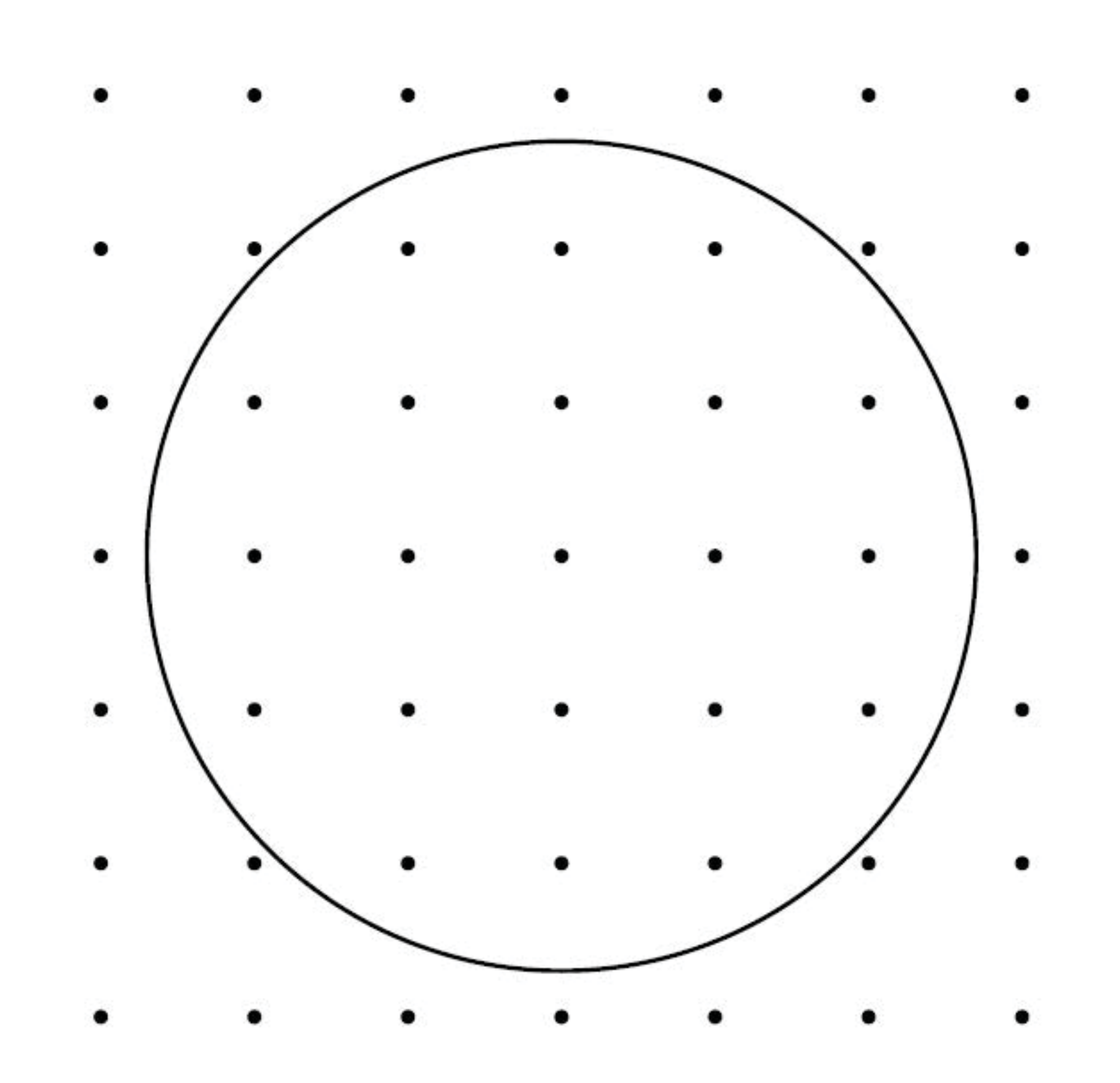} 
\caption{Gauss's circle problem}
\label{fig:lattice1}
\end{center}
\end{minipage}
\begin{minipage}{.48\linewidth}
\begin{center}
\includegraphics[width=.80\linewidth]{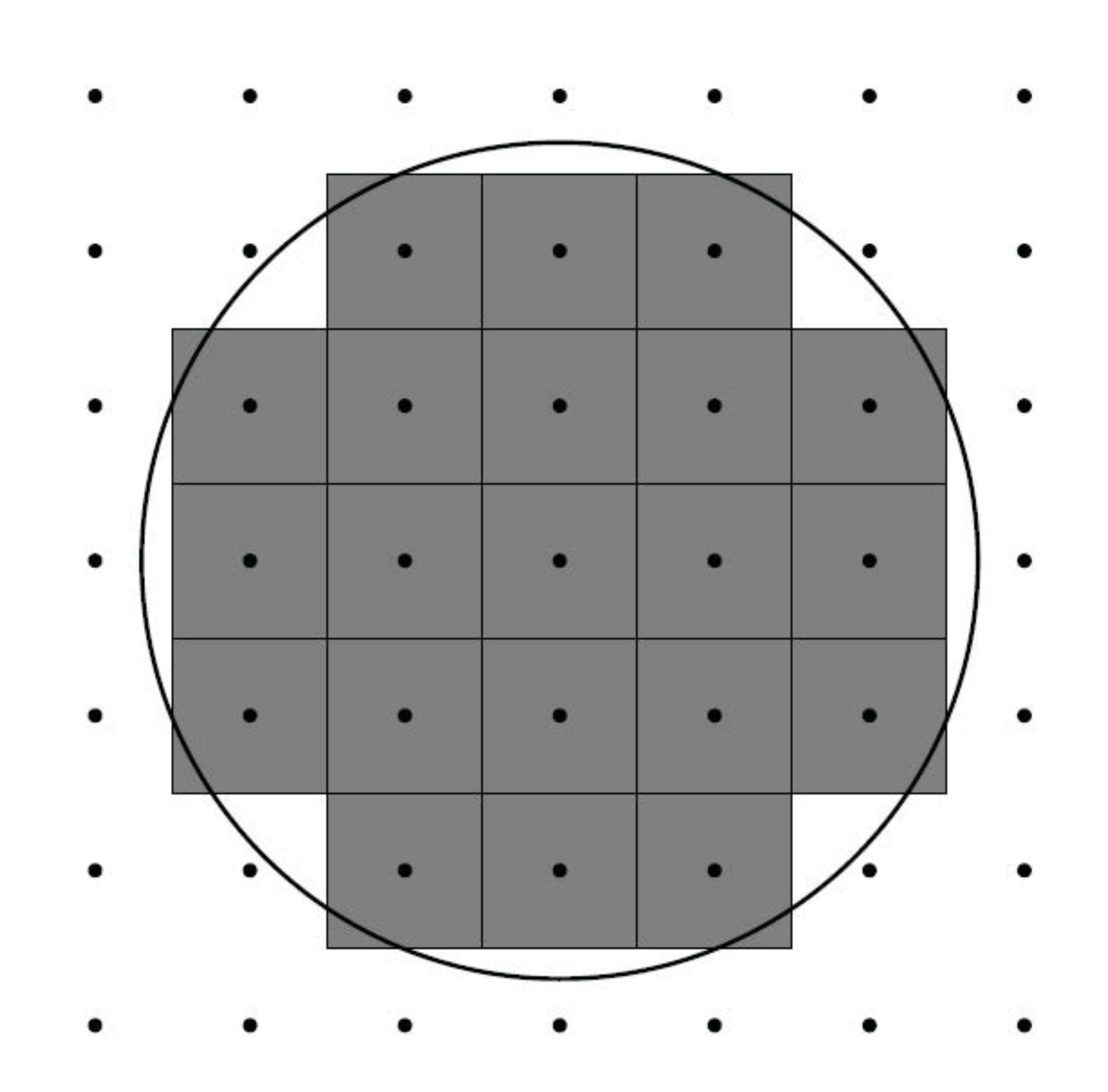} 
\caption{The union of the unit  squares.}
\label{fig:lattice2}
\end{center}
\end{minipage}
\end{figure}
In the above $O$ is Landau's symbol,
that is, $f(s)=O(g(s))$ as $s\to\infty$ 
means that $\limsup_{s\to\infty}|f(s)|/g(s)<\infty$ for the positive valued function $g$.
Similarly, $f(s)=o(g(s))$ as $s\to\infty$ means that $\lim_{s\to\infty}f(s)/g(s)=0$.
In 1915 Hardy~\cite{Hardy1915} proved that
\begin{equation*}
 P(s)\ne O(s^{\theta}) \ \text{if $\theta\le1/4$}.
\end{equation*}
In fact $P(s)\ne o(s^{1/4}\log^{1/4}s)$. 
The best bound on $\theta$ for $P(s)=O(s^{\theta})$
is a very famous open problem known as Gauss's circle problem.
While Hardy's conjecture(\cite{Hardy1917}) is $P(s)=O(s^{1/4+\ve})$ for any $\ve>0$,
the most sharp result up to now 
is $P(s)=O(s^{131/416}(\log s)^{18637/8320})$
by M.~Huxley~\cite{Huxley2003} in 2003,
where $131/416=0.3149\dots$.
(Recently, Bourgain and Watt~\cite{Bourgain-Watt_arXiv1709-04340v1} 
gave $\theta=517/1648=0.31371\dots$ in arXiv, 2017.)

Note that $A(s)$ is a special case of 
\begin{equation}\label{LPweight}
 \sum_{m_1^2+\cdots+m_d^2\le s}\exp\left(2\pi i\sum_{k=1}^dm_kx_k\right),
\end{equation}
where $m_1,\cdots,m_d$ are integers and $x_1,\cdots,x_d$ are real numbers,
that is, $A(s)$ is the case $(x_1,\cdots,x_d)=(0,\cdots,0)$ of \eqref{LPweight}
and $d=2$.
The sum \eqref{LPweight} is related to the Fourier series.
Especially the research on the sum \eqref{LPweight} 
by Czechoslovakian mathematician B.~Nov\'ak (1938--2003) 
is very important for the study of the convergence problem 
of multiple Fourier series. 

Recently, Taylor~\cite{Taylor-preprint-web,Taylor-preprint} 
found that the Pinsky phenomenon arises even in two-dimension.
He treated the radial function
\begin{equation*}\label{U2}
 U(x)
 =
 \begin{cases}
  1/\sqrt{a^2-|x|^2}, & |x|<a, \\
  0, & |x|\ge a,
 \end{cases}
\quad x\in\R^2,\ a>0, 
\end{equation*}
Then $U(x)$ is the fundamental solution to the wave equation on $\R\times\T^2$, evaluated at $t=a$.
Our aim in this paper which is  motivated by Taylor~\cite{Taylor-preprint-web,Taylor-preprint}  
is to study the Gibbs-Wilbraham phenomenon, the Pinsky phenomenon 
and the third phenomenon on
the Fourier series of 
\begin{equation}\label{U}
 U_{\beta,a}(x)
 =
 \begin{cases}
  (a^2-|x|^2)^{\beta}, & |x|<a \\
  0, & |x|\ge a,
 \end{cases}
\quad x\in\R^d,\ a>0,
\end{equation}
for $\beta>-1$, $a>0$ and dimension $d$,
involving the convergence properties of them.
If $\beta=0$, 
then $U_{\beta,a}(x)$ is the same as the indicator function of the ball 
centered at the origin and of radius $a$.

By $\R^d$, $\Z^d$ and $\T^d=\R^d/\Z^d$ 
we denote the $d$-dimensional Euclidean space, 
integer lattice and torus, respectively.
In this paper, however,
we always identify $\T^d$ with $(-1/2,1/2]^d$,
that is, $x\in\T^d$ means $x\in(-1/2,1/2]^d$ and $\T^d\subset\R^d$.
Let $\Q$ be the set of all rational numbers, and let 
$\Q^d=\{(x_1,\cdots,x_d):x_1,\cdots,x_d\in\Q\}$.

For an integrable function $F(x)$ on $\R^d$, 
its Fourier transform ${\hat F}(\xi)$ and 
its Fourier spherical partial integral $\sigma_\lambda(F)(x)$ 
of order $\lambda\ge 0$ are defined by
\begin{align}\label{FI}
 {\hat F}(\xi)
 &=
 \int_{{\R}^d}F(x)e^{-2\pi i\xi x}\,dx, 
 \quad \xi=(\xi_1,\cdots, \xi_d)\in\R^d,
\\
 \sigma_\lambda(F)(x)
 &=
 \int_{|\xi|<\lambda}{\hat F}(\xi)e^{2\pi i \xi x}d\,\xi, 
 \quad |\xi|=\sqrt{\sum_{k=1}^d {\xi_k}^2},\ x\in\R^d,
\end{align}
respectively,
where $\xi x$ is the inner product $\sum_{k=1}^d\xi_kx_k$. 
Also, for an integrable function $f(x)$ on $\T^d$, 
its Fourier coefficients ${\hat f}(m)$ and 
its Fourier spherical partial sum $S_\lambda(f)(x)$ of order $\lambda\ge 0$ 
are defined by
\begin{alignat}{3}\label{FC}
 {\hat f}(m)
 &=
 \int_{{\T}^d}f(x)e^{-2\pi imx}\,dx, 
 &\quad m=(m_1,\cdots,m_d)\in\Z^d, 
\\
 S_\lambda(f)(x)
 &=
 \sum_{|m|<\lambda}{\hat f}(m)e^{2\pi imx}, 
 &|m|=\sqrt{\sum_{k=1}^d {m_k}^2},\ x\in\T^d,
\end{alignat}
respectively.

For an integrable function $F(x)$ on $\R^d$, 
we consider its periodization 
\begin{equation}\label{periodization}
 f(x)=\sum_{m\in{\Z}^d}F(x+m), \quad x\in\T^d.
\end{equation}
Note that in \eqref{periodization} the series converges with respect to the $L^1$-norm on $\T^d$ 
and then $f$ is an integrable function on $\T^d$.
Then it is known as the Poisson summation formula that
the equation
\begin{equation}\label{Poisson}
 {\hat f}(m)={\hat F}(m), \quad m\in\Z^d
\end{equation}
holds,
see for example \cite[Theorem 2.4 (page 251)]{Stein-Weiss1971}.
The left hand side of \eqref{Poisson} is defined by \eqref{FC}
and the right hand side of \eqref{Poisson} is defined by \eqref{FI} 
with $\xi=m$.

In particular, we denote by $u_{\beta,a}(x)$ the periodization of $U_{\beta,a}(x)$.
That is,
\begin{equation}\label{periodization U}
 u_{\beta,a}(x)=\sum_{m\in{\Z}^d}U_{\beta,a}(x+m), \quad x\in\T^d.
\end{equation}
If $\beta=0$, then $u_{0,a}$ is the periodization of the indicator function of the ball
centered at the origin and of radius $a$. 
See Figure~\ref{fig:u0 d2} and also Figure~\ref{fig:Ea Ga}.
\begin{figure}[htbp]
\begin{center}
\includegraphics[width=.48\linewidth]{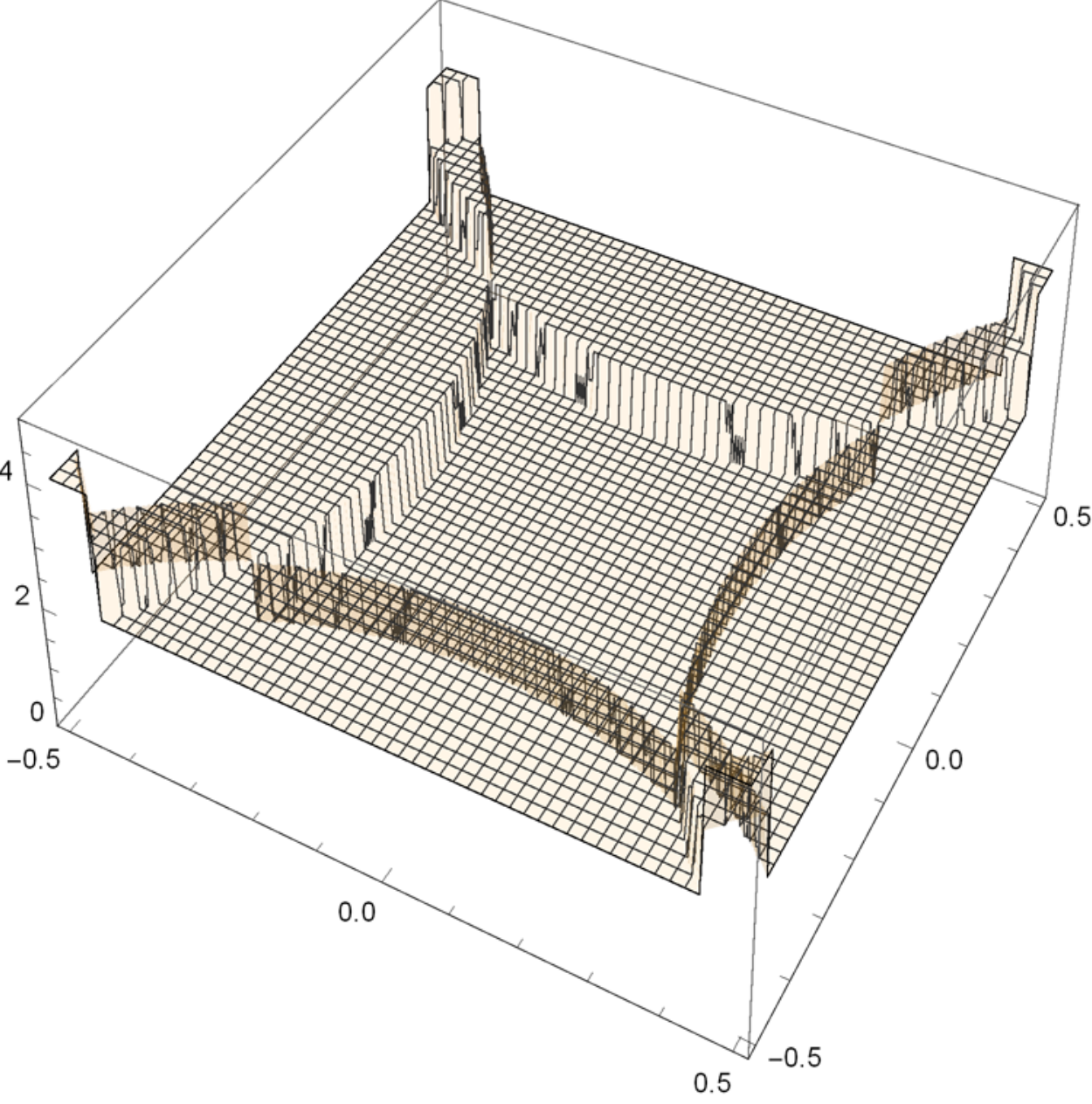} 
\caption{$u_{\beta,a}(x)$ ($d=2$, $\beta=0$, $a=3/4$).}
\label{fig:u0 d2}
\end{center}
\end{figure}
Let
\begin{equation*}
 \overline{u_{0,a}}(x)=\sum_{m\in{\Z}^d}\overline{U_{0,a}}(x+m), \quad x\in\T^d,
 \qquad
 \overline{U_{0,a}}(x)
 =
 \begin{cases}
  1, & |x|<1, \\
  1/2, & |x|=1, \\
  0, & |x|>1,
 \end{cases}
 \quad x\in\R^d.
\end{equation*}
If $d=2$, then 
\begin{equation*}
 S_{\lambda}(u_{0,a})(x)
 =
 \pi a^2 + a\sum_{0<|m|<\lambda} \frac{J_1(2\pi a|m|)}{|m|}e^{2\pi imx}
\end{equation*}
and 
\begin{equation*}
 \lim_{\lambda\to\infty}S_{\lambda}(u_{0,a})(0)
 =\overline{u_{0,a}}(0),
\end{equation*}
see Remark~\ref{rem:Hardy}.
More precisely, we have
\begin{equation}\label{HI0}
  \lim_{\lambda\to\infty}
  \left( \pi a^2 + a\sum_{0<|m|<\lambda} \frac{J_1(2\pi a|m|)}{|m|} \right)
 =
  \sum_{|m|<a}1 + \frac12 \sum_{|m|=a}1.
\end{equation}
The equation~\eqref{HI0}
is well known as Hardy's identity \cite{Hardy1915}.
See Corollary~\ref{cor:HI x} 
for $\displaystyle\lim_{\lambda\to\infty}S_{\lambda}(u_{0,a})(x)$, $x\in\T^2$.

We first show an identity (Theorem~\ref{thm:GHI})
for the periodization of any integrable radial function with compact support.
Then it turns out that 
the difference between the Fourier partial sum and the Fourier partial integral
is closely related to lattice point problems. 
Therefore, to study the convergence problem on the Fourier series of $u_{\beta,a}(x)$
we must also investigate the behavior of $\sigma_{\lambda}(U_{\beta,a})(x)$
as $\lambda\to\infty$
and lattice point problems.

In the case $\beta=0$, 
the convergence problem on the Fourier series of the function $u_{0,a}$
has been studied in detail 
by \cite{%
Kuratsubo1996,Kuratsubo1999,Kuratsubo2008,Kuratsubo2009,Kuratsubo2010,
Kuratsubo-Nakai-Ootsubo2006,Kuratsubo-Nakai-Ootsubo2010,
Pinsky-Stanton-Trapa1993,Pinsky1994,Pinsky2002}.
Some of results in these papers were proved by using 
Nov\'ak's results in \cite{Novak1967,Novak1967/1968,Novak1969,Novak1972}.
His results were very useful and sufficient 
for the affirmative results of the convergence problem on the Fourier series of 
the function $u_{0,a}$.

On the other hand,
in the case $-1<\beta<0$, 
Nov\'ak's results on the sum \eqref{LPweight} are not sufficient
to study the convergence problem on the Fourier series of $u_{\beta,a}$.
Our results on the convergence of the Fourier series of $u_{\beta,a}$
are obtained by using the best estimates up to now on lattice point problems. 
Therefore, if the lattice point problems will be improved in the future,
then our results can be also improved.
Actually, 
the convergence problem on the Fourier series 
and the lattice point problems are equivalent in a sense
as we will show in Section~\ref{sec:relation}.
In particular, the convergence problem at the origin in two dimension
is equivalent to Hardy's conjecture on Gauss's circle problem,
see Remark~\ref{rem:conj 2}.

To state our main results,
for $a>0$, let 
\begin{align}\label{Ea}
 E_a&=\{x\in\T^d: 
   x\neq 0\ \text{and}\ |x-m|\neq a\ \text{for all}\ m\in\Z^d \},
\\
\label{Ga}
 G_a&=
 \{x\in\T^d : x\ne0\ \text{and}\ |x-m|=a\ \text{for some}\ m\in\Z^d \}.
\end{align}
Then
$\T^d=\{0\}\cup G_a\cup E_a$.
See Figures~\ref{fig:Ea Ga}.
If $0<a<1/2$, then
\begin{align*}
 &E_a=\{x\in\T^d: x\ne0\ \text{and}\ |x|\neq a\},
\\
 &G_a=\{x\in\T^d: |x|=a \},
\end{align*}
because $x\in\T^d$ means $x\in(-1/2,1/2]^d$
and then $\{x\in \T^d : |x-m|=a\}$ is empty for $m\neq 0$.
For $a>0$, let also 
\begin{equation}\label{rd}
 r_d(a:x)=\sum_{m\in\Z^d, \ |x-m|=a}1,
 \qquad
 x\in\T^d.
\end{equation}
Then $r_d(a:x)=0$ for $x\in E_a$.
If $0<a<1/2$, then $r_d(a:0)=0$ and $r_d(a:x)=1$ for $x\in G_a$.

\newcommand{\radius}{1.25}
\newcommand{\radiusb}{0.75}
\newcommand{\radiush}{0.25}
\begin{figure}[htbp]
\begin{center}
\begin{tikzpicture}[scale=2.2]
\filldraw[fill=gray!50] (-0.5,-0.5)--(0.5,-0.5)--(0.5,0.5)--(-0.5,0.5)--cycle;
\draw[step=1cm,gray!20] (-1.2,-1.2) grid (1.2,1.2);
\draw[->] (-1.4,0)--(1.4,0) node[right] {$x_1$};
\draw[->] (0,-1.4)--(0,1.4) node[above] {$x_2$};
\filldraw (0,0) circle (0.02);
\node (O) at (-0.1,-0.11) {$\rm O$};
\node (O) at (0.4,0.65) {$\T^2$};
\node (O) at (1.4,1.4) {$\R^2$};
\node (O) at (0.35,0.35) {$E_a$};
\draw[->, thick] (0.6,-0.45) node[right] {$G_a$} --(0.17,-0.176);

\clip (-1.3,-1.3)--(1.3,-1.3)--(1.3,1.3)--(-1.3,1.3)--cycle;
\draw[dashed]
 (0,0) circle (\radiush)
 (1,0) circle (\radiush)
 (1,1) circle (\radiush)
 (0,1) circle (\radiush)
 (-1,1) circle (\radiush)
 (-1,0) circle (\radiush)
 (-1,-1) circle (\radiush)
 (0,-1) circle (\radiush)
 (1,-1) circle (\radiush);

\foreach \x in {-1,1}
\draw (\x cm, 0.5pt) -- (\x cm, -0.5pt) node[below] {$\bf \x$};
\foreach \y in {-1,1}
\draw (0.5pt, \y cm) -- (-0.5pt, \y cm) node[left] {$\bf \y$};

\clip (-0.5,-0.5)--(0.5,-0.5)--(0.5,0.5)--(-0.5,0.5)--cycle;
\draw[thick]
 (0,0) circle (\radiush) 
 (1,0) circle (\radiush) 
 (1,1) circle (\radiush) 
 (0,1) circle (\radiush) 
 (-1,1) circle (\radiush) 
 (-1,0) circle (\radiush) 
 (-1,-1) circle (\radiush) 
 (0,-1) circle (\radiush) 
 (1,-1) circle (\radiush);
\end{tikzpicture}
\begin{tikzpicture}[scale=2.2]
\filldraw[fill=gray!50] (-0.5,-0.5)--(0.5,-0.5)--(0.5,0.5)--(-0.5,0.5)--cycle;
\draw[step=1cm,gray!20] (-1.2,-1.2) grid (1.2,1.2);
\draw[->] (-1.4,0)--(1.4,0) node[right] {$x_1$};
\draw[->] (0,-1.4)--(0,1.4) node[above] {$x_2$};
\filldraw (0,0) circle (0.02);
\node (O) at (-0.1,-0.11) {$\rm O$};
\node (O) at (0.4,0.65) {$\T^2$};
\node (O) at (1.4,1.4) {$\R^2$};

\clip (-1.3,-1.3)--(1.3,-1.3)--(1.3,1.3)--(-1.3,1.3)--cycle;
\draw[dashed]
 (0,0) circle (\radiusb)
 (1,0) circle (\radiusb)
 (1,1) circle (\radiusb)
 (0,1) circle (\radiusb)
 (-1,1) circle (\radiusb)
 (-1,0) circle (\radiusb)
 (-1,-1) circle (\radiusb)
 (0,-1) circle (\radiusb)
 (1,-1) circle (\radiusb);

\foreach \x in {-1,1}
\draw (\x cm, 0.5pt) -- (\x cm, -0.5pt) node[below] {$\bf \x$};
\foreach \y in {-1,1}
\draw (0.5pt, \y cm) -- (-0.5pt, \y cm) node[left] {$\bf \y$};

\clip (-0.5,-0.5)--(0.5,-0.5)--(0.5,0.5)--(-0.5,0.5)--cycle;
\draw[thick]
 (0,0) circle (\radiusb) 
 (1,0) circle (\radiusb) 
 (1,1) circle (\radiusb) 
 (0,1) circle (\radiusb) 
 (-1,1) circle (\radiusb) 
 (-1,0) circle (\radiusb) 
 (-1,-1) circle (\radiusb) 
 (0,-1) circle (\radiusb) 
 (1,-1) circle (\radiusb);
\end{tikzpicture}
\caption{$E_a$ and $G_a$ for $a=1/4$ and $a=3/4$}
\label{fig:Ea Ga}
\end{center}
\end{figure}
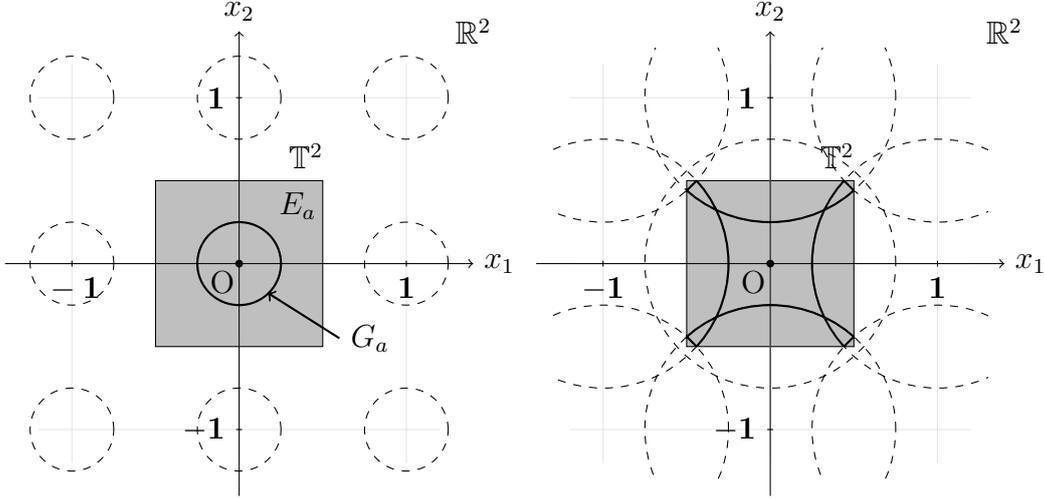

In the following, 
Theorem~\ref{thm:PinskyPh} and Corollary~\ref{cor:PinskyPh} deal with 
the behavior of $S_{\lambda}(u_{\beta,a})$ at $x=0$
including the Pinsky phenomenon, 
Theorem~\ref{thm:|x|=a and PWC} deals with the pointwise behaviors 
including the third phenomenon,
Theorem~\ref{thm:GibbsPh} deals with the Gibbs-Wilbraham phenomenon
near $G_a$. 
Theorem~\ref{thm:AEC} deals with the almost everywhere convergence.


Our first result is on the behavior of $S_{\lambda}(u_{\beta,a})(0)$
which includes the Pinsky phenomenon.
Let $\Gamma(s)$ be the Gamma function.
For $\beta>-1$, $a>0$ and the dimension $d$, let
\begin{equation}\label{P}
P_{\beta,a}^{[d]}
 =
  \frac{\Gamma(\beta+1)}{\Gamma(d/2)}
   a^{(d-3)/2+\beta} \pi^{(d-4)/2-\beta}
\end{equation}
and
\begin{equation}\label{J1}
 L_{\beta,a}
 =
 \dfrac{\Gamma(\beta+1)}2 \left(\dfrac a\pi\right)^{\beta}
 \left(\dfrac{\sin\dfrac{\beta\pi}2}{\dfrac{\beta\pi}2}\right), 
\end{equation}
where $\left(\sin\frac{\beta\pi}2\right)/\frac{\beta\pi}2$
is regarded as $1$ if $\beta=0$, that is, $L_{0,a}=1/2$. 

\begin{thm}\label{thm:PinskyPh}
Let $\beta>-1$ and $a>0$.
\begin{enumerate}
\item
If $d\ge1$ and $\beta>(d-3)/2$,
then
\begin{multline}\label{HK-I}
 S_{\lambda}(u_{\beta,a})(0)
 =
 u_{\beta,a}(0)
 + r_d(a:0)
  \left(L_{\beta,a}+o(1)\right)
  \lambda^{-\beta}
 + O(\lambda^{\frac{d-3}2-\beta})
\\
 \quad\text{as}\quad \lambda\to\infty.
\end{multline}
\item
If $d\ge 2$ and $-1<\beta\le(d-3)/2$, 
then
$S_{\lambda}(u_{\beta,a})$ reveals the Pinsky phenomenon.
More precisely, 
\begin{multline*}
 S_{\lambda}(u_{\beta,a})(0)
 =
 u_{\beta,a}(0)
 +\big(\sigma_{\lambda}(U_{\beta,a})(0)-U_{\beta,a}(0)\big)
\\
 + 
 r_d(a:0)
 (L_{\beta,a}+o(1))\lambda^{-\beta}
 +o(\lambda^{\frac{d-3}2-\beta})
 \quad\text{as}\quad \lambda\to\infty,
\end{multline*}
where
\begin{equation*}
 \sigma_{\lambda}(U_{\beta,a})(0)-U_{\beta,a}(0)
 =
-P_{\beta,a}^{[d]}
 \cos\left(2\pi a\lambda-\frac{d-1+2\beta}4\pi\right)
 \lambda^{\frac{d-3}2-\beta}
 +O(\lambda^{\frac{d-5}2-\beta}).
\end{equation*}
\end{enumerate}
\end{thm}

If $0<a<1/2$, then $r_d(a:0)=0$.
Hence we have the following corollary.

\begin{cor}\label{cor:PinskyPh}
Let $d\ge 1$ and $0<a<1/2$. 
\begin{enumerate}
\item 
If $\beta>\frac{d-3}2$, then
\begin{equation*}
\lim_{\lambda\to\infty}S_{\lambda}(u_{\beta,a})(0)=u_{\beta,a}(0).
\end{equation*}
\item
If $-1<\beta\le\frac{d-3}2$, then
\begin{align*}
 \liminf_{\lambda\to\infty}
 \frac{S_{\lambda}(u_{\beta,a})(0)-u_{\beta,a}(0)}
      {\lambda^{\frac{d-3}2-\beta}}
 &=-P_{\beta,a}^{[d]},
\\
 \limsup_{\lambda\to\infty}
 \frac{S_{\lambda}(u_{\beta,a})(0)-u_{\beta,a}(0)}
      {\lambda^{\frac{d-3}2-\beta}}
 &=P_{\beta,a}^{[d]}.
\end{align*}
\end{enumerate}
\end{cor}

For example, even if $d=2$,
we can see the Pinsky phenomena
on the graphs of $S_{\lambda}(u_{\beta,a})$ with $\beta=-1/2$ and $a=1/4$ for $\lambda\in\N$, 
as Taylor pointed out in ~\cite{Taylor-preprint-web,Taylor-preprint},
see Figure~\ref{fig:Su d2 3D}.
In this case $P_{\beta,a}^{[d]}=4$.
\begin{figure}[htbp]
\begin{center}
\includegraphics[width=.48\linewidth,height=36mm]{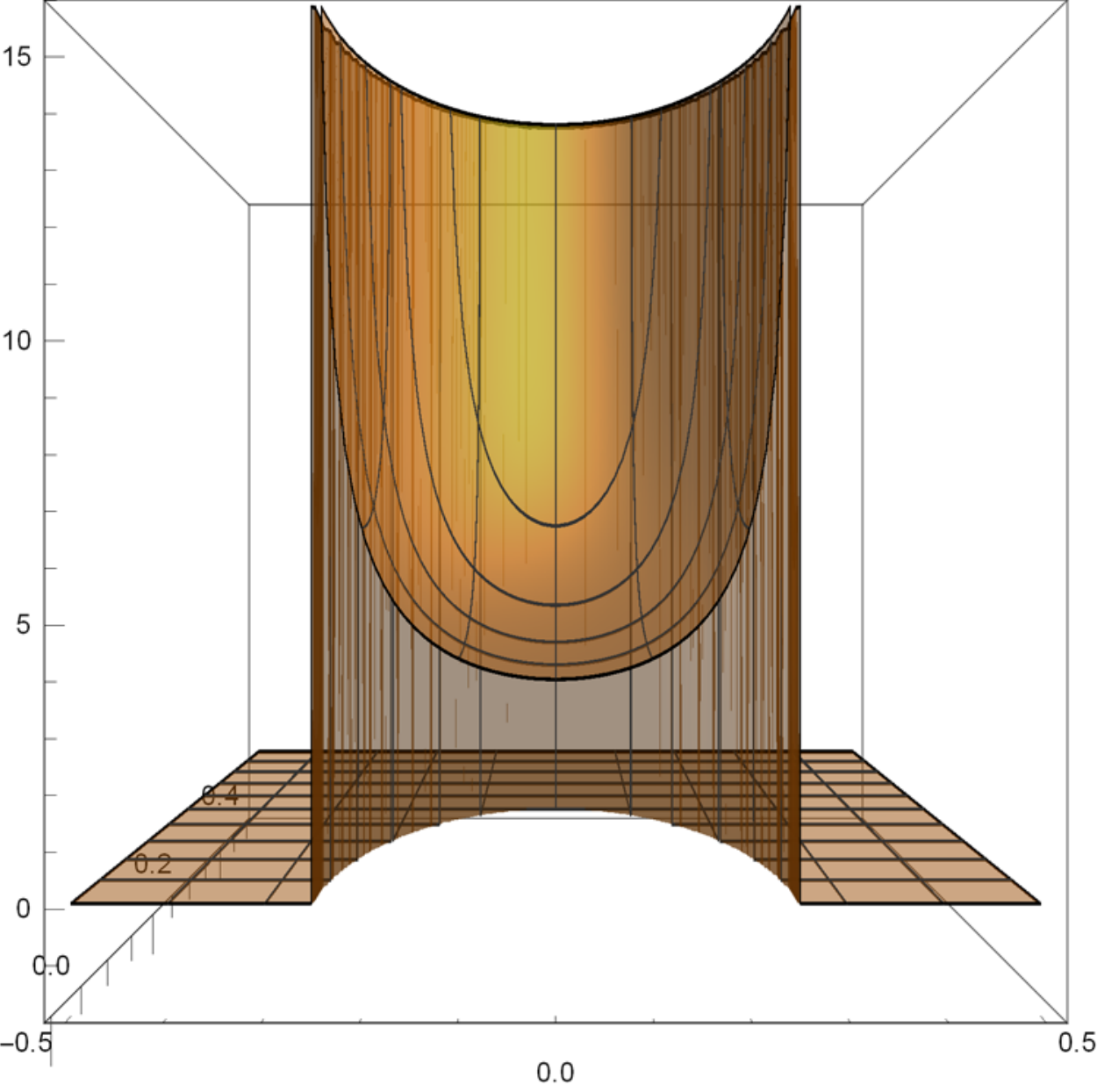} 
\caption{$u_{\beta,a}(x_1,x_2)$ ($d=2$, $\beta=-1/2$, $a=1/4$).}
\label{fig:u d2 3D}
\end{center}
\begin{center}
\includegraphics[width=.48\linewidth,height=36mm]{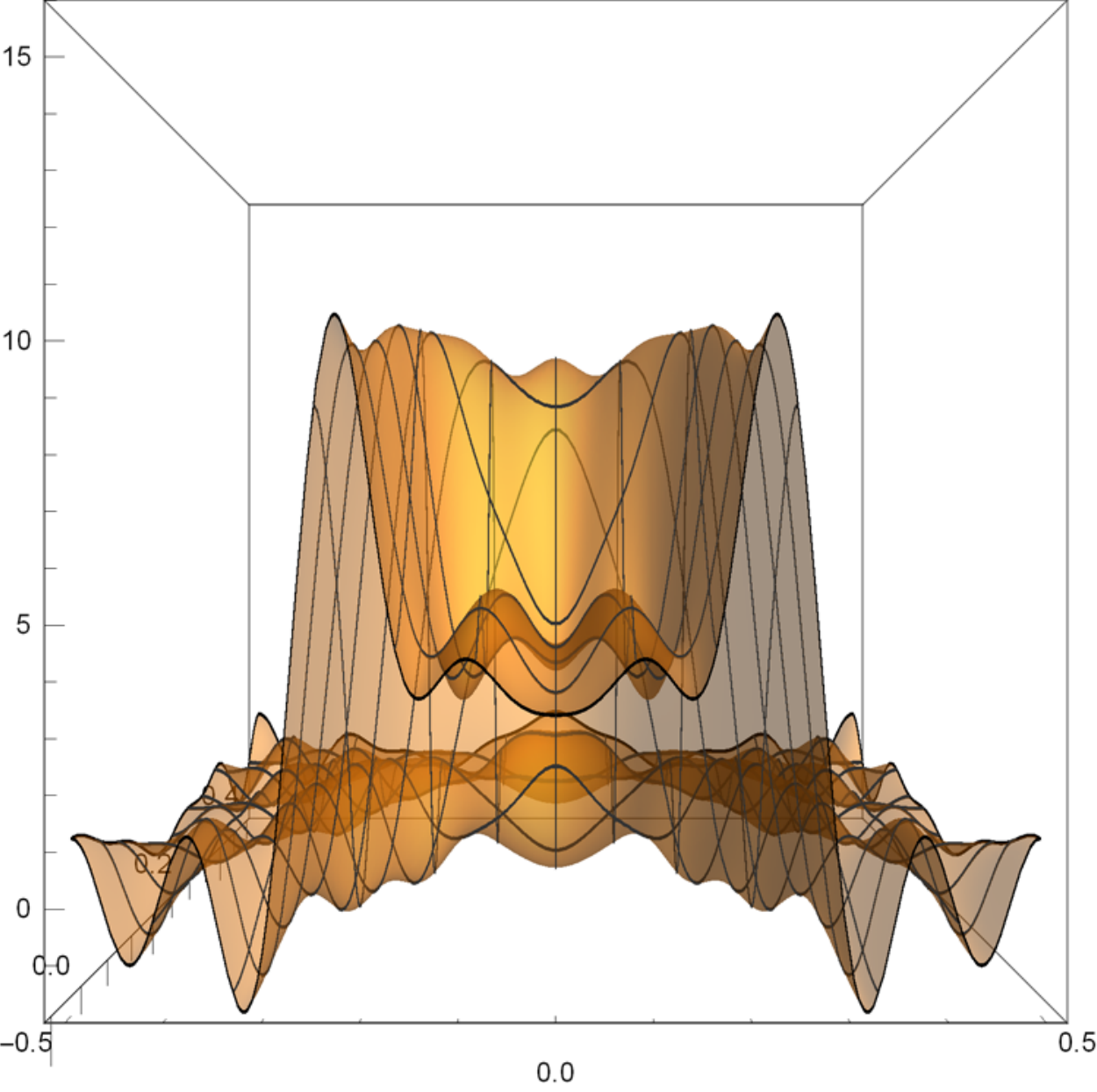} 
\includegraphics[width=.48\linewidth,height=36mm]{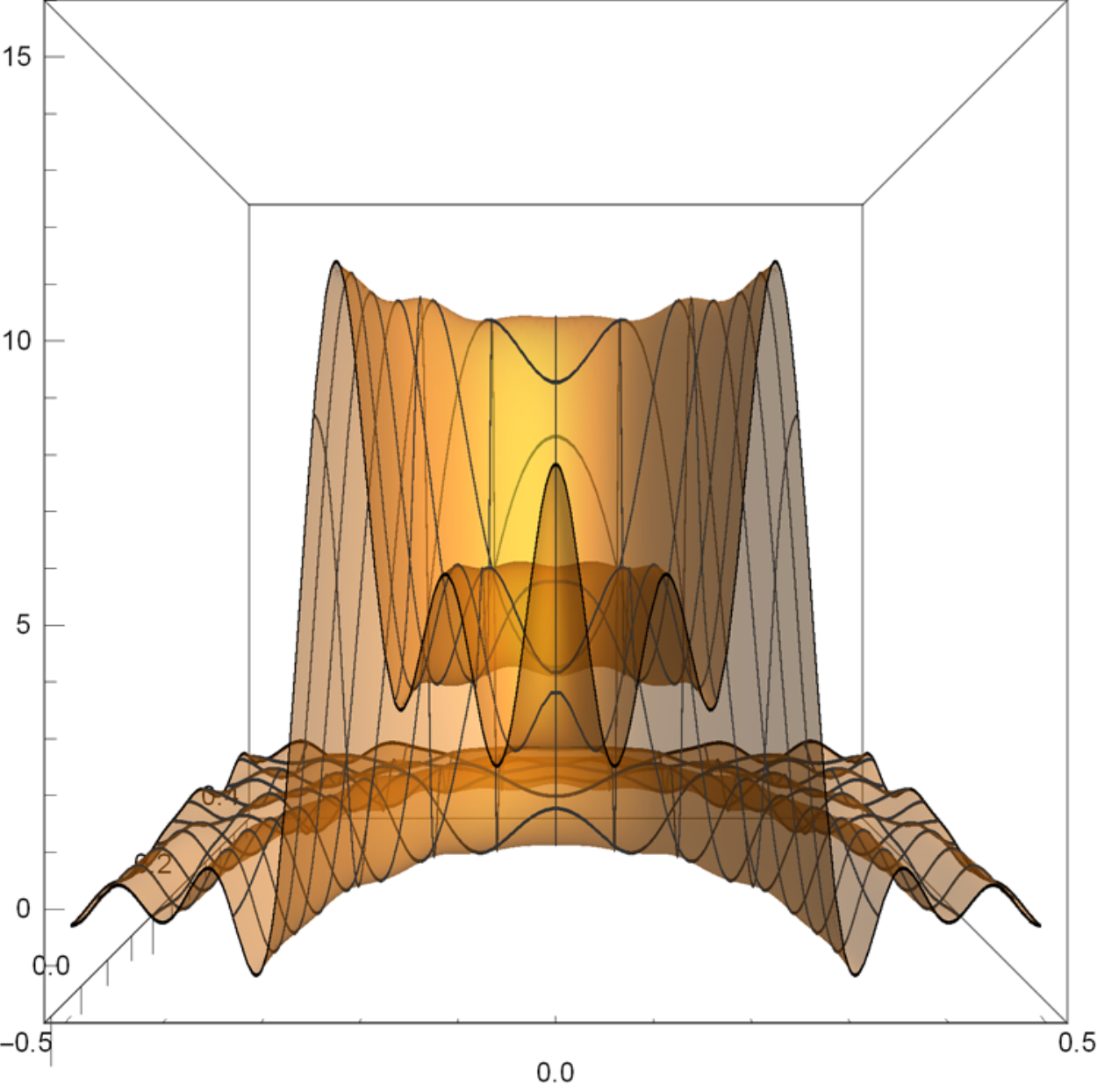} 
\includegraphics[width=.48\linewidth,height=36mm]{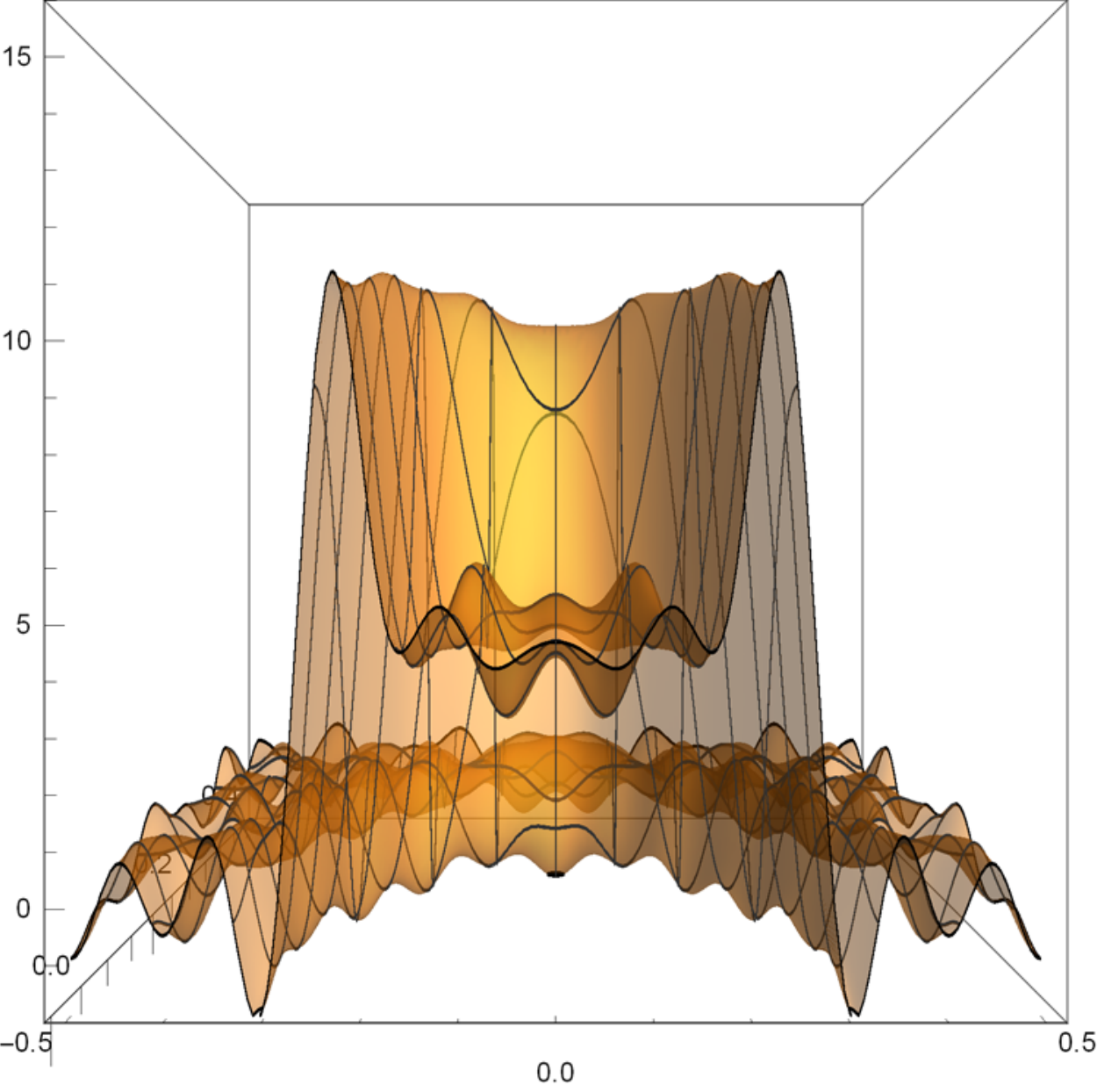} 
\includegraphics[width=.48\linewidth,height=36mm]{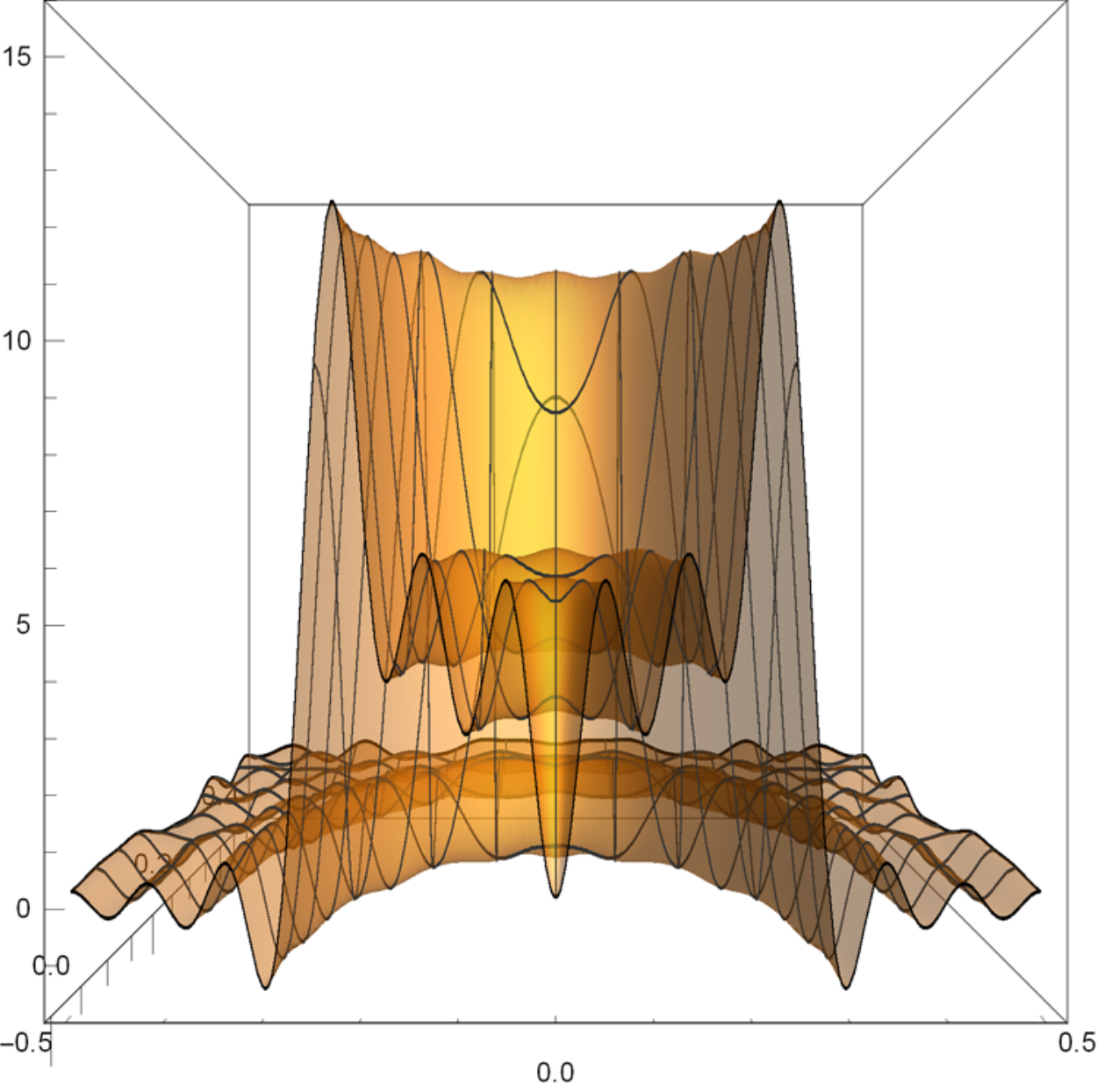} 
\caption{$S_{\lambda}(u_{\beta,a})(x_1,x_2)$ ($d=2$, $\beta=-1/2$, $a=1/4$) for $\lambda=9,10,11,12.$}
\label{fig:Su d2 3D}
\end{center}
\end{figure}

If $d=3$, $\beta=0$ and $a=1/4$, then $P_{\beta,a}^{[d]}=2/\pi$,
see Figure~\ref{fig:Su 3d}.
If $d=4$, $\beta=1/2$ and $a=1/4$, then $P_{\beta,a}^{[d]}=1/8$,
see Figure~\ref{fig:Su 4d}.

\begin{figure}
\begin{center}
\includegraphics[width=.48\linewidth,height=30mm]{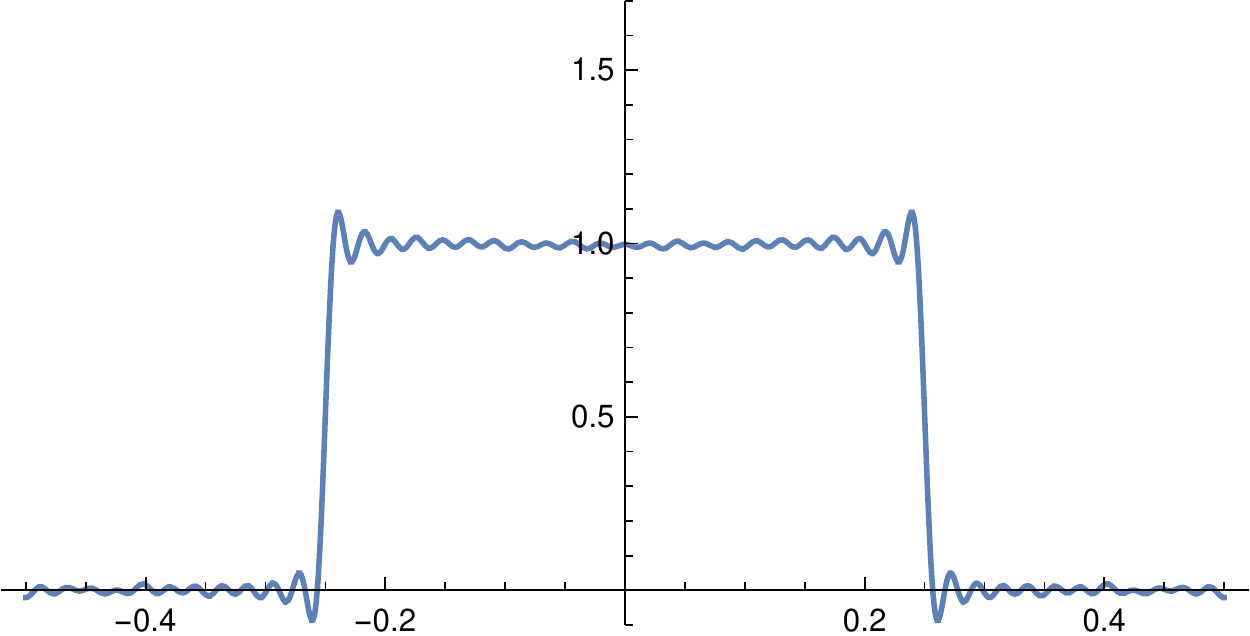} 
\includegraphics[width=.48\linewidth,height=30mm]{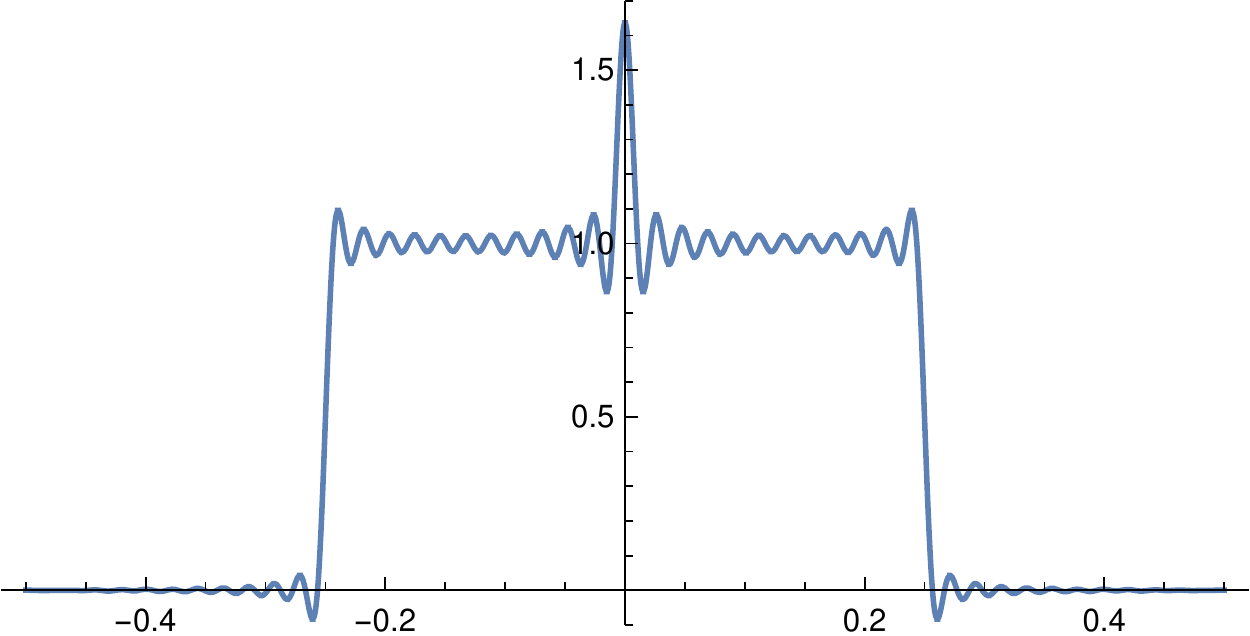} 
\includegraphics[width=.48\linewidth,height=30mm]{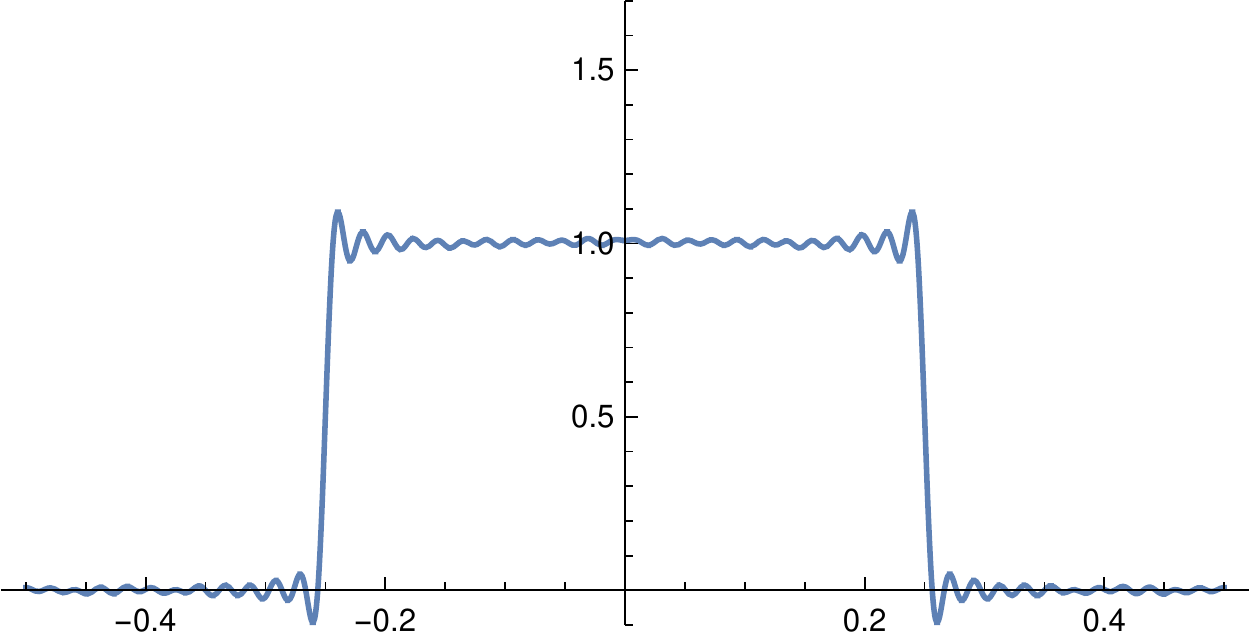} 
\includegraphics[width=.48\linewidth,height=30mm]{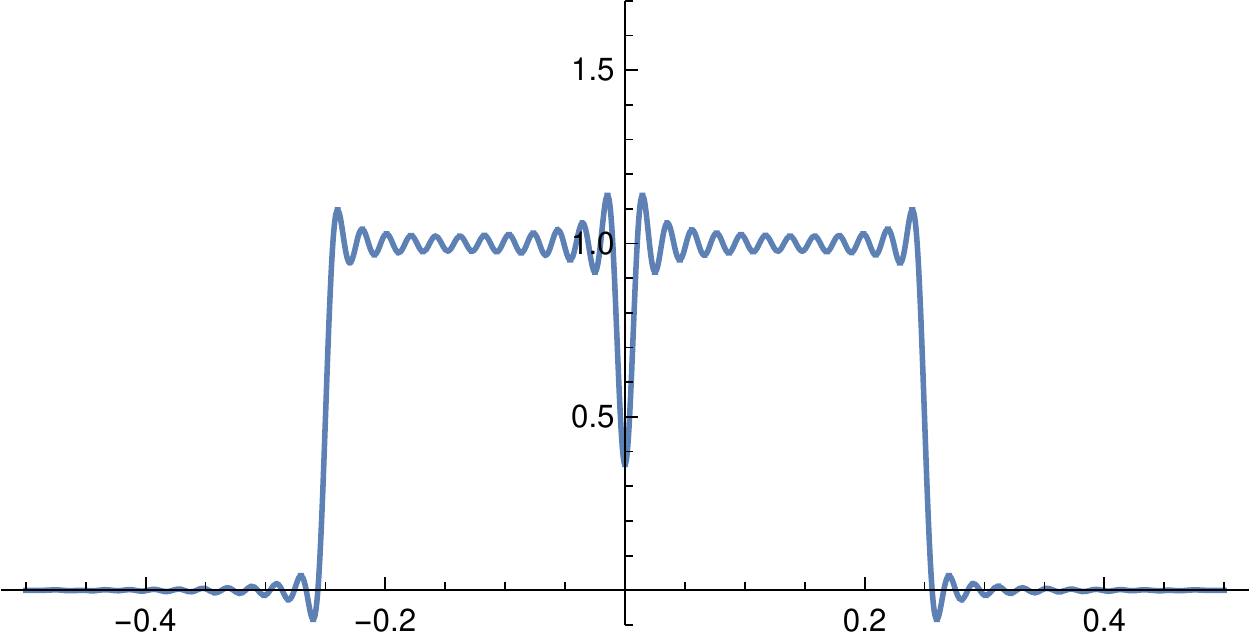} 
\caption{$S_{\lambda}(u_{\beta,a})(x,0,0)$ ($d=3$, $\beta=0$, $a=1/4$) for $\lambda=46,47,48,49$.}
\label{fig:Su 3d}
\end{center}
%
\begin{center}
\includegraphics[width=.48\linewidth]{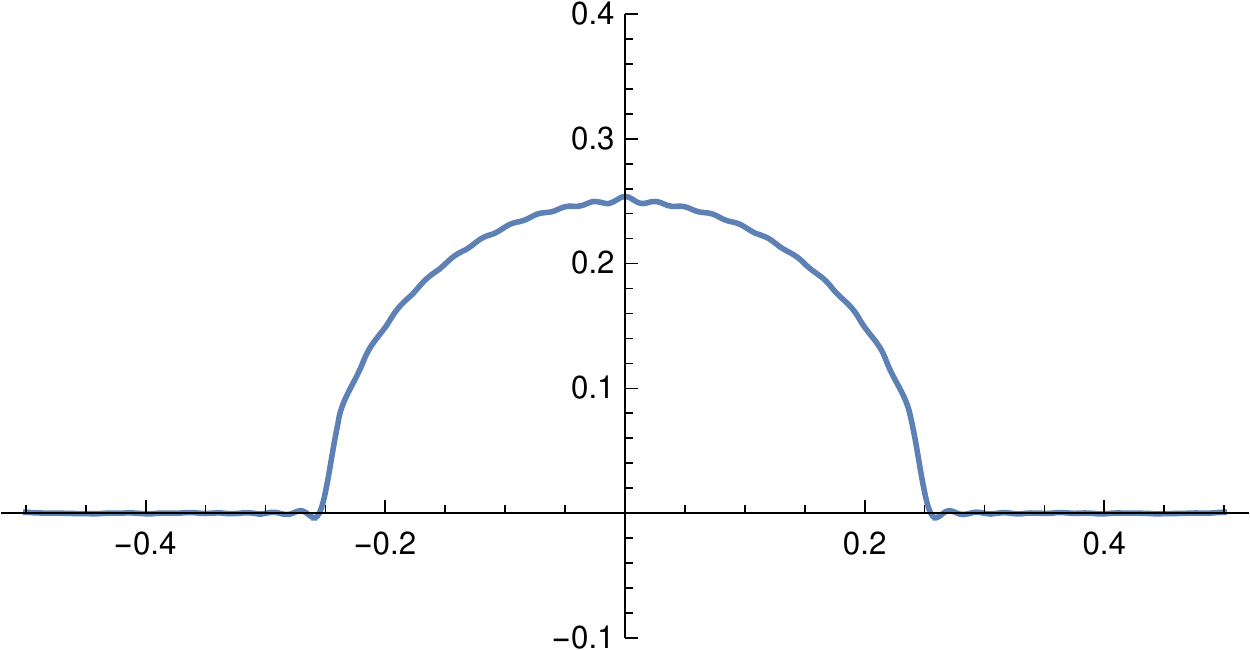} 
\includegraphics[width=.48\linewidth]{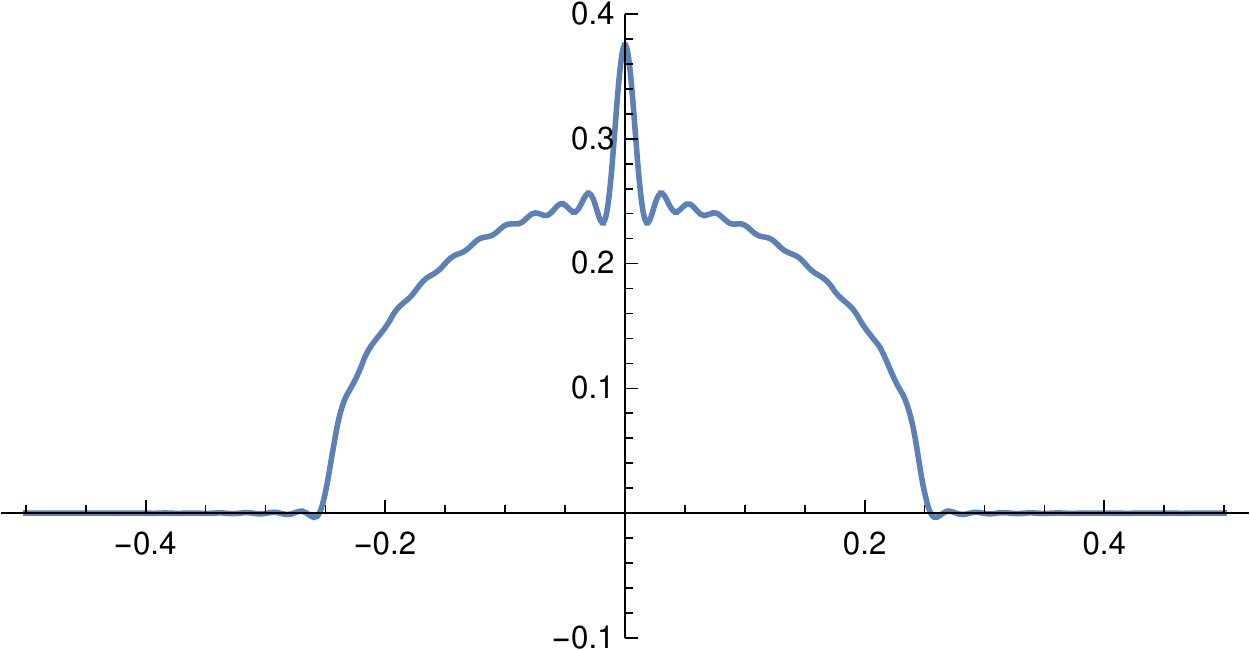} 
\includegraphics[width=.48\linewidth]{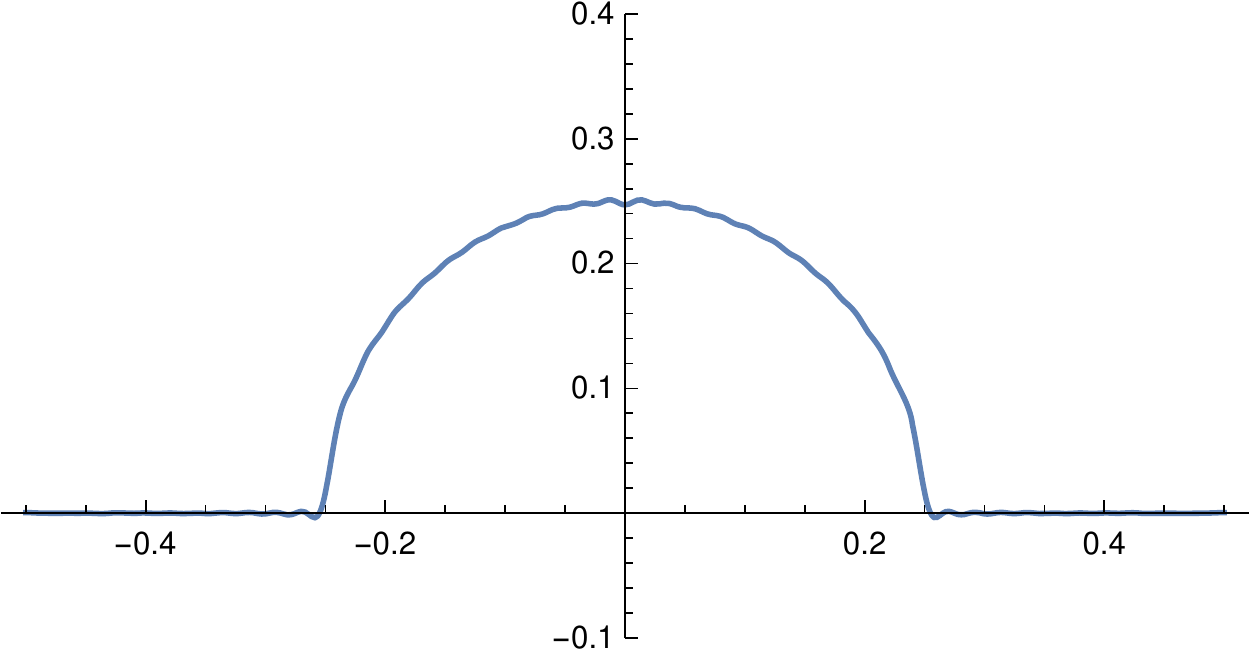} 
\includegraphics[width=.48\linewidth]{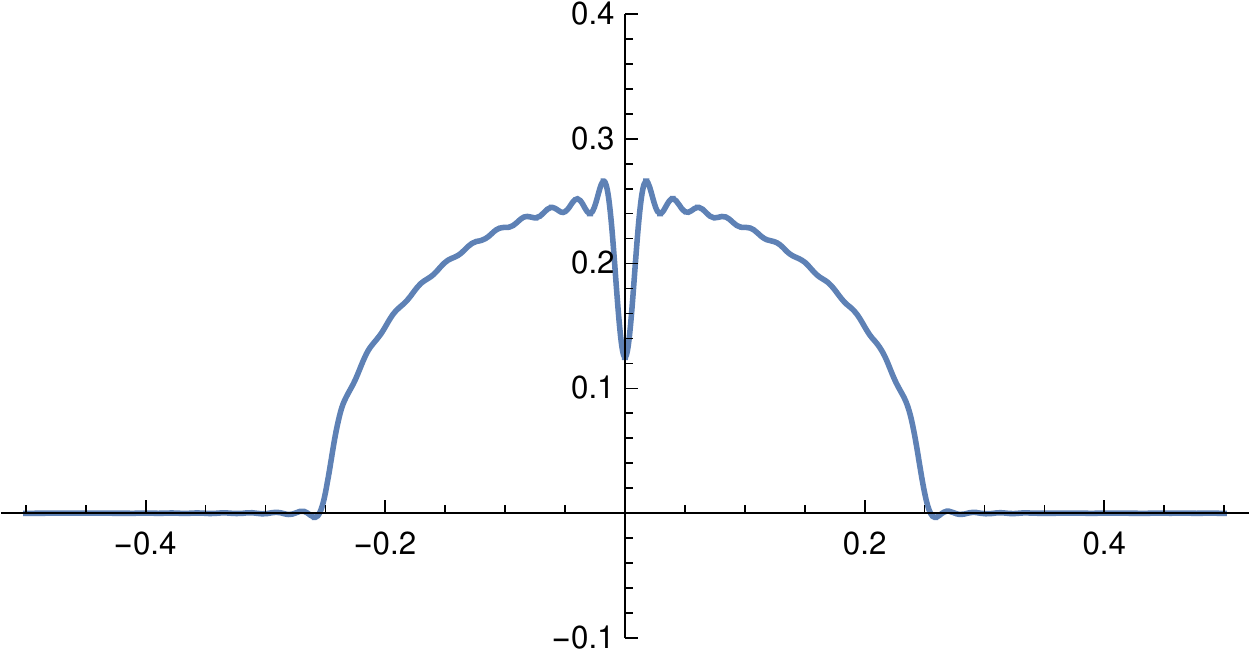} 
\caption{$S_{\lambda}(u_{\beta,a})(x,0,0,0)$ ($d=4$, $\beta=1/2$, $a=1/4$) for $\lambda=43,44,45,46$.}
\label{fig:Su 4d}
\end{center}
\end{figure}

Next, we state the pointwise behaviors of $S_{\lambda}(u_{\beta,a})$ on $\T^d\setminus\{0\}$
which includes the third phenomenon.
We define the number $c(d)$ for the dimension $d$ as the following:
\begin{equation}\label{c(d)}
 c(d)=d-\frac{2d}{d+1}-\frac{d+1}2=\frac{d-5}2+\frac2{d+1}=\frac{d(d-4)-1}{2(d+1)},
\end{equation}
that is,
\begin{equation*}
 c(1)=-1,\ c(2)=-5/6,\ c(3)=-1/2,\ c(4)=-1/10.
\end{equation*}

\begin{thm}\label{thm:|x|=a and PWC}
Let $\beta>-1$ and $a>0$.
\begin{enumerate}
\item
If $1\le d\le4$, then, 
\begin{enumerate}
\item
 for $ \beta>-1$,
 \begin{equation*}
 \lim_{\lambda\to\infty}
 \frac{S_{\lambda}(u_{\beta,a})(x)-u_{\beta,a}(x)}{\lambda^{-\beta}}
 =
 r_d(a:x)L_{\beta,a}
 \quad\text{for all $x\in G_a$}.
\end{equation*}
\item
for $\beta>c(d)$,
\begin{equation*}
 \lim_{\lambda\to\infty}S_{\lambda}(u_{\beta,a})(x)=u_{\beta,a}(x)
 \quad\text{uniformly on any compact set in } E_a.
\end{equation*}
\end{enumerate}

\item
If $d\ge5$, then
\begin{multline*}
  \lim_{\lambda\to\infty}
  \frac1{\lambda^{\frac{d-5}2}}
  \left|\frac{S_{\lambda}(u_{\beta,a})(x)-u_{\beta,a}(x)}{\lambda^{-\beta}}
        -r_d(a:x)L_{\beta,a}\right|
  =0
 \\
  \text{for all $x\in(E_a\cup G_a)\setminus\Q^d$},
\end{multline*}
and
\begin{multline*}
  0<
  \limsup_{\lambda\to\infty}
  \frac1{\lambda^{\frac{d-5}2}}
  \left|\frac{S_{\lambda}(u_{\beta,a})(x)-u_{\beta,a}(x)}{\lambda^{-\beta}}
        -r_d(a:x)L_{\beta,a}\right|
  <\infty 
 \\
  \text{for all $x\in(E_a\cup G_a)\cap\Q^d$}.
\end{multline*}
Namely, if $(d-5)/2-\beta\ge0$, then $S_{\lambda}(u_{\beta,a})$ reveals the third phenomenon.
\end{enumerate}
\end{thm}

Note that $r_d(a:x)=0$ for $x\in E_a$ in the above.
For example, we can see 
the Gibbs-Wilbraham phenomena, 
the Pinsky phenomena and
the third phenomena on the graphs of 
$S_{\lambda}(u_{\beta,a})$ with $d=5$, $\beta=-1/2$ and $a=1/4$ for $\lambda\in\N$,
see Figures~\ref{fig:Su 5d} and \ref{fig:Su 5d ex}.

\begin{figure}
\begin{center}
\includegraphics[width=.48\linewidth,height=40mm]{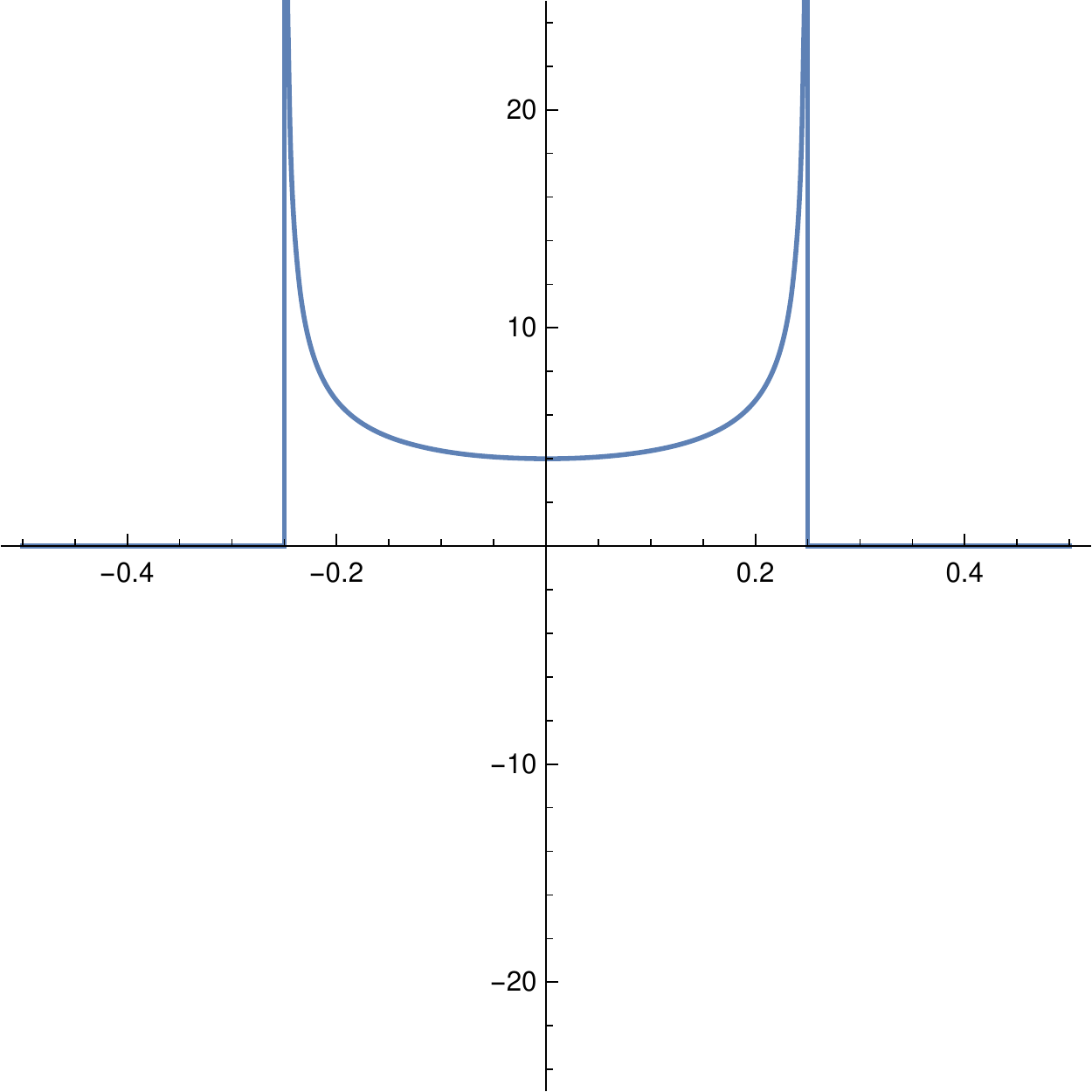} 
\caption{$u_{\beta,a}(x,0,0,0,0)$ ($d=5$, $\beta=-1/2$, $a=1/4$).}
\label{fig:u 5d}
\end{center}
\begin{center}
\includegraphics[width=.48\linewidth,height=40mm]{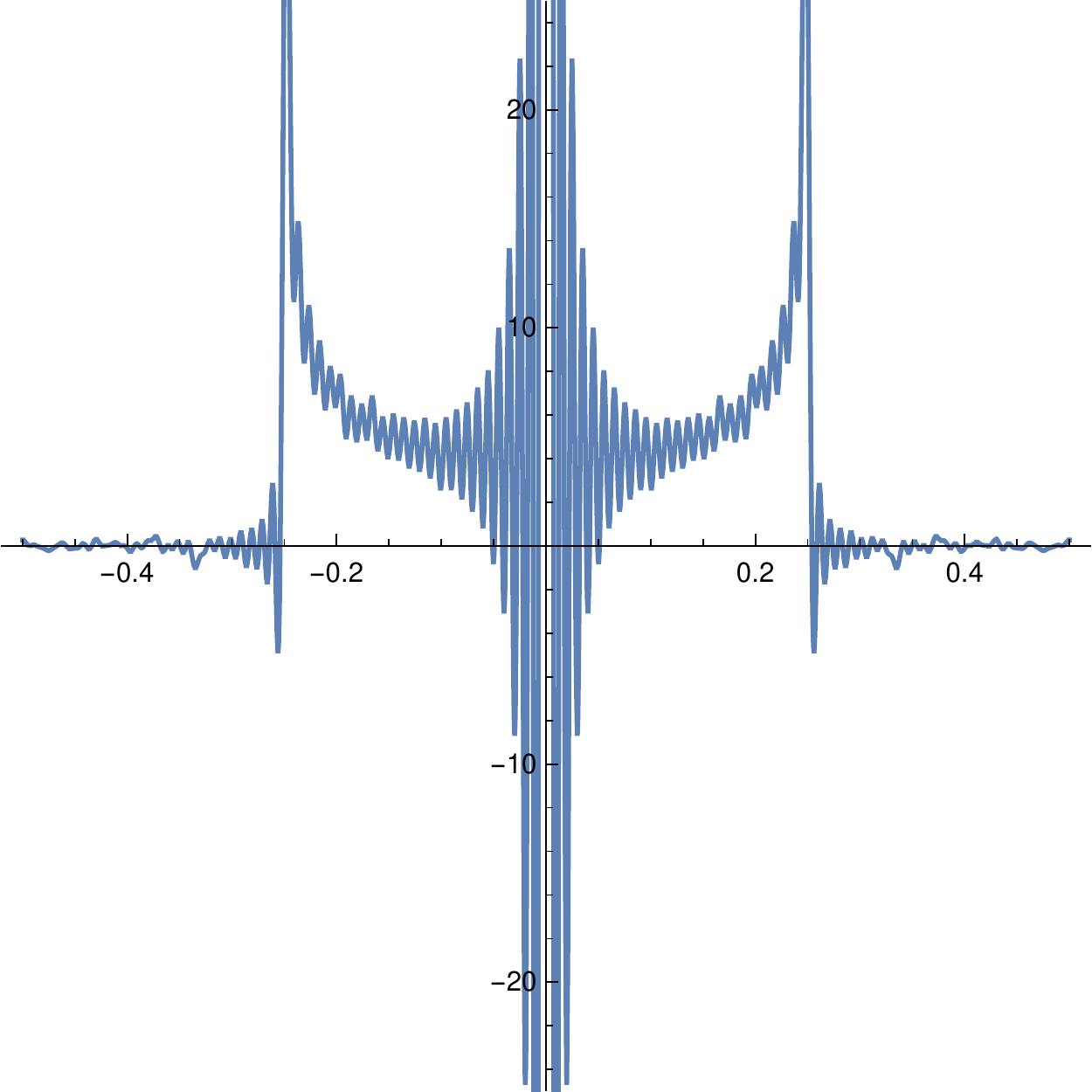} 
\includegraphics[width=.48\linewidth,height=40mm]{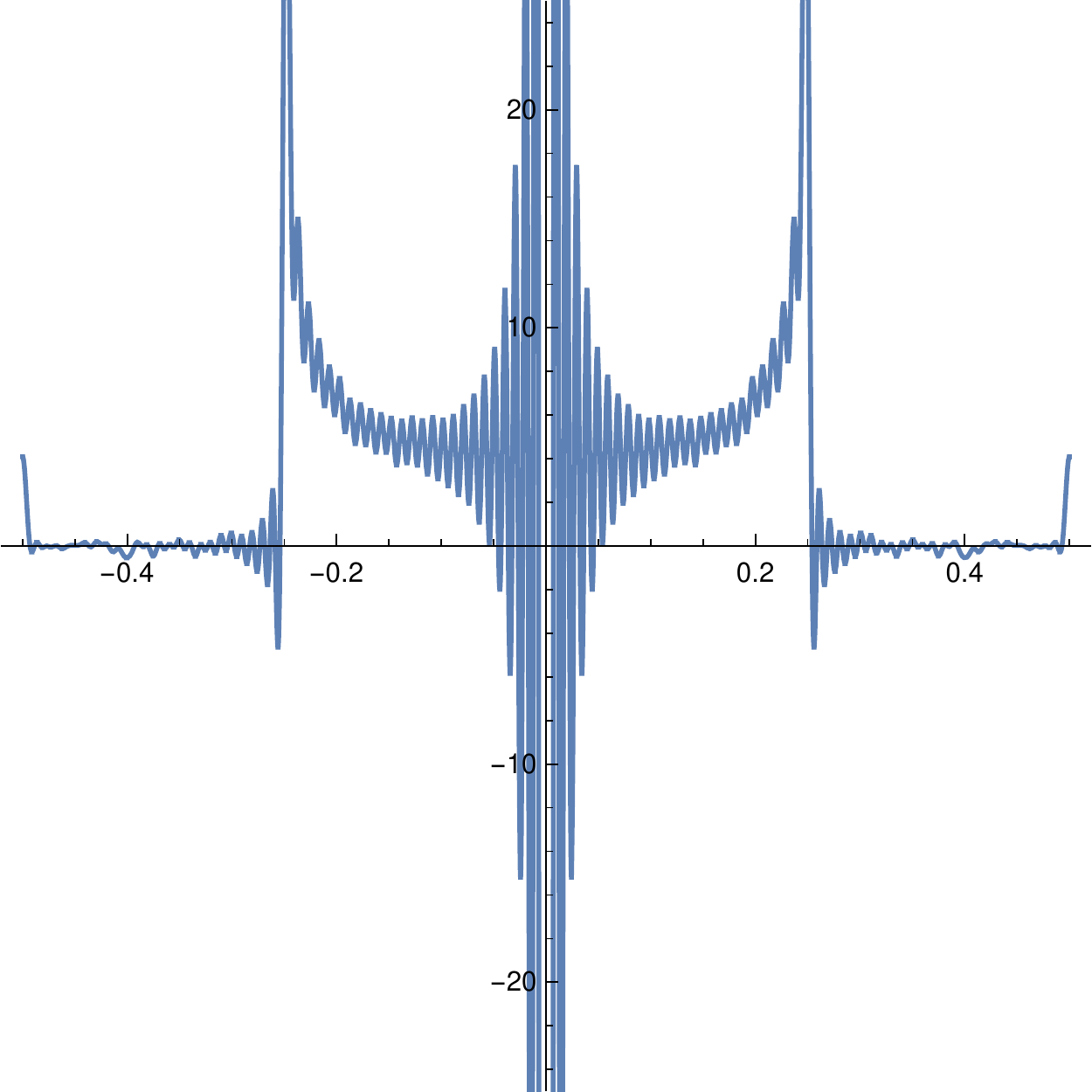} 
\includegraphics[width=.48\linewidth,height=40mm]{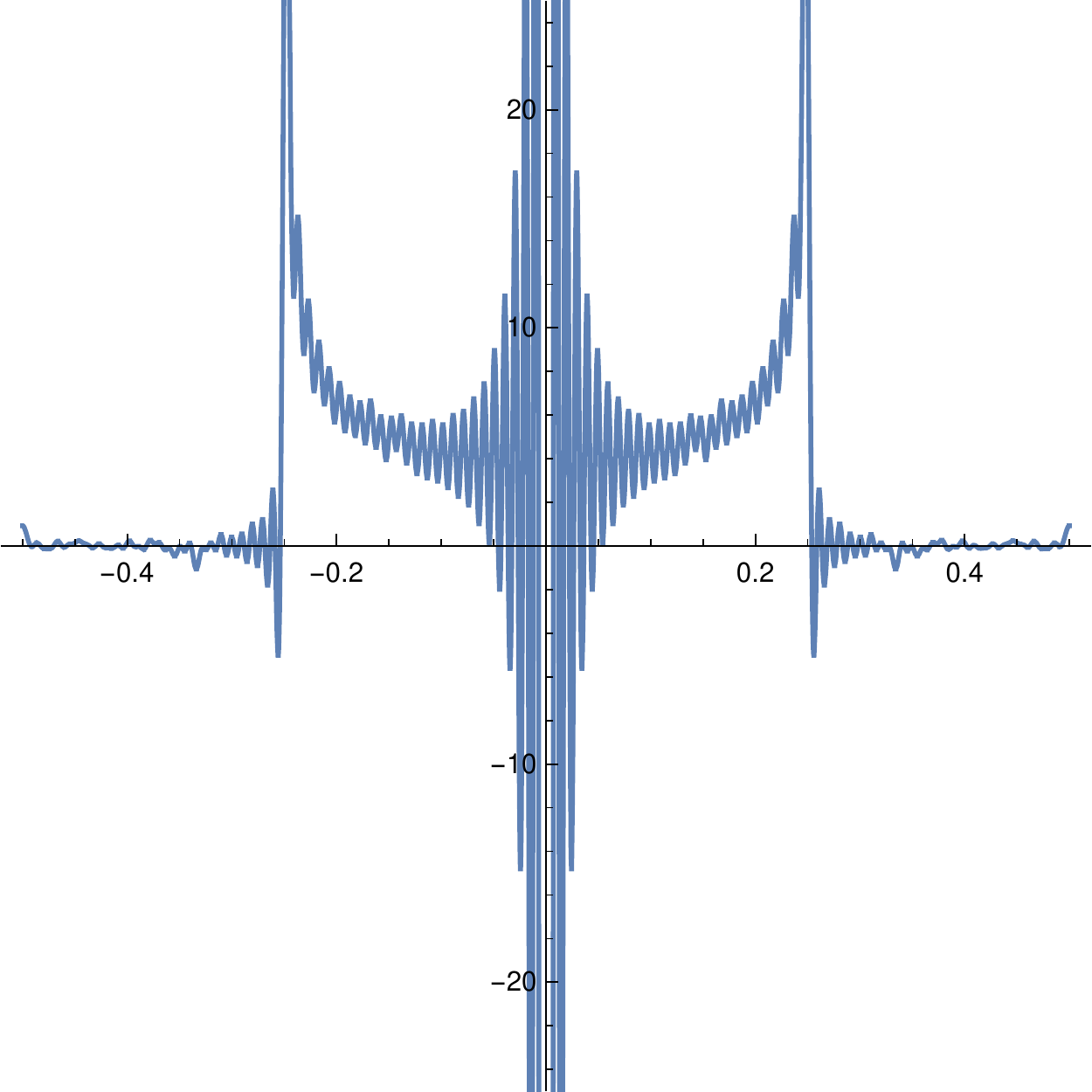} 
\includegraphics[width=.48\linewidth,height=40mm]{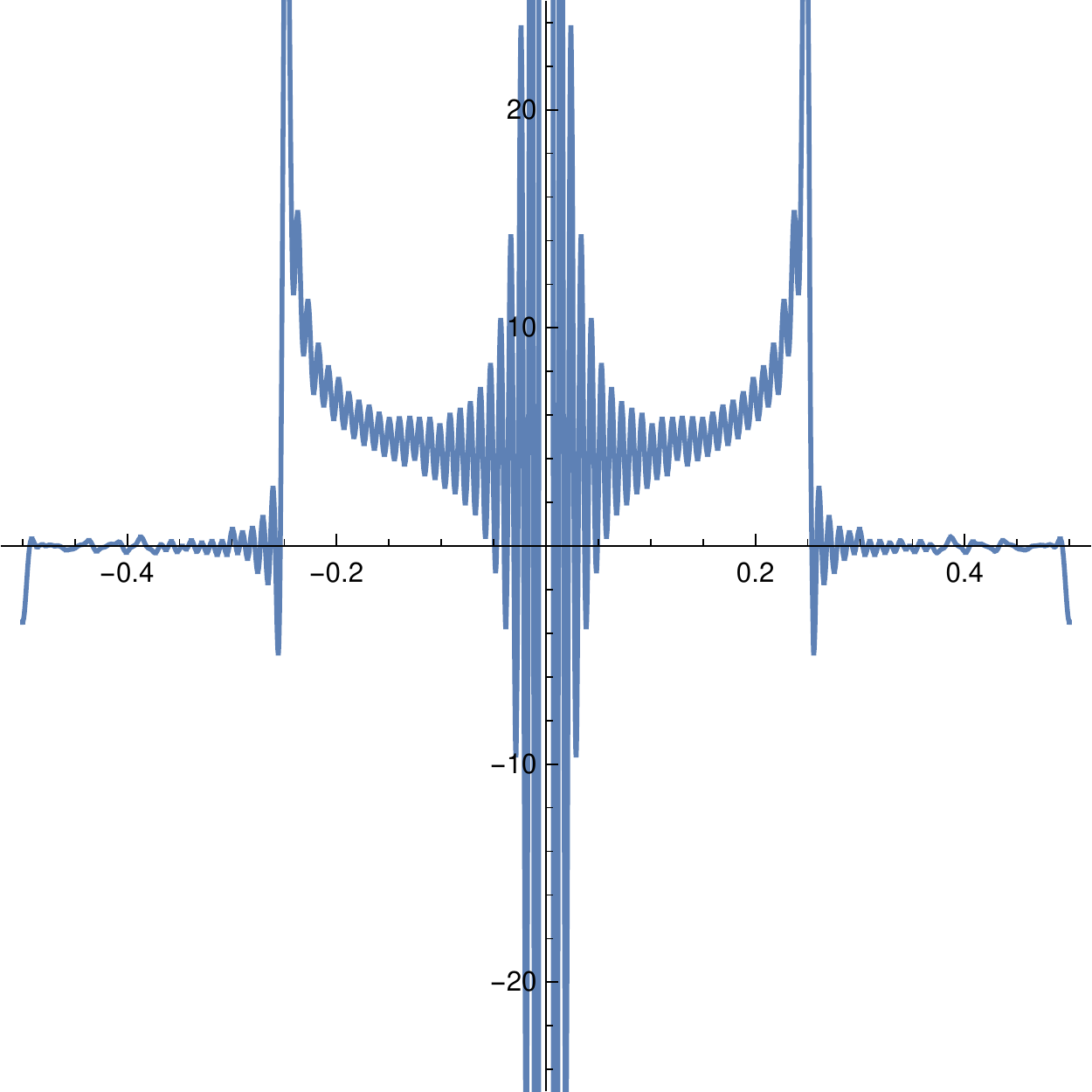} 
\caption{
$S_{\lambda}(u_{\beta,a})(x,0,0,0,0)$ ($d=5$, $\beta=-1/2$, $a=1/4$) for $\lambda=100,\dots,103$.}
\label{fig:Su 5d}
\end{center}
\begin{center}
\includegraphics[width=.24\linewidth,height=30mm]{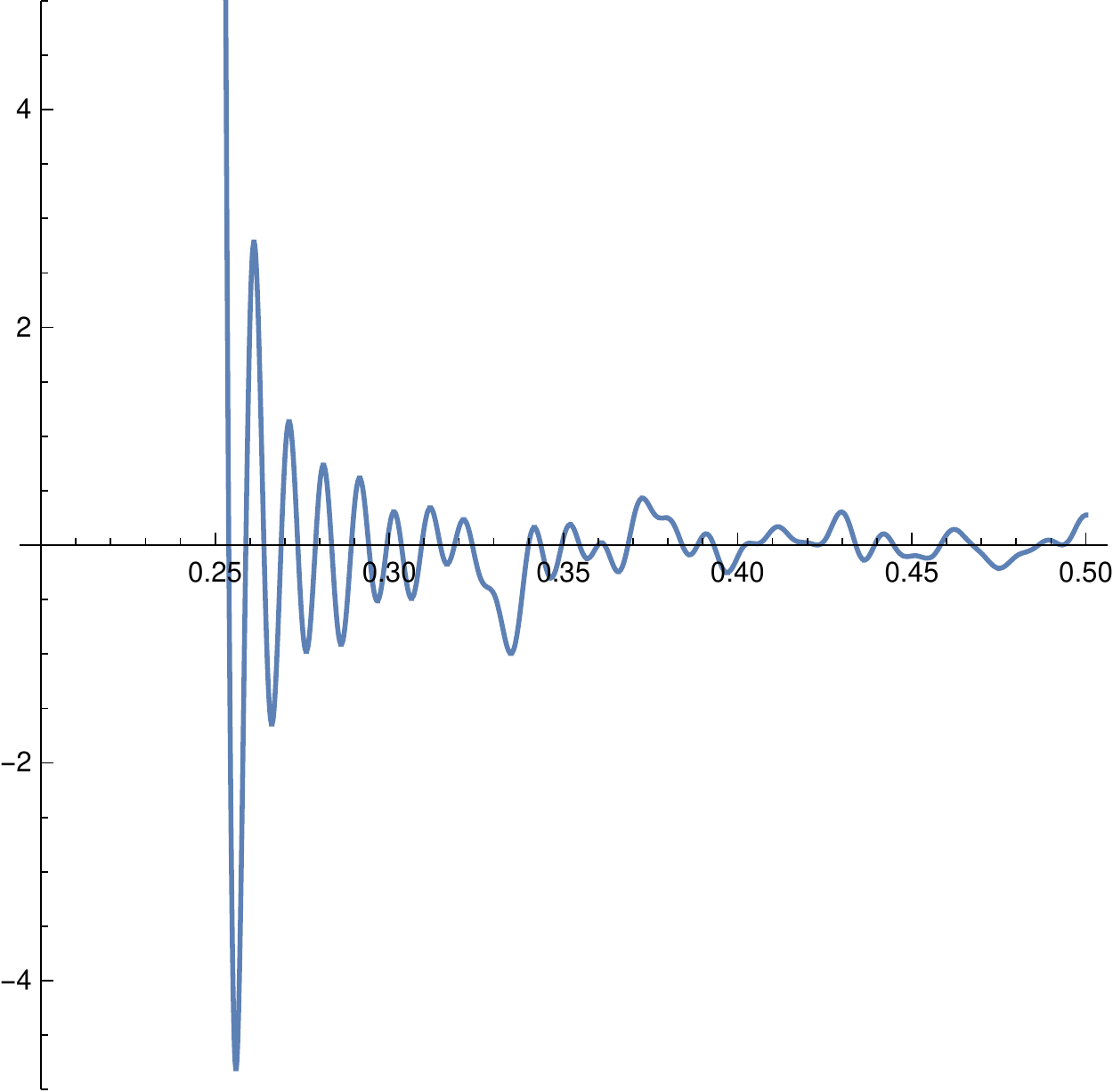} 
\includegraphics[width=.24\linewidth,height=30mm]{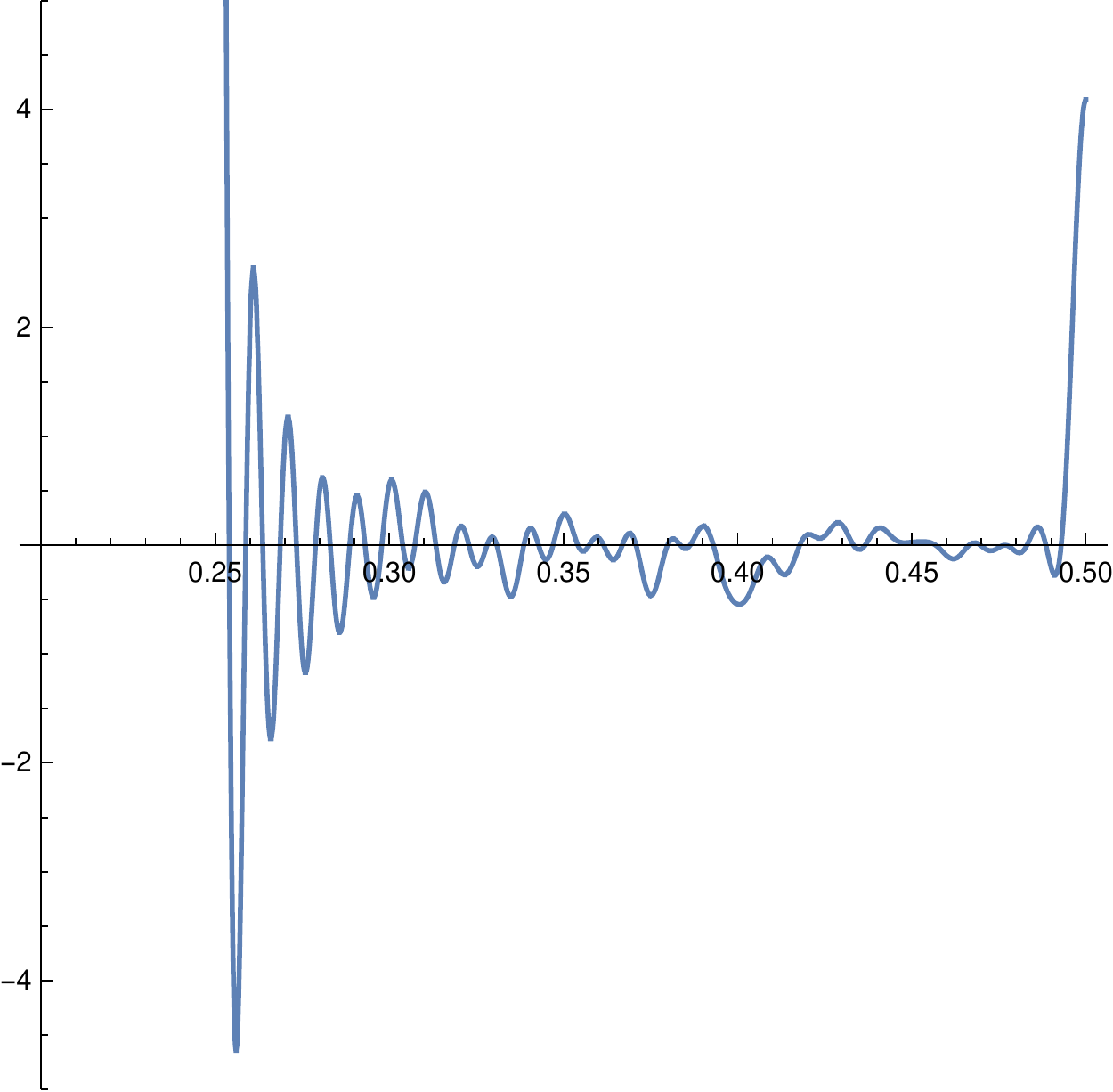} 
\includegraphics[width=.24\linewidth,height=30mm]{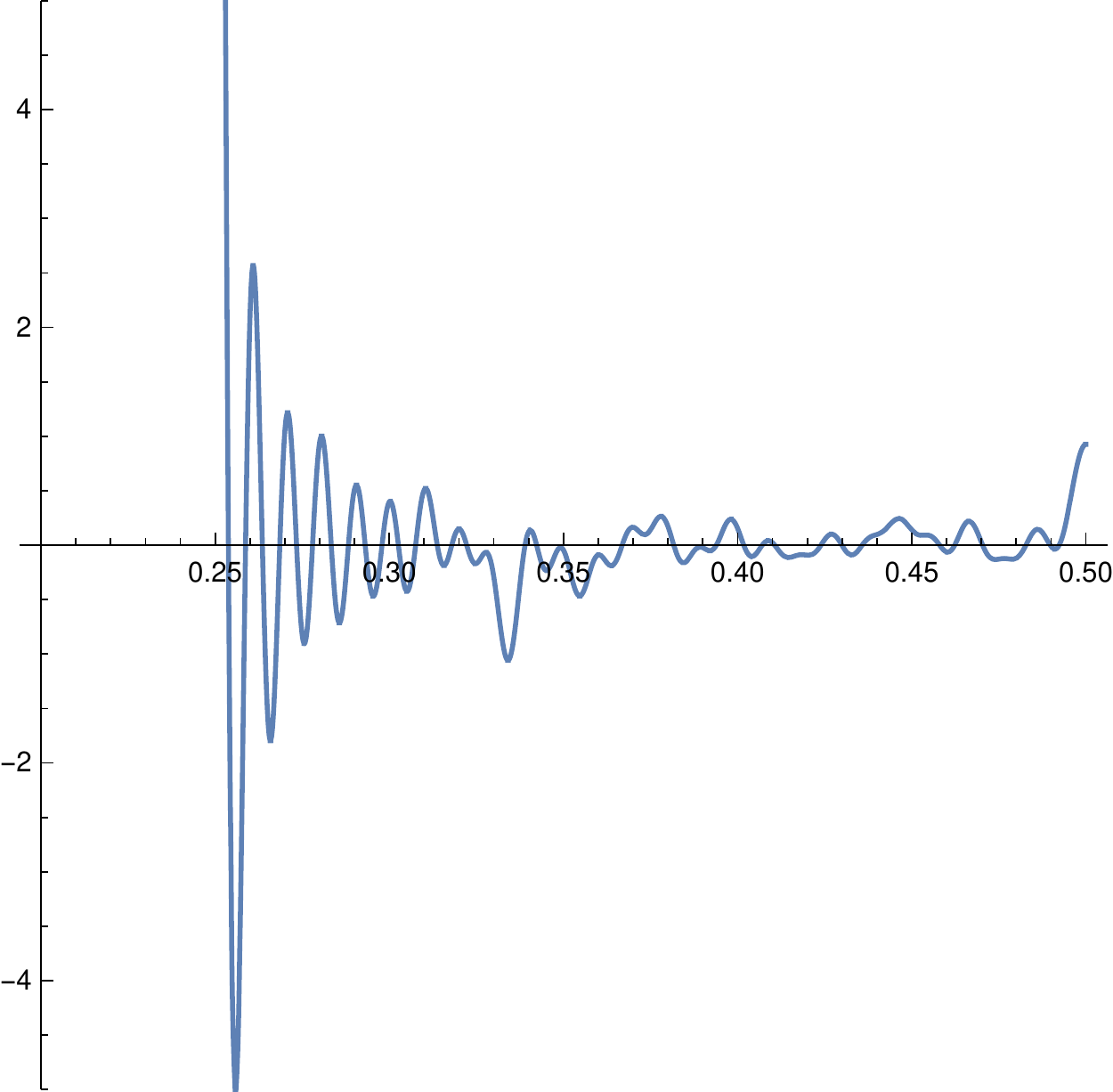} 
\includegraphics[width=.24\linewidth,height=30mm]{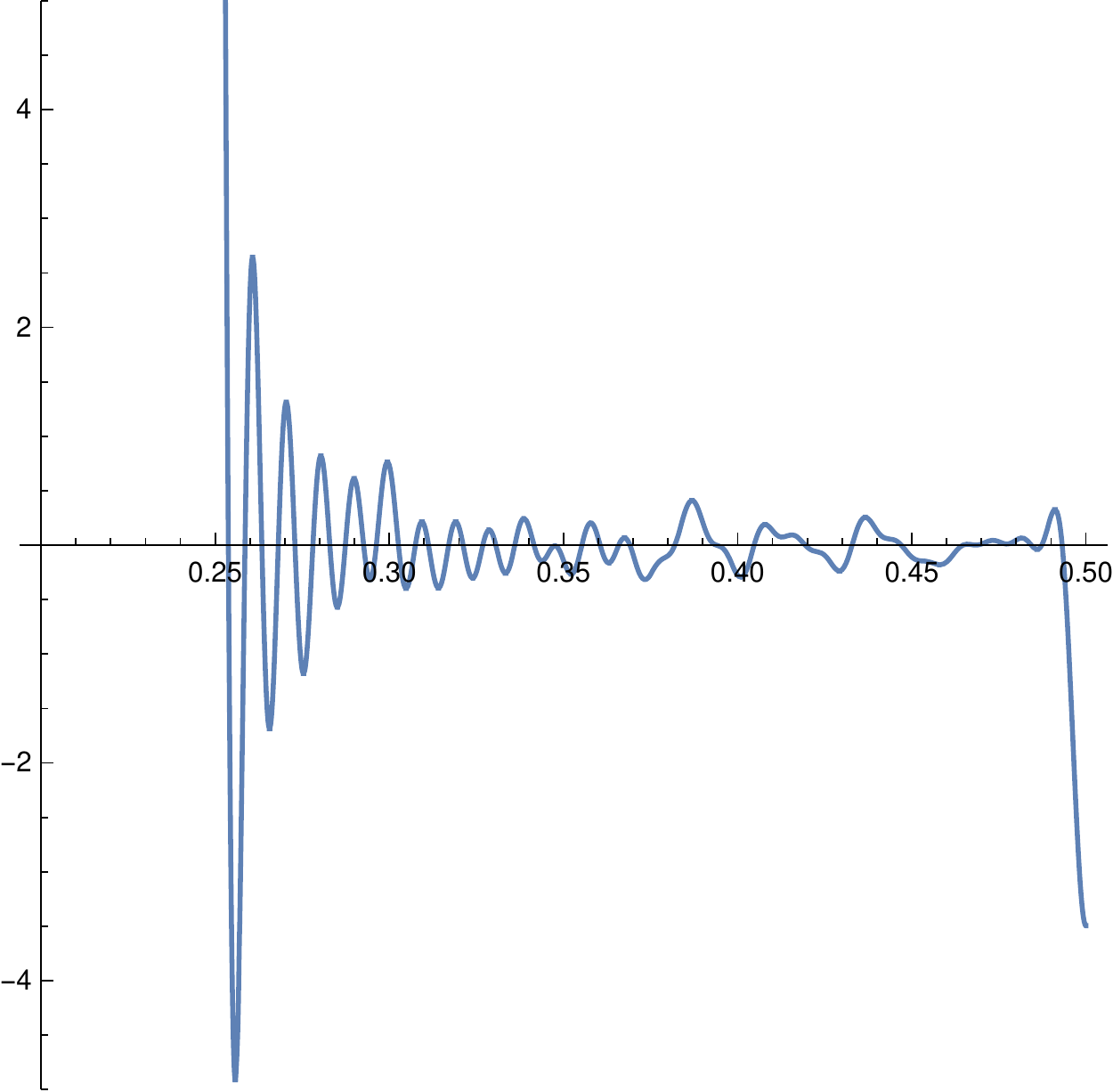} 
\caption{
$S_{\lambda}(u_{\beta,a})(x,0,0,0,0)$ ($d=5$, $\beta=-1/2$, $a=1/4$) for $\lambda=100,\dots,103$: Expansion of the part $0.2\le x\le 0.5$ in Figure~\ref{fig:Su 5d}.}
\label{fig:Su 5d ex}
\end{center}
\end{figure}

\begin{rem}\label{rem:PWC}
In the case $2\le d\le4$ and $-1<\beta\le c(d)$,
the pointwise convergence of $S_{\lambda}(u_{\beta,a})$ on $E_a$ is
an open problem, which is closely related to the lattice point problems.
Especially, if $d=2$ and $x=0$, then
the pointwise convergence of $S_{\lambda}(u_{\beta,a})(0)$ is 
equivalent to Hardy's conjecture on Gauss's circle problem
(see Section~\ref{sec:relation}).
If $d=4$ and $-1<\beta<-1/2$, 
then we get a partial result on the divergence of $S_{\lambda}(u_{\beta,a})$ on $\T^d\cap\Q^d$
(\cite{Oo-Fu-Ku-Na-preparation}).
\end{rem}

Next, we state the Gibbs-Wilbraham phenomenon near $G_a$ for $1\le d\le4$.

\begin{thm}\label{thm:GibbsPh}
Let $1\le d\le 4$, $c(d)<\beta\le 0$ and $0<a<1/2$. 
Then,
on any neighborhood of the set $G_a$,
$S_{\lambda}(u_{\beta,a})$ 
reveals a phenomenon like the Gibbs-Wilbraham phenomenon.
More precisely, the following holds:
For each $x_0\in G_a$, 
let $\{x_\lambda^\pm\}\subset\T^d$ be the sequences
which satisfy $\displaystyle\lim_{\lambda\to\infty}x_\lambda^\pm=x_0$ 
and $|x_\lambda^\pm|=a\mp(2\pm\beta)/(4\lambda)$. 
Then 
\begin{equation*}
 \lim_{\lambda\to\infty}
 \frac{S_{\lambda}(u_{\beta,a})(x_{\lambda}^{\pm})
       -u_{\beta,a}(x_{\lambda}^{\pm})}
      {\lambda^{-\beta}}
 =
 G_{\beta,a}^{\pm},
\end{equation*}
where
\begin{equation}\label{G}
 G_{\beta,a}^{\pm}
 =
 \mp\frac{\Gamma(\beta+1)a^{\beta}(2\pm\beta)^{\beta}}{\pi 2^\beta}
  \int_{\pi}^{\infty} \frac{\sin s}{(s\pm\frac\beta 2\pi)^{\beta+1}}\,ds.
\end{equation}
\end{thm}

If $d=3$, $\beta=0$ and $a=1/4$, 
then we can see both
the Gibbs-Wilbraham phenomena and the Pinsky phenomena 
on the graphs in Figure~\ref{fig:Su 3d}.
If $d=4$, $\beta=1/2$ and $a=1/4$, 
then we can see 
the Pinsky phenomena but not see the Gibbs-Wilbraham phenomena
on the graphs in Figure~\ref{fig:Su 4d}.

\begin{rem}\label{rem:GibbsPh}
(i)
The constant $G_{\beta,a}^{+}$ is positive and 
$G_{\beta,a}^{-}$ is negative.
Especially
\begin{equation*}
G_{0,a}^{\pm}
 =\mp
 {\frac1\pi}\int_{\pi}^{\infty}\frac{\sin s}s\,ds
 =\mp
 \frac12\pm{\frac1\pi}\int_0^{\pi}\frac{\sin s}s\,ds
 =\pm 0.08949\cdots.
\end{equation*}
(ii)
In the case $2\le d\le4$ and $-1<\beta\le c(d)$,
Theorem~\ref{thm:GibbsPh} is also an open problem 
by the same reason as Remark~\ref{rem:PWC}.
\end{rem}

Finally, we state the almost everywhere convergence of $S_\lambda(u_{\beta,a})$
for $d\ge4$.

\begin{thm}\label{thm:AEC}
Let $d\ge 4$, $\beta>-1/2$ and $a>0$. Then 
\begin{equation*}
 \lim_{\lambda\rightarrow\infty}S_\lambda(u_{\beta,a})(x)
 =
 u_{\beta,a}(x),
 \quad \text{a.e. $x\in\T^d$}.
\end{equation*}
\end{thm}

\begin{rem}\label{rem:Kuratsubo1996}
(i) From Theorems~\ref{thm:|x|=a and PWC} and \ref{thm:AEC}
we get Theorem in \cite{Kuratsubo1996} as a corollary,
which is the case $\beta=0$.

\noindent (ii)
In the case $d\ge4$ and $-1<\beta\le-1/2$,
the almost everywhere convergence of $S_{\lambda}(u_{\beta,a})$ is
an open problem. 
\end{rem}

We first prove in Section~\ref{sec:GHI} a fundamental identity 
for the periodization $f$ 
of any integrable radial function $F$ on $\R^d$ with compact support.
This identity gives the relation among
$S_{\lambda}(f)(x)$, $\sigma_{\lambda}(F)(x)$ 
and the term related to lattice point problems.
In Section~\ref{sec:B function} we collect the results with respect to Bessel functions.
In Section~\ref{sec:F inversion} 
we investigate the behavior of $\sigma_{\lambda}(U_{\beta,a})(x)$.
We show the uniform convergence of $\sigma_{\lambda}(U_{\beta,a})(x)$
on any compact subset of $E_a$, 
the Pinsky phenomenon at the origin,
and, the Gibbs-Wilbraham phenomenon near the spherical surface.
In Section~\ref{sec:LPP} we prove lemmas related to the lattice point problems.
Then, in Section~\ref{sec:proof}, 
using the results in Sections~\ref{sec:GHI}--\ref{sec:LPP},
we prove the main results,
in which
it turns out that
the third phenomenon is closely related to the lattice point problems.
Finally, in Section~\ref{sec:relation} we give
the relation between the convergence of the spherical partial sum 
and lattice point problems.


\section{Fundamental identity}\label{sec:GHI}

Let $J_{\nu}$ be the Bessel function of order $\nu$.
If $\nu>-1$, then
\begin{equation*} 
  \frac{J_{\nu}(s)}{s^{\nu}} \to \frac{1}{2^{\nu}\Gamma(\nu+1)}
  \quad\text{as}\quad
  s\to0.
\end{equation*}
For this reason,
in this paper 
we always regard 
\begin{equation}\label{Bessel-at-zero}
  \frac{J_{\nu}(s)}{s^{\nu}} = \frac{1}{2^{\nu}\Gamma(\nu+1)}
  \quad\text{for}\quad
  s=0.
\end{equation}
For $\nu>-1$, let
\begin{equation}\label{Lam}
 \Lambda_{\nu}(t:s)
 =
 \dfrac{J_\nu(2\pi t\sqrt{s})}{s^{\frac\nu 2}}, 
 \quad t\ge0,\ s>0.
\end{equation}
By \eqref{Bessel-at-zero} we regard
\begin{equation*}
 \Lambda_{\nu}(t:0)
 =
 \frac{(\pi t)^\nu}{\Gamma(\nu+1)},
 \quad t\ge0.
\end{equation*}

For any $\alpha>-1$ 
we define $D_{\alpha}(s:x)$, ${\cD}_{\alpha}(s:x)$ and $\Delta_{\alpha}(s:x)$ 
as the following:
\begin{align*}
 D_{\alpha}(s:x)
 &=
 \begin{cases}\displaystyle
 {\frac1{\Gamma(\alpha+1)}}\sum_{|m|^2<s}(s-|m|^2)^\alpha e^{2\pi imx}, 
  & s>0,\\
 0, & s=0,
 \end{cases}
 \quad x\in\R^d,
\\
 {\cD}_{\alpha}(s:x)
 &=
 \begin{cases}\displaystyle
 {\frac1{\Gamma(\alpha+1)}}\int_{|\xi|^2<s}(s-|\xi|^2)^\alpha e^{2\pi ix\xi}\,d\xi,
  & s>0,\\
 0, & s=0,
 \end{cases}
 \quad x\in\R^d,
\end{align*}
and
\begin{equation}\label{Delta}
 \Delta_\alpha(s:x)
 =
 D_\alpha(s:x)-{\cD}_\alpha(s:x),
 \quad s\ge0,\ x\in\R^d.
\end{equation}

We consider a function $\phi$ on $[0,\infty)$ 
and the periodization $f_{\phi}$ of the function $F_{\phi}(x)=\phi(|x|)$, $x\in\R^d$,
that is, 
\begin{equation}
 f_{\phi}(x)
 =
 \sum_{m\in{\Z}^d}F_\phi(x+m)=\sum_{m\in{\Z}^d}\phi(|x+m|), 
 \quad x\in\T^d. 
\end{equation}
If the radial function $F_{\phi}$ is integrable on $\R^d$,
then the Fourier transform ${\hat F}_{\phi}$ is also radial and 
has the form
\begin{equation}\label{FT-radial}
 {\hat F}_\phi(\xi)
 =
 \int_{\R^d}\phi(|x|)e^{-2\pi ix\xi}\,dx
\\
 =
 2\pi\int_0^{\infty} 
  \phi(t)\frac{J_{\frac d2-1}(2\pi t|\xi|)}{(t|\xi|)^{\frac d2-1}}t^{d-1}\,dt,
\end{equation}
that is, 
\begin{equation}\label{FT-radial Lam}
 {\hat F}_\phi(\xi)
 =
 2\pi\int_0^{\infty} 
  \phi(t)\Lambda_{{\frac d2}-1}(t:|\xi|^2)\,t^{\frac d2}\,dt.
\end{equation}
For the equation \eqref{FT-radial},
see \cite[Theorem 3.3, page 155]{Stein-Weiss1971} if $d\ge2$.
If $d=1$, then \eqref{FT-radial} is also valid by the equation 
$J_{-\frac12}(s)=\sqrt{({2}/{\pi s})}\cos s$.
Let  
\begin{equation}\label{A_phi}
 A_{\phi}(s)
 =
 2\pi\int_0^{\infty} \phi(t)\Lambda_{{\frac d2}-1}(t:s)\,t^{\frac d2}\,dt,
 \quad s\ge0.
\end{equation}
Then
\begin{equation}\label{FT-Fphi}
 {\hat F}_{\phi}(\xi)=A_{\phi}(|\xi|^2).
\end{equation}
Moreover, if the support of $\phi$ is compact,
then $A_{\phi}$ is infinitely differentiable and
\begin{equation*}
 A_{\phi}^{(j)}(s)
 =
 2\pi(-\pi)^j\int_0^{\infty} \phi(t)\Lambda_{{\frac d2}-1+j}(t:s)\,t^{\frac d2+j}\,dt,
 \quad s\ge0,\ j=1,2,\cdots,
\end{equation*}
since
\begin{equation}\label{d Lam}
 \frac{\partial}{\partial s}\Lambda_{\mu}(t:s)=(-\pi t)\Lambda_{\mu+1}(t:s),
\end{equation}
by the Bessel recurrence formula (see \eqref{Bessel-recurrence-formula} below).

Now, we define $\dsharp$ as the smallest integer that is greater than $(d-1)/2$, that is,
\begin{equation}\label{kd}
  \dsharp 
  =\left[\frac{d-1}2\right]+1,
\end{equation}
where 
$[t]$ is the integer part of $t\ge0$.
In this section we prove the following:
\begin{thm}\label{thm:GHI}
Let $\phi$ be a function on $[0,\infty)$ with compact support.
Suppose that $F_{\phi}$ is integrable on $\R^d$.
Let $f_{\phi}$ be the periodization of the function $F_{\phi}$.
Then
\begin{multline}\label{fphi}
 S_\lambda(f_\phi)(x)
 =
 \sigma_{\lambda}(F_\phi)(x) 
 +\sum_{j=0}^{\dsharp}(-1)^j\Delta_j(\lambda^2:x)A_{\phi}^{(j)}(\lambda^2) 
\\ 
 +(-1)^{\dsharp+1}\sum_{m\in\Z^d\setminus\{0\}}
 \int_0^{\lambda^2}{\cD}_{\dsharp}(s:x-m)A_{\phi}^{(\dsharp+1)}(s)\,ds,
\end{multline}
for all $x\in\T^d$. 
\end{thm}

As $\phi$ in Theorem~\ref{thm:GHI}, take
\begin{equation}\label{phi_beta,a}
 \phi_{\beta,a}(t)
 =
 \begin{cases}
  (a^2-t^2)^{\beta}, & 0\le t<a, \\
  0, & t\ge a,
 \end{cases}
\end{equation}
with  $\beta>-1$ and $a>0$.
Then $U_{\beta,a}(x)=\phi_{\beta,a}(|x|)$, $x\in\R^d$, and
the function $u_{\beta,a}$ is the periodization 
of $U_{\beta,a}(x)$ defined by \eqref{periodization U}.
In this case 
we denote $A_{\phi}$ by $A_{\beta,a}$.
Then 
$A_{\beta,a}$ can be calculated explicitly 
by using \eqref{FT-radial}, \eqref{FT-Fphi} and \eqref{cD} below.
Its derivatives $A^{(j)}_{\beta,a}$ are also calculated by \eqref{d Lam}.
That is,
\begin{equation}\label{A_beta,a}
\begin{cases}
 A_{\beta,a}(s)
 =\displaystyle
 \frac{\Gamma(\beta+1)}{\pi^\beta}
 a^{{\frac d2}+\beta}
 \Lambda_{{\frac d2}+\beta}(a:s),
\\[2ex]
 A_{\beta,a}^{(j)}(s)
 =\displaystyle
 (-1)^j{\frac{\Gamma(\beta+1)}{\pi^{\beta-j}}}a^{{\frac d2}+\beta+j}
 \Lambda_{{\frac d2}+\beta+j}(a:s).
\end{cases}
\end{equation}
Then we have the following corollary:

\begin{cor}\label{cor:GHI}
Let $\beta>-1$ and $a>0$.
Then 
\begin{multline}\label{GHI for u}
 S_\lambda(u_{\beta,a})(x)
 =
 \sigma_{\lambda}(U_{\beta,a})(x) 
 +\sum_{j=0}^{\dsharp}(-1)^j
  \Delta_j(\lambda^2:x)A_{\beta,a}^{(j)}(\lambda^2) 
\\ 
 +(-1)^{\dsharp+1}\sum_{m\in\Z^d\setminus\{0\}}
 \int_0^{\lambda^2}{\cD}_{\dsharp}(s:x-m)A_{\beta,a}^{(\dsharp+1)}(s)\,ds,
\end{multline}
for all $x\in\T^d$. 
\end{cor}

To prove Theorem~\ref{thm:GHI},
first we state the properties of the Bessel functions,
$D_{\alpha}$ and $\cD_{\alpha}$.
The following are the recurrence formulas for Bessel functions.
\begin{equation}\label{Bessel-recurrence-formula}
  {\frac d{ds}}\left(s^\nu J_\nu(s)\right)
  =
  s^\nu J_{\nu-1}(s),
 \quad
  {\frac d{ds}}\left({\frac{J_\nu(s)}{s^\nu}}\right)
  =
  -{\frac{J_{\nu+1}(s)}{s^\nu}}.
\end{equation}
The Bessel functions have also the following asymptotic behavior.
\begin{align}\label{Bessel-asymp-zero}
  J_{\nu}(s)
  &=
  \frac{s^{\nu}}{2^{\nu}\Gamma(\nu+1)} + O(s^{\nu+1})
  \quad\text{as}\quad
  s\to0,
\\ 
  J_{\nu}(s)
  &=
  \sqrt{\frac{2}{\pi s}} \cos\!\left(s-{\frac{2\nu+1}{4}}\pi\right) + O(s^{-3/2})
  \quad\text{as}\quad
  s\to\infty. \label{Bessel-asymp-infty}
\end{align}
If $\alpha>-1$, then
\begin{equation}\label{D-int}
 \int_0^t D_{\alpha}(s:x)\,ds = D_{\alpha+1}(t:x), 
\quad t\ge0,\ x\in\R^d,
\end{equation}
and 
\begin{equation}\label{cD}
  \cD_{\alpha}(s:x)
  =
  \frac{1}{\Gamma(\alpha+1)}
  \int_{|\xi|^2<s} (s-|\xi|^2)^{\alpha} e^{2\pi i\xi x}\,d\xi
  =
  \frac{s^{\frac d2+\alpha}}{\pi^{\alpha}}
  \frac{J_{{\frac d2}+\alpha}(2\pi\sqrt{s}|x|)}{(\sqrt{s}|x|)^{{\frac d2}+\alpha}}.
\end{equation}
If $\alpha>(d-1)/2$, for example, $\alpha=\dsharp$, then
\begin{equation}\label{D} 
  D_{\alpha}(s:x)
  = \sum_{m\in\Z^d}\cD_\alpha(s:x-m)=
  \frac{s^{\frac d2+\alpha}}{\pi^{\alpha}}
  \sum_{m\in\Z^d}
  \frac{J_{{\frac d2}+\alpha}(2\pi\sqrt{s}|x-m|)}{(\sqrt{s}|x-m|)^{{\frac d2}+\alpha}},
\end{equation}
where the sum in \eqref{D} converges absolutely.
In the above the equation \eqref{D-int} follows from the definition.
For the equalities \eqref{cD} and \eqref{D}, 
see (4.2) and (4.3) in \cite{Kuratsubo-Nakai-Ootsubo2010}, respectively. 
By \eqref{Bessel-recurrence-formula} and \eqref{cD} we have
\begin{equation}\label{cD-dif}
 \begin{cases}
 \dfrac{\partial}{\partial s}\cD_{\alpha}(s:x)
 =
 \cD_{\alpha-1}(s:x), & \alpha>0, 
\\[1ex]
 \dfrac{\partial}{\partial s}\cD_{0}(s:x)
 =
 \pi\left({\dfrac{\sqrt{s}}{|x|}}\right)^{{\frac d2}-1}
 J_{{\frac d2}-1}(2\pi|x|\sqrt{s}),
 & \alpha=0. 
 \end{cases}
\end{equation}

\begin{proof}[Proof of Theorem~\ref{thm:GHI}]\label{proof:thm:GHI}
Combining \eqref{Poisson} with \eqref{FT-Fphi} we have
\begin{equation*}
 {\hat f}_\phi(m)={\hat F_\phi}(m)=A_{\phi}(|m|^2), \quad m\in\Z^d.
\end{equation*}
Then
\begin{align*}
 &S_\lambda(f_\phi)(x)-D_0(\lambda^2:x)A_\phi(\lambda^2)
\\
 &=
 \sum_{|m|<\lambda}{\hat f}_\phi(m)e^{2\pi imx}
 -\left(\sum_{|m|<\lambda}e^{2\pi imx}\right)A_\phi(\lambda^2)
\\
 &=
 \sum_{|m|<\lambda}\left(A_\phi(|m|^2)-A_\phi(\lambda^2)\right) e^{2\pi imx}
\\
 &=
 \sum_{|m|<\lambda}\left(-\int_{|m|^2}^{\lambda^2}A_{\phi}^{(1)}(s)\,ds\right)
 e^{2\pi imx}
\\
 &=
 -\int_{0}^{\lambda^2}\left(\sum_{|m|^2<s}e^{2\pi imx}\right)A_{\phi}^{(1)}(s)
  \,ds
\\
 &=
 -\int_0^{\lambda^2}D_0(s:x) A_{\phi}^{(1)}(s)\,ds.
\end{align*}
Using \eqref{D-int} and integration by parts, we have
\begin{align*}
 &-\int_0^{\lambda^2}D_0(s:x) A_{\phi}^{(1)}(s)\,ds
\\
 &=
 -\left[D_1(s:x)A^{(1)}_\phi(s)\right]_0^{\lambda^2}+\int_0^{\lambda^2}D_1(s:x)
 A^{(2)}_\phi(s)\,ds
\\
 &=
-D_1(\lambda^2:x)A^{(1)}_\phi(\lambda^2)+\left[D_2(s:x)A^{(2)}_\phi(s)\right]_0^{\lambda^2}-\int_0^{\lambda^2}D_2(s:x)
 A^{(3)}_\phi(s)\,ds
\\
 &=
 \pt\cdots\cdots\cdots 
\\
 &=
 \sum_{j=1}^{\dsharp}(-1)^jD_j(\lambda^2:x)A^{(j)}_\phi(\lambda^2)
 +(-1)^{{\dsharp}+1}\int_0^{\lambda^2}D_{\dsharp}(s:x)
 A^{({\dsharp}+1)}_\phi(s)\,ds.
\end{align*}
Therefore, by \eqref{D} we have
\begin{multline*}
 S_\lambda(f_{\phi})(x) \\
 =
 \sum_{j=0}^{\dsharp}(-1)^jD_j(\lambda^2:x)A^{(j)}_\phi(\lambda^2) 
 +(-1)^{{\dsharp}+1}\sum_{m\in Z^d}\int_0^{\lambda^2}\cD_{\dsharp}(s:x-m)A^{({\dsharp}+1)}_\phi(s)\,ds.
\end{multline*}
Using \eqref{cD-dif} and integration by parts again, 
we have
\begin{align*}
 &
 \int_0^{\lambda^2} {\cD}_{\dsharp}(s:x) A_{\phi}^{(\dsharp+1)}(s)\,ds
\\
 &=
 {\cD}_{\dsharp}(\lambda^2:x) A_{\phi}^{(\dsharp)}(\lambda^2)
 -
 \int_0^{\lambda^2} {\cD}_{\dsharp-1}(s:x) A_{\phi}^{(\dsharp)}(s)\,ds
\\
 &=\ \cdots\cdots
\\
 &=
 \sum_{j=0}^{\dsharp}
  (-1)^j{\cD}_{\dsharp-j}(\lambda^2:x) A^{(\dsharp-j)}_\phi(\lambda^2)
\\
 &\pt
  +(-1)^{\dsharp+1} \int_0^{\lambda^2} 
   \pi\left({\dfrac{\sqrt{s}}{|x|}}\right)^{{\frac d2}-1}
   J_{{\frac d2}-1}(2\pi|x|\sqrt{s}) A_{\phi}(s) \,ds
\\
 &=
 \sum_{j=0}^{\dsharp}
  (-1)^{\dsharp-j}{\cD}_{j}(\lambda^2:x) A^{(j)}_\phi(\lambda^2)
\\
 &\pt
  +(-1)^{\dsharp+1} (2\pi)\int_0^{\lambda} 
    A_{\phi}(u^2)\frac{J_{{\frac d2}-1}(2\pi|x|u)}{(|x|u)^{\frac d2-1}}
   u^{d-1}\,du.
\end{align*}
Here we note that, by \eqref{FT-Fphi} and \eqref{FT-radial},
\begin{align*}
 \sigma_{\lambda}(F_{\phi})(x)
 &=
 \int_{|\xi|<\lambda} {\hat F}_\phi(\xi) e^{2\pi i x\xi}\,d\xi
 =
 \int_{|\xi|<\lambda} A_\phi(|\xi|^2) e^{2\pi i x\xi}\,d\xi
\\
 &=
  2\pi\int_0^{\lambda} 
    A_{\phi}(u^2)\frac{J_{{\frac d2}-1}(2\pi|x|u)}{(|x|u)^{\frac d2-1}}
   u^{d-1}\,du.
\end{align*}
Therefore
\begin{multline*}
 S_\lambda(f_{\phi})(x) 
 =
 \sigma_{\lambda}(F_{\phi})(x)
 +
 \sum_{j=0}^{\dsharp}(-1)^j\Delta_j(\lambda^2:x)A^{(j)}_\phi(\lambda^2) \\
 +(-1)^{{\dsharp}+1}\sum_{m\in Z^d,m\ne0}\int_0^{\lambda^2}\cD_{\dsharp}(s:x-m)A^{({\dsharp}+1)}_\phi(s)\,ds.
\qedhere
\end{multline*}
\end{proof}

\section{Further properties of Bessel functions}\label{sec:B function}

In this section we collect properties of Bessel functions which we use in the following sections.
Adding to \eqref{Bessel-asymp-infty},
we first recall the following asymptotic behavior of the Bessel functions:
\begin{multline}\label{asymp-infty2}
 J_{\nu}(s)
 =
 \sqrt{\dfrac2{\pi s}} 
 \left(\cos\left(s-\dfrac{2\nu+1}4\pi\right)
       -\frac{4\nu^2-1}{8s}\sin\left(s-\frac{2\nu+1}4\pi\right)
 \right)+O(s^{-5/2})
 \\
 \text{as}\quad
 s\to\infty.
\end{multline}

\begin{lem}[Gradshteyn and Ryzhik~\cite{Gradshteyn-Ryzhik2007}, page~676, {\bf 6.561,~14}]\label{lem:Bessel11}
Let $\nu+1>-\mu>-1/2$. Then
\begin{equation*}
 \int_0^{\infty} x^{\mu} J_{\nu}(x)\,dx
 =
 2^{\mu}\frac{\Gamma((1+\nu+\mu)/2)}{\Gamma((1+\nu-\mu)/2)}.
\end{equation*}
\end{lem}

Take
\begin{equation*}
 \nu=\dfrac 12+\beta \ \text{and} \ \mu=-\dfrac 12-\beta 
\end{equation*}
in Lemma~\ref{lem:Bessel11},
and we have the following corollary.

\begin{cor}\label{cor:Bessel11}
Let $\beta>-1$. Then
\begin{equation*}
 \int_0^{\infty}\frac{J_{\frac12+\beta}(s)}{s^{\frac12+\beta}}\,ds
 =
 \frac1{2^\beta\Gamma(1+\beta)}\sqrt{\dfrac\pi 2}.
\end{equation*}
\end{cor}

\begin{lem}\label{lem:Bessel1a}
Let $\beta>-1$. Then
\begin{equation*}
 \int_u^{\infty}\frac{J_{\frac12+\beta}(s)}{s^{\frac12+\beta}}\,ds
 =
 \sqrt{\dfrac2{\pi}} 
 u^{-\beta-1}\cos\left(u-\frac\beta2\pi\right)+O(u^{-\beta-2})
 \quad\text{as}\quad u\to\infty.
\end{equation*}
\end{lem}

\begin{proof}
Use \eqref{asymp-infty2} and two equations 
\begin{align*}
 \int_u^{\infty}\frac{\cos(s-\theta)}{s^{1+\beta}}\,ds
 &=
 u^{-\beta-1}\cos\left(u-\theta+\frac{\pi}2\right)+O(u^{-2-\beta}),
\\
 \int_u^{\infty}\frac{\sin(s-\theta)}{s^{2+\beta}}\,ds
 &=
 O(u^{-2-\beta}),
\end{align*}
and we have the conclusion.
\end{proof}

\begin{lem}[Gradshteyn and Ryzhik~\cite{Gradshteyn-Ryzhik2007}, page~683, {\bf 6.575,~1}]\label{lem:Bessel1}
Let $\nu+1>\mu>-1$, $A>0$ and $B>0$. 
Then
\begin{equation*}
 \int_0^{\infty} J_{\nu+1}(As)J_\mu(Bs)s^{\mu-\nu}\,ds
 =
\begin{cases}
 0 &  (A<B) \\
 \dfrac{(A^2-B^2)^{\nu-\mu}B^\mu}{2^{\nu-\mu}A^{\nu+1}\Gamma(\nu-\mu+1)}
 & (A> B).
 \end{cases}
 \end{equation*}
\end{lem}

Take
\begin{equation*}
 \nu+1=\mu+\beta+1, \ 
 A=1 \ \text{and} \ B=\dfrac ta,
\end{equation*}
in Lemma~\ref{lem:Bessel1}, 
and we have the following corollary.

\begin{cor}\label{cor:Bessel1}
Let $\mu>-1$, $a>0$, $\beta>-1$ and $t>0$. 
Then
\begin{equation*}
 2^{\beta}\Gamma(\beta+1)a^{2\beta}
 \int_0^{\infty}
  \frac{J_{\mu}(\frac ta s)J_{\mu+\beta+1}(s)}
       {\left(\frac ta\right)^{\mu}s^{\beta}}
 \,ds
 =
 \begin{cases}
  0 & (t>a) \\
  (a^2-t^2)^{\beta} & (0<t<a),
 \end{cases}
\end{equation*}
\end{cor}

\begin{lem}\label{lem:Bessel2}
Let $\nu>-1$, $\mu>-1$, $A>0$ and $B>0$. 
Then
\begin{align*}
  J_{\nu}(Au) J_{\mu}(Bu)
 &=
 \dfrac 1{\pi\sqrt{AB}u}
 \bigg(
  {\cos\left((A-B)u-\frac{\nu-\mu}{2}\pi\right)}
\\
 &\phantom{*********}+
  {\cos\left((A+B)u-\frac{\nu+\mu+1}{2}\pi\right)}
 \bigg)
\\
 &\phantom{***********}+
 \frac1{\sqrt{AB}}
 \left(
  \frac1A+\frac1B 
 \right)
 O(u^{-2})
 \quad\text{as}\quad u\to\infty.
\end{align*}
\end{lem}

\begin{proof}
Use \eqref{Bessel-asymp-infty} and 
$2\cos\theta \cos\phi=\cos(\theta+\phi)+\cos(\theta-\phi)$,
and we have the conclusion.
\end{proof}

Let 
\begin{equation*}
 \sign(r)
 =
 \begin{cases}
 1, & r>0, \\
 0, & r=0, \\
 -1,& r<0. 
\end{cases}
\end{equation*}

\begin{lem}\label{lem:psi}
For $\beta>-1$,
let
\begin{equation}\label{psi}
 \psi_{\beta}(\lambda,r)
 =
 |r|^{\beta} 
 \int_{2\pi|r|\lambda}^{\infty}
  \frac1{s^{\beta+1}} \cos\left(s-{\sign}(r){\frac{\beta+1}2}\pi\right)
 \,ds, 
 \quad 
 \lambda>0, \ r\in\R\setminus\{0\}.
\end{equation}
Then 
$\psi_{\beta}$ has the following properties:
\begin{enumerate}
\item
For each $r\in\R\setminus\{0\}$,
\begin{equation}\label{psi nega}
 \psi_{\beta}(\lambda,r)=|r|^{-1}O(\lambda^{-\beta-1})
 \quad\text{as}\quad \lambda\to\infty.
\end{equation}
\item
If $\beta>0$, then
\begin{equation}\label{psi pos}
 \psi_{\beta}(\lambda,r)=O(\lambda^{-\beta})
 \quad\text{as}\quad \lambda\to\infty
\end{equation}
uniformly with respect to $r\in\R\setminus\{0\}$.
\item
If $-1<\beta\le0$,
then, for each $\lambda>0$, 
$-\psi_{\beta}(\lambda,r)$ has values 
\begin{equation*}
 \left(\frac{2+\beta}{4\lambda}\right)^\beta
 \int_\pi^\infty\frac{-\sin s}{(s+\frac\beta 2 \pi)^{\beta+1}}ds
 \quad\text{at}\quad r=\frac{2+\beta}{4\lambda},
\end{equation*}
and 
\begin{equation*}
 \left(\frac{2-\beta}{4\lambda}\right)^\beta
 \int_\pi^\infty\frac{\sin s}{(s-\frac\beta 2 \pi)^{\beta+1}}ds
 \quad\text{at}\quad r=-\frac{2-\beta}{4\lambda}.
\end{equation*}
\end{enumerate}
\end{lem}

\begin{proof}
(i)
Using the equation
\begin{equation}\label{cos/s}
 \int_v^{\infty}\frac{\cos(s-\theta)}{s^{\beta+1}}\,ds
 =
 O(v^{-\beta-1}),
\end{equation}
we have
$
 \psi_{\beta}(\lambda,r)
 =
 |r|^{\beta} 
 O((|r|\lambda)^{-\beta-1})
 =
 |r|^{-1} 
 O(\lambda^{-\beta-1})
$.

\noindent
(ii)
Changing of variables, we have
\begin{equation*}
 \psi_{\beta}(\lambda,r)
 =
 \frac 1{(2 \pi \lambda)^\beta}
 \int_1^\infty\frac{\cos(2\pi |r|\lambda u-\sign(r)\frac{\beta+1}2 \pi)}{u^{\beta+1}}\,du,
\end{equation*}
which shows $\psi_{\beta}(\lambda,r)=O(\lambda^{-\beta})$ uniformly with respect to $r$,
since $\beta>0$.

\noindent
(iii)
If $r>0$, 
then
\begin{equation*}
 -\psi_{\beta}(\lambda,r)
 =
 r^{\beta}
 \int_{2\pi r\lambda}^{\infty}
 \frac{-\sin\left(s-{\frac{\beta}2}\pi\right)}{s^{\beta+1}} 
 \,ds
 =
 r^{\beta}
 \int_{2\pi r\lambda-\frac{\beta}2\pi}^{\infty}
 \frac{-\sin s}{\left(s+{\frac{\beta}2}\pi\right)^{\beta+1}} 
 \,ds,
\end{equation*}
and, if $r<0$,
then
\begin{equation*}
 -\psi_{\beta}(\lambda,r)
 =
 (-r)^{\beta}
 \int_{2\pi(-r)\lambda}^{\infty}
 \frac{\sin\left(s+{\frac{\beta}2}\pi\right)}{s^{\beta+1}} 
 \,ds
 =
 (-r)^{\beta}
 \int_{2\pi(-r)\lambda+\frac{\beta}2\pi}^{\infty}
 \frac{\sin s}{\left(s-{\frac{\beta}2}\pi\right)^{\beta+1}} 
 \,ds.
\end{equation*}
Then we have the conclusion.
\end{proof}

\begin{lem}\label{lem:Bessel3}
Let $\mu>-1$, $\beta>-1$ and $a>0$. 
If $0<t<a$ or $t>a$, then
\begin{multline*} 
 \int_{2\pi a\lambda}^{\infty}
  \frac{J_{\mu}(\frac ta u)J_{\mu+\beta+1}(u)}
       {\left(\frac ta\right)^{\mu}u^{\beta}}
 \,du
\\
 =
 \dfrac1{\pi a^{\beta}} \left(\frac at\right)^{\mu+\frac12}
 \psi_{\beta}(\lambda,a-t)
 + \frac1{a^{\beta+1}}
  \left(\frac at\right)^{\mu+\frac12}
  \left(1+{\frac at}\right)O(\lambda^{-\beta-1})
 \quad\text{as}\quad \lambda\to\infty,
\end{multline*}
where the term $O(\lambda^{-\beta-1})$ is uniform with respect to $t$.
\end{lem}

\begin{proof}
By Lemma~\ref{lem:Bessel2}
we have
\begin{multline*}
 J_\mu(\frac ta u)J_{\mu+\beta+1}(u)
\\
 =
 \frac1{\pi u}\sqrt{\frac at}
 \left(\cos\left(\frac{|a-t|}{a}u-\sign(a-t)\frac{1+\beta}2\pi\right)
       +\cos\left(\frac{a+t}{a}u-\frac{2\mu+\beta+1}2\pi\right)\right)
\\
 +\sqrt{\frac at}\left(1+{\frac at}\right)O(u^{-2}).
\end{multline*}
Let
\begin{align*}
 I_1
 &=
 \frac1{\pi}\left(\frac at\right)^{\mu+\frac12}
 \int_{2\pi a\lambda}^{\infty}
  \frac{\cos\left(\frac{|a-t|}{a}u-\sign(a-t)\frac{\beta+1}2\pi\right)}
       {u^{\beta+1}}\,du,
\\
 I_2
 &=
 \frac1{\pi}\left(\frac at\right)^{\mu+\frac12}
 \int_{2\pi a\lambda}^{\infty}
  \frac{\cos\left(\frac{a+t}{a}u-\frac{2\mu+\beta+1}2\pi\right)}
       {u^{\beta+1}}\,du.
\end{align*}
Then
\begin{equation*}
 \int_{2\pi a\lambda}^{\infty}
  \frac{J_{\mu}(\frac ta u)J_{\mu+\beta+1}(u)}
       {\left(\frac ta\right)^{\mu}u^{\beta}}
 \,du
 =
 I_1+I_2
 +
 \left(\frac at\right)^{\mu+\frac12}\left(1+{\frac at}\right) a^{-\beta-1}O(\lambda^{-\beta-1}).
\end{equation*}
Changing of variables and using \eqref{cos/s}, 
we have
\begin{align*}
 I_1
 &=
 \frac1{\pi}\left(\frac at\right)^{\mu+\frac12}
 \left(\frac{|a-t|}{a}\right)^{\beta}
 \int_{2\pi|a-t|\lambda}^{\infty}
  \frac{\cos\left(s-\sign(a-t)\frac{\beta+1}2\pi\right)}
       {s^{\beta+1}}\,ds
\\
 &=
 \dfrac1{\pi a^{\beta}} \left(\frac at\right)^{\mu+\frac12}
 \psi_{\beta}(\lambda,a-t)
\end{align*}
and
\begin{align*}
 I_2
 &=
 \frac1{\pi}\left(\frac at\right)^{\mu+\frac12}
 \left(\frac{a+t}{a}\right)^{\beta}
 \int_{2\pi(a+t)\lambda}^{\infty}
  \frac{\cos\left(s-\frac{2\mu+\beta+1}2\pi\right)}
       {s^{\beta+1}}\,ds
\\
 &=
 \left(\frac at\right)^{\mu+\frac12}
 \left(\frac{a+t}{a}\right)^{\beta}(a+t)^{-\beta-1}
 O(\lambda^{-\beta-1})
\\
 &=
 \frac1{a^{\beta+1}}\left(\frac at\right)^{\mu+\frac12}\left(\frac{a}{a+t}\right)
 O(\lambda^{-\beta-1}).
\end{align*}
Then we have the conclusion.
\end{proof}

\begin{lem}[Gradshteyn and Ryzhik~\cite{Gradshteyn-Ryzhik2007}, page~683, {\bf 6.574,~2}]\label{lem:Bessel4}
Let $\nu+\mu+1>\kappa>0$ and $b>0$.
Then
\begin{equation*}
 \int_0^{\infty} \frac{J_{\nu}(bs)J_{\mu}(bs)}{s^{\kappa}}\,ds
 =
 \frac{b^{\kappa-1}\Gamma(\kappa)\Gamma(\frac{\mu+\nu-\kappa+1}{2})}
      {2^{\kappa}\Gamma(\frac{\mu-\nu+\kappa+1}{2})
       \Gamma(\frac{-\mu+\nu+\kappa+1}{2})\Gamma(\frac{\mu+\nu+\kappa+1}{2})}.
\end{equation*}
Especially, for  $\mu>-1$ and $\beta>0$,
\begin{equation*}
 \int_0^{\infty}
  \frac{J_{\mu+\beta+1}(s)J_{\mu}(s)}
       {s^{\beta}}
 \,ds
 =0.
\end{equation*}
\end{lem}

\begin{lem}[Gradshteyn and Ryzhik~\cite{Gradshteyn-Ryzhik2007}, page~660, {\bf 6.512,~3}]\label{lem:Bessel42}
Let $\nu>0$. Then
\begin{equation*}
 \int_0^{\infty}
  J_\nu(s)J_{\nu-1}(s)
 \,ds
 =\dfrac12.
\end{equation*}
\end{lem}

\begin{lem}\label{lem:Bessel43}
Let $\mu>-1$ and $\beta\ge0$. Then
\begin{multline*}
 \int_u^{\infty}
  \frac{J_{\mu+\beta+1}(s)J_{\mu}(s)}{s^{\beta}}
 \,ds
 =
 \begin{cases}
 -\dfrac1{\pi\beta}
 \sin\left(\dfrac{\beta\pi}2\right)
 u^{-\beta}
 +O(u^{-1-\beta}),
 & \text{if}\ \beta>0,\\ 
 O(u^{-1}), 
 & \text{if}\ \beta=0,
\end{cases}  
\\ \quad\text{as}\quad u\to\infty.
\end{multline*}
\end{lem}

\begin{proof}
By Lemma~\ref{lem:Bessel2}
we have
\begin{equation}\label{JJ}
 J_{\mu+\beta+1}(s)J_{\mu}(s)
 =
 \frac1{\pi s}
 \left(\cos\left(\frac{\beta+1}2\pi\right)
       +\cos\left(2s-\frac{2\mu+2+\beta}2\pi\right)\right)
 +O(s^{-2}).
\end{equation}
Using
\begin{equation*}
 \frac1{\pi}
 \int_u^{\infty}
 \frac{1}{s^{\beta+1}}\cos\left({\frac{\beta+1}2}\pi\right)\,ds
 =
 \begin{cases}
  -\dfrac1{\pi\beta}\sin\left({\dfrac{\beta\pi}2}\right)u^{-\beta}, 
   & \beta>0, \\
  0, & \beta=0,
 \end{cases}
\end{equation*}
and \eqref{cos/s}, we have the conclusion.
\end{proof}

\begin{cor}\label{cor:Bessel44}
Let $\mu>-1$ and $\beta>-1$. Then
\begin{multline}\label{x=a}  
 2^\beta\Gamma(\beta+1)a^{2\beta}
 \int_0^{2\pi a\lambda}
 {\frac{J_{\mu+\beta+1}(s)J_{\mu}(s)}{s^\beta}}
 \,ds
\\
 =
 L_{\beta,a}
 \lambda^{-\beta} 
 +
 \begin{cases}
  O(\lambda^{-\beta-1}), & \beta\ge0, \\
  O(1), & -1<\beta<0,
 \end{cases}
 \quad\text{as $\lambda\to\infty$}.
\end{multline}
In the above $L_{\beta,a}$ is defined by \eqref{J1}, that is,
\begin{equation*} 
 L_{\beta,a}
 =
 \dfrac{\Gamma(\beta+1)}2 \left(\dfrac a\pi\right)^{\beta}
 \left(\dfrac{\sin\dfrac{\beta\pi}2}{\dfrac{\beta\pi}2}\right), 
\end{equation*}
where $\left(\sin\frac{\beta\pi}2\right)/\frac{\beta\pi}2$
is regarded as $1$ if $\beta=0$. 
\end{cor}

\begin{proof}
{\bf Case 1.} 
Let $\beta\ge0$.
By Lemmas~\ref{lem:Bessel4} and \ref{lem:Bessel42}
we have 
\begin{equation*}
 \int_0^\infty
 \frac{J_{\mu+\beta+1}(s)J_{\mu}(s)}{s^{\beta}} \,ds
 =
 \begin{cases}
  0, & \text{if}\ \beta>0,\\
  \frac12, & \text{if}\ \beta=0.
 \end{cases}  
\end{equation*}
Then
\begin{equation*}
 \int_0^{2\pi a\lambda}
 {\frac{J_{\mu+\beta+1}(s)J_{\mu}(s)}{s^\beta}}
 \,ds
 =
 -\int_{2\pi a\lambda}^{\infty}
 {\frac{J_{\mu+\beta+1}(s)J_{\mu}(s)}{s^\beta}}
 \,ds
 +
 \begin{cases}
  0, & \text{if}\ \beta>0,\\
  \frac12, & \text{if}\ \beta=0.
\end{cases}  
\end{equation*}
By Lemma~\ref{lem:Bessel43} we have \eqref{x=a}.

\noindent
{\bf Case 2.} 
Let $-1<\beta<0$.
By \eqref{JJ} we have
\begin{align*}
 &\int_1^{2\pi a\lambda}
 \frac{J_{\mu+\beta+1}(s)J_{\mu}(s)}{s^\beta}\,ds
\\
 &=
 \frac1{\pi}
 \int_1^{2\pi a\lambda}\frac{1}{s^{\beta+1}}
 \left\{\cos\left({\frac{\beta+1}2}\pi\right)
       +\cos\left(2s-\frac{2\nu+2+\beta}2\pi\right)
 \right\}\,ds
  +O(\lambda^{-\beta-1}) 
\\
 &=
 \frac1{\pi}\left(-\sin\left({\frac{\beta\pi}2}\right)\right)
 \frac{(2\pi a\lambda)^{-\beta}}{-\beta}
  +O(1).
\end{align*}
Since
\begin{equation*}
 \int_0^1
 \frac{J_{\mu+\beta+1}(s)J_{\mu}(s)}{s^\beta}\,ds
\end{equation*}
is a constant, we have the conclusion.
\end{proof}

The following lemma is an extension of 
\cite[Lemma~3.3]{Kuratsubo-Nakai-Ootsubo2010}
which is for the case $\beta=0$.

\begin{lem}\label{lem:Bessel5}
Let $\nu>0$, $\lambda>0$, $\omega\ge0$ and $\beta>-1$.
Then
\begin{equation}\label{int part}
 \int_0^{\lambda}
  \frac{J_{\nu+\beta+1}(s)J_{\nu}(\omega s)}{\omega^{\nu}s^{\beta}}\,ds
 =
 \int_0^{\lambda}
  \frac{J_{\nu+\beta}(s)J_{\nu-1}(\omega s)}{\omega^{\nu-1}s^{\beta}}\,ds
 -
 \frac{J_{\nu+\beta}(\lambda)J_{\nu}(\omega\lambda)}
      {\omega^{\nu}\lambda^{\beta}}.
\end{equation}
In the above, if $\omega=0$, 
then by \eqref{Bessel-at-zero} we regard \eqref{int part} as
\begin{equation*}
 \int_0^{\lambda}
 \frac{J_{\nu+\beta+1}(s)}{s^{\beta}}
 \frac{s^{\nu}}{2^{\nu}\Gamma(\nu+1)}
 \,ds 
 =
 \int_0^{\lambda}
 \frac{J_{\nu+\beta}(s)}{s^{\beta}}\frac{s^{\nu-1}}{2^{\nu-1}\Gamma(\nu)}\,ds 
 -\frac{J_{\nu+\beta}(\lambda)}{{\lambda}^{\beta}}
  \frac{{\lambda}^{\nu}}{2^{\nu}\Gamma(\nu+1)}.
\end{equation*}
\end{lem}

\begin{proof}
{\bf Case 1.} 
Let $\omega>0$.
Then, 
using \eqref{Bessel-recurrence-formula} and integration by parts, 
we have
\begin{align*}
 &\int_0^{\lambda}
 \frac{J_{\nu+\beta}(s)J_{\nu-1}(\omega s)}{s^{\beta}}\omega^{\nu+1}\,ds 
\\
 &=
 \int_0^{\lambda}
 \frac{J_{\nu+\beta}(s)}{s^{\nu+\beta}}
 \big((\omega s)^{\nu}J_{\nu-1}(\omega s)\big)\omega
 \,ds
\\
 &=
 \int_0^{\lambda}
 \frac{J_{\nu+\beta}(s)}{s^{\nu+\beta}}
 \frac{\partial}{\partial s}\big((\omega s)^\nu J_\nu(\omega s)\big)
 \,ds
\\
 &=
 \left[
  \frac{J_{\nu+\beta}(s)}{s^{\nu+\beta}}
  \big((\omega s)^\nu J_\nu(\omega s)\big)
 \right]_0^{\lambda}
 +\int_0^{\lambda}
  \frac{J_{\nu+\beta+1}(s)}{s^{\nu+\beta}}
  \big((\omega s)^\nu J_\nu(\omega s)\big)
  \,ds
\\
 &=
 \frac{J_{\nu+\beta}(\lambda)J_\nu(\omega\lambda)}{\lambda^\beta}\omega^{\nu}
 +\int_0^{\lambda}
 \frac{J_{\nu+\beta+1}(s) J_{\nu}(\omega s)}{s^\beta}{\omega}^{\nu}
  \,ds.
\end{align*}

\noindent
{\bf Case 2.} 
Let $\omega=0$.
Then, 
\begin{align*}
 &\int_0^{\lambda}
 \frac{J_{\nu+\beta}(s)}{s^{\beta}}\frac{s^{\nu-1}}{2^{\nu-1}\Gamma(\nu)}\,ds 
\\
 &=
 \int_0^{\lambda}
 \frac{J_{\nu+\beta}(s)}{s^{\beta+\nu}}
 \frac{s^{2\nu-1}}{2^{\nu-1}\Gamma(\nu)}
 \,ds 
\\
 &=
 \left[
  \frac{J_{\nu+\beta}(s)}{s^{\beta+\nu}}
  \frac{s^{2\nu}}{2{\nu}2^{\nu-1}\Gamma(\nu)}
 \right]_0^{\lambda}
 +\int_0^{\lambda}
  \frac{J_{\nu+\beta+1}(s)}{s^{\beta+\nu}}
  \frac{s^{2\nu}}{2{\nu}2^{\nu-1}\Gamma(\nu)}
 \,ds 
\\
 &=
  \frac{J_{\nu+\beta}(\lambda)}{{\lambda}^{\beta}}
  \frac{{\lambda}^{\nu}}{2^{\nu}\Gamma(\nu+1)}
 +\int_0^{\lambda}
  \frac{J_{\nu+\beta+1}(s)}{s^{\beta}}
  \frac{s^{\nu}}{2^{\nu}\Gamma(\nu+1)}
  \,ds.
\end{align*}
Therefore we have the conclusion.
\end{proof}

Using Lemma~\ref{lem:Bessel5} several times, we have the following.
\begin{cor}\label{cor:Bessel51}
Let $d\ge3$, $\beta>-1$, $a>0$ and $t\ge0$,
and let $\dsharp$ be as in \eqref{kd}.
\begin{enumerate}
\item 
If $t>0$, then
\begin{multline*}
 \int_0^{2\pi a\lambda}
  \frac{J_{\frac d2+\beta}(s)J_{\frac d2-1}(\frac ta s)}
       {(\frac ta)^{\frac d2-1}s^{\beta}}
 \,ds
\\
 =
 \int_0^{2\pi a\lambda}
  \frac{J_{\frac d2+\beta-\dsharp+1}(s)J_{\frac d2-\dsharp}(\frac ta s)}
       {(\frac ta)^{\frac d2-1-\dsharp+1}s^{\beta}}
 \,ds
 -\sum_{\ell=1}^{\dsharp-1}
  \frac{J_{\frac d2+\beta-\ell}(2\pi a\lambda)J_{\frac d2-\ell}(2\pi t\lambda)}
       {(\frac ta)^{\frac d2-\ell}(2\pi a\lambda)^{\beta}},
\end{multline*}
where the first term of the right hand side in the above equation 
is equal to
\begin{equation*}
\begin{cases}
\displaystyle
 \int_0^{2\pi a\lambda}
  \dfrac{J_{\beta+\frac12}(s)J_{-\frac12}(\frac ta s)}
        {(\frac ta)^{-\frac12}s^\beta}\,ds,
   & \text{if $d$ is odd}, \\[3ex]
\displaystyle
 \int_0^{2\pi a\lambda}
  \dfrac{J_{\beta+1}(s)J_0(\frac ta s)}{s^\beta}\,ds, 
   & \text{if $d$ is even}.
\end{cases}
\end{equation*}
\item
If $t=0$, then
\begin{multline*}
 \int_0^{2\pi a\lambda}
  \frac{J_{\frac d2+\beta}(s)s^{\frac d2-1}}
       {2^{\frac d2-1}\Gamma(\frac d2)s^{\beta}}
 \,ds
\\
 =
 \int_0^{2\pi a\lambda}
  \frac{J_{\frac d2+\beta-\dsharp+1}(s)s^{\frac d2-\dsharp}}
       {2^{\frac d2-\dsharp}\Gamma(\frac d2-\dsharp+1)s^{\beta}}
 \,ds
 -\sum_{\ell=1}^{\dsharp-1}
  \frac{(\pi a\lambda)^{\frac d2-\ell} 
         J_{\frac d2+\beta-\ell}(2\pi a\lambda)}
       {\Gamma(\frac d2-\ell+1) (2\pi a\lambda)^{\beta}},
\end{multline*}
where the first term of the right hand side in the above equation 
is equal to
\begin{equation*}
\begin{cases}
\displaystyle
 \sqrt{\frac 2\pi}
 \int_0^{2\pi a\lambda}
  \dfrac{J_{\beta+\frac12}(s)}{s^{\beta+\frac12}}ds,
   & \text{if $d$ is odd}, \\[3ex]
\displaystyle
 \int_0^{2\pi a\lambda}
  \dfrac{J_{\beta+1}(s)}{s^\beta}ds, 
   & \text{if $d$ is even}.
\end{cases}
\end{equation*}
\end{enumerate}
\end{cor}

\section{Fourier inversion for the function $U_{\beta,a}(x)$}\label{sec:F inversion}

Recall that $U_{\beta,a}(x)=\phi_{\beta,a}(|x|)$ with
\begin{equation*}
 \phi_{\beta,a}(t)
 =
 \begin{cases}
  (a^2-t^2)^{\beta}, & 0\le t<a, \\
  0, & t\ge a.
 \end{cases}
\end{equation*}
Then $\phi_{\beta,a}$ and $U_{\beta,a}$ have the following properties:
\begin{enumerate}
\item 
$\phi_{\beta,a}$ is continuous and bounded variation
if and only if $\beta>0$;
\item
$\phi_{0,a}$ is discontinuous and bounded variation
and $U_{0,a}$ is
the indicator function 
of the ball in $\R^d$ centered at the origin with the radius $a$;
\item
$\phi_{\beta,a}$ is not bounded variation
if and only if $\beta<0$;
\item 
$U_{\beta,a}$ is integrable on $\R^d$
if and only if $\beta>-1$ 
for all dimensions $d$.
\end{enumerate}

In this section, for $\beta>-1$, 
we investigate the behavior of 
the Fourier spherical partial integral $\sigma_\lambda(U_{\beta,a})(x)$ 
as $\lambda\to\infty$.

For $a>0$, let 
\begin{align}\label{tEa}
 \tE_a&=\{x\in\R^d: x\neq 0\ \text{and}\ |x|\neq a\},
\\
\label{tGa}
 \tG_a&=\{x\in\R^d: |x|=a\}.
\end{align}
Then
$\R^d=\{0\}\cup \tE_a\cup \tG_a$.

\begin{thm}\label{thm:FI}
Let $d\ge1$, $a>0$ and $\beta>-1$.
Then
\begin{equation}\label{F inversion}
 \sigma_\lambda(U_{\beta,a})(x)
 =
 2^\beta\Gamma(\beta+1)a^{2\beta}
 \int_0^{2\pi a\lambda}
 \frac{J_{{\frac d2}-1}(\frac{|x|}{a}s)J_{{\frac d2}+\beta}(s)}
      {\left(\frac{|x|}{a}\right)^{\frac d2-1}s^\beta}
 \,ds,
\end{equation}
for all $x\in\R^d$ and $\lambda>0$.
Moreover, $\sigma_\lambda(U_{\beta,a})$ has the following properties:
\begin{enumerate}
\item
At $x=0$, 
\begin{multline}\label{FI 0}
 \sigma_{\lambda}(U_{\beta,a})(0)
 =
 U_{\beta,a}(0)-P_{\beta,a}^{[d]}
  \cos\left(2\pi a\lambda-\frac{d-1+2\beta}4\pi\right)
  \lambda^{\frac{d-3}2-\beta}
\\
 +O(\lambda^{\frac{d-5}2-\beta})
 \quad\text{as}\quad \lambda\to\infty,
\end{multline}
where $P_{\beta,a}^{[d]}$ is the constant defined by \eqref{P}.
Consequently,
\begin{enumerate}
\item
if $\beta>(d-3)/2$,
then
\begin{equation*}
\sigma_{\lambda}(U_{\beta,a})(0)
 =
 U_{\beta,a}(0)+O(\lambda^{\frac{d-3}2-\beta})
 \quad\text{as}\quad \lambda\to\infty,
\end{equation*}
\item
if  $-1<\beta\le(d-3)/2$,
then
$\sigma_{\lambda}(U_{\beta,a})(x)$ reveals the Pinsky phenomenon,
that is,
\begin{align*}
 \liminf_{\lambda\to\infty}
 \frac{\sigma_{\lambda}(U_{\beta,a})(0)-U_{\beta,a}(0)}
      {\lambda^{(d-3)/2-\beta}}
 &=-P_{\beta,a}^{[d]},
\\
 \limsup_{\lambda\to\infty}
 \frac{\sigma_\lambda(U_{\beta,a})(0)-U_{\beta,a}(0)}
      {\lambda^{(d-3)/2-\beta}}
 &=P_{\beta,a}^{[d]}.
\end{align*}
\end{enumerate}

\item
For $x\in\tG_a$, 
\begin{equation*}
 \lim_{\lambda\to\infty}
 \frac{\sigma_{\lambda}(U_{\beta,a})(x)}{\lambda^{-\beta}}
 =
 L_{\beta,a},
\end{equation*}
where $L_{\beta,a}$ is the constant defined by \eqref{J1},
and consequently, 
\begin{equation*}
 \lim_{\lambda\to\infty}\sigma_{\lambda}(U_{\beta,a})(x)
 =
 \begin{cases}
 0, & \beta>0, \\
 \dfrac12, & \beta=0, \\
 \infty, & -1<\beta<0.
 \end{cases} 
\end{equation*}

\item
For $x\in \tE_a$,
\begin{equation*}
 \sigma_{\lambda}(U_{\beta,a})(x)= U_{\beta,a}(x)+O(\lambda^{-\beta-1})
 \quad\text{as}\quad \lambda\to\infty,
\end{equation*}
where the last term $O(\lambda^{-\beta-1})$ is uniform on any compact subset of $\tE_a$.

\item
If $\beta>0$, 
then
\begin{equation*}
 \lim_{\lambda\to\infty}\sigma_{\lambda}(U_{\beta,a})(x)
 =
 U_{\beta,a}(x)
 \quad\text{for}\ x\in\tE_a\cup\tG_a 
\end{equation*}
where the convergence is uniform on any compact subset of $\tE_a\cup\tG_a$ 
and 
$\sigma_{\lambda}(U_{\beta,a})$ does not reveal the Gibbs-Wilbraham phenomenon. 

\item
If $-1<\beta\le0$, then 
\begin{equation*}
 \lim_{\lambda\to\infty}\sigma_{\lambda}(U_{\beta,a})(x)
 =
 U_{\beta,a}(x)
 \quad\text{for}\ x\in\tE_a,
\end{equation*}
where the convergence is uniform on any compact subset of $\tE_a$, 
and
$\sigma_{\lambda}(U_{\beta,a})$
reveals a phenomenon like the Gibbs-Wilbraham phenomenon near $\tG_a$.
More precisely, the following holds:
For each $x_0\in\tG_a$, 
let $\{x_\lambda^\pm\}$ be the sequences which satisfy 
$\displaystyle\lim_{\lambda\to\infty}x_\lambda^\pm=x_0$ 
and $|x_\lambda^\pm|=a\mp(2\pm\beta)/(4\lambda)$.
Then
\begin{equation*}
 \lim_{\lambda\to\infty}
 \frac{\sigma_{\lambda}(U_{\beta,a})(x_{\lambda}^+)-U_{\beta,a}(x_{\lambda}^+)}
      {\lambda^{-\beta}}
 =
 G_{\beta,a}^{+},
\end{equation*}
and
\begin{equation*}
 \lim_{\lambda\to\infty}
 \frac{\sigma_{\lambda}(U_{\beta,a})(x_{\lambda}^-)-U_{\beta,a}(x_{\lambda}^-)}
      {\lambda^{-\beta}}
 =
 G_{\beta,a}^{-},
\end{equation*}
where $G_{\beta,a}^{\pm}$ are the constants defined by \eqref{G}.
\end{enumerate}
\end{thm}

\begin{rem}\label{rem:FI}
The function $U_{\beta,a}$ is piecewise smooth 
in the sense of Pinsky~\cite{Pinsky1994}
if and only if $\beta$ is a nonnegative integer.
In this case the result in Theorem~\ref{thm:FI} (i) is contained in \cite[Theorem 1a]{Pinsky1994}.
\end{rem}

In the following we first prove \eqref{F inversion} in Subsection~\ref{ssec F inversion}.
Then, using \eqref{F inversion}, we prove (i)--(v) of Theorem~\ref{thm:FI} 
in Subsections~\ref{ssec (i)}--\ref{ssec (v)}, respectively.
Let 
\begin{equation}\label{U[d]}
 \uU_{\beta,a,\lambda}^{[d]}(t)
 =
 2^{\beta}\Gamma(\beta+1)a^{2\beta}
 \int_0^{2\pi a\lambda}
 \frac{J_{{\frac d2}-1}(\frac ta s)J_{{\frac d2}+\beta}(s)}
      {\left(\frac ta\right)^{\frac d2-1}s^\beta}
 \,ds.
\end{equation}
Then \eqref{F inversion} means that
\begin{equation}\label{U=Ud}
 \sigma_{\lambda}(U_{\beta,a})(x)=\uU_{\beta,a,\lambda}^{[d]}(|x|).
\end{equation}

\subsection{Proof of \eqref{F inversion}}\label{ssec F inversion}
By \eqref{cD} the Fourier transform of $U_{\beta,a}(x)$ is expressed
by the following:
\begin{equation}\label{hat U}
 {\hat U}_{\beta,a}(\xi)
 = \Gamma(\beta+1)\cD_\beta(a^2:\xi)=
 \Gamma(\beta+1) \frac{a^{d+2\beta}}{\pi^{\beta}}
  \frac{J_{{\frac d2}+\beta}(2\pi a|\xi|)}{(a|\xi|)^{{\frac d2}+\beta}}.
\end{equation}
Since ${\hat U}_{\beta,a}(\xi)$ is a radial function,
using \eqref{FT-radial}, we have
\begin{align*}
 \sigma_{\lambda}(U_{\beta,a})(x) 
 &=
 \int_{|\xi|<\lambda}{\hat U}_{\beta,a}(\xi)e^{2\pi ix\xi}\,d\xi
\\
 &=2\pi \Gamma(\beta+1) \frac{a^{d+2\beta}}{\pi^{\beta}}
 \int_0^{\lambda}   \frac{J_{\frac d2-1}(2\pi s|x|)}{(s|x|)^{\frac d2-1}}\frac{J_{{\frac d2}+\beta}(2\pi as)}{(as)^{{\frac d2}+\beta}}
 s^{d-1}\,ds
 \\
 &=2^\beta \Gamma(\beta+1)a^{2\beta} 
\int_0^{2\pi a\lambda} \frac{J_{\frac d2-1}(\frac {|x|}a s)J_{\frac d2+\beta}(s) }{(\frac {|x|}a)^{\frac d2-1}{s^\beta}}\,ds.
\end{align*}
This shows \eqref{F inversion}.


\subsection{Proof of Theorem~\ref{thm:FI} (i)}\label{ssec (i)}

Since $\sigma(U_{\beta,a})(0)=\uU_{\beta,a,\lambda}^{[d]}(0)$ 
and $U_{\beta,a}(0)=\phi_{\beta,a}(0)=a^{2\beta}$, 
it is enough to prove that
\begin{equation}\label{U[d](1)}
 U_{\beta,a,\lambda}^{[d]}(0)
 =
 a^{2\beta}
 -P_{\beta,a}^{[d]}
  \cos\left(2\pi a\lambda-\frac{d-1+2\beta}4\pi\right)
  \lambda^{\frac{d-3}2-\beta}
 +O(\lambda^{\frac{d-5}2-\beta}).
\end{equation}
To do this we show the following two lemmas.
The first lemma is for the cases $d=1,2$.
For the cases $d\ge3$ we have the conclusion by using both lemmas.

\begin{lem}\label{lem:Pin d=12}
Let $\beta>-1$, $a>0$ and $\lambda>0$.
Then 
\begin{align*}
 U_{\beta,a,\lambda}^{[1]}(0)
 &=
 a^{2\beta}
 -P_{\beta,a}^{[1]}
  \cos\left(2\pi a\lambda-\frac{\beta}2\pi\right)
  \lambda^{-\beta-1}
  +O(\lambda^{-\beta-2})
 \quad \text{as} \ \lambda\to\infty.
\\
 U_{\beta,a,\lambda}^{[2]}(0)
 &=
 a^{2\beta}
 -P_{\beta,a}^{[2]}
  \cos\left(2\pi a\lambda-\frac{2\beta+1}4\pi\right)
  \lambda^{-\beta-\frac12}
  +O(\lambda^{-\beta-\frac 32})
 \quad \text{as} \ \lambda\to\infty.
\end{align*}
\end{lem}

\begin{lem}\label{lem:Pin d>2}
Let $d\ge3$, $\beta>-1$, $a>0$ and $\lambda>0$, and let
\begin{equation}\label{Pin}
 {\cP}_{\beta,a,\lambda}^{[d]}(t)
 =
 -\frac{\Gamma(\beta+1)a^{\beta}}{(\pi\lambda)^{\beta}}
 \sum_{\ell=1}^{\dsharp-1}
  \frac{J_{\frac d2-\ell+\beta}(2\pi a\lambda)
        J_{\frac d2-\ell}(2\pi t\lambda)}
       {\left(\frac ta\right)^{\frac d2-\ell}}, 
 \quad t\ge0,
\end{equation}
where $\dsharp$ is defined by \eqref{kd}.
Then
\begin{equation}\label{U[d] reduction}
 U_{\beta,a,\lambda}^{[d]}(t)
 =
\begin{cases}
 U_{\beta,a,\lambda}^{[1]}(t)+{\cP}_{\beta,a,\lambda}^{[d]}(t), 
 &\text{if $d$ is odd},\\
 U_{\beta,a,\lambda}^{[2]}(t)+{\cP}_{\beta,a,\lambda}^{[d]}(t), 
 &\text{if $d$ is even},
\end{cases}
 \quad t\ge0.
\end{equation}
Moreover, if $t\ne0$, 
then ${\cP}_{\beta,a,\lambda}^{[d]}(t)=O(\lambda^{-\beta-1})$ as $\lambda\to\infty$,
If $t=0$, then
\begin{equation}
 {\cP}_{\beta,a,\lambda}^{[d]}(0)
\\
 = -P_{\beta,a}^{[d]}
 \cos\left(2\pi a\lambda-\frac{d-1+2\beta}{4}\pi\right)
 \lambda^{\frac{d-3}2-\beta}
 +O(\lambda^{\frac{d-5}2-\beta})
 \quad \text{as} \ \lambda\to\infty.
\end{equation}
\end{lem}

\begin{rem}\label{Grafakos}
(i)
In Lemma~\ref{lem:Pin d=12}, if $\beta=-1/2$, then
\begin{equation}\label{Pin(d=2 and beta=-1/2)}
  U_{-1/2,a,\lambda}^{[2]}(0)
 =
 \frac{1-\cos\left(2\pi a\lambda\right)}{a}
  +O(\lambda^{-1})
 \quad \text{as} \ \lambda\to\infty.
\end{equation}
since $P_{-1/2,a}^{[2]}=a^{-1}$.
This 
shows the amplitude of Pinsky phenomenon
which was found by Taylor~\cite{Taylor-preprint}.

\noindent 
(ii)
Recently,
Grafakos and Teschl~\cite{Grafakos-Teschl2003} 
showed some related results with Lemma~\ref{lem:Pin d>2}.
\end{rem}

\begin{proof}[Proof of Lemma~\ref{lem:Pin d=12}]
Let $d=1$.
By \eqref{U[d]} with \eqref{Bessel-at-zero} and Corollary~\ref{cor:Bessel11} we have
\begin{equation*}
 U_{\beta,a,\lambda}^{[1]}(0)
 =
 2^\beta\Gamma(\beta+1)a^{2\beta} \sqrt{\dfrac 2\pi}
  \int_0^{2\pi a\lambda}
  \frac{J_{\frac12+\beta}(s)}{s^{\frac12+\beta}}\,ds
\end{equation*}
and
\begin{equation*}
 a^{2\beta}
 =
 2^\beta\Gamma(\beta+1)a^{2\beta} \sqrt{\dfrac 2\pi}
  \int_0^{\infty}
  \frac{J_{\frac12+\beta}(s)}{s^{\frac12+\beta}}\,ds,
\end{equation*}
respectively.
Then by Lemma~\ref{lem:Bessel1a}
we have
\begin{align*}
 U_{\beta,a,\lambda}^{[1]}(0)-a^{2\beta}
 &=
 -2^\beta\Gamma(\beta+1)a^{2\beta}\sqrt{\dfrac 2\pi}
 \int_{2\pi a\lambda}^\infty
 \frac{J_{\frac12+\beta}(s)}{s^{\frac12+\beta}}\,ds
\\
 &=
 -P_{\beta,a}^{[1]}
  \cos\left(2\pi a\lambda-\frac{\beta}2\pi\right)
  \lambda^{-\beta-1}
  +O(\lambda^{-\beta-2}).
\end{align*}
Let $d=2$.
By \eqref{U[d]} with \eqref{Bessel-at-zero} we have
\begin{align*}
 \uU_{\beta,a,\lambda}^{[2]}(0) 
 &=
 2^\beta\Gamma(\beta+1)a^{2\beta}
 \int_0^{2\pi a\lambda}
 \frac{J_{1+\beta}(s)}{s^\beta}
 \,ds
 \\
 &=
2^\beta\Gamma(\beta+1)a^{2\beta}
 \left[-\frac{J_{\beta}(s)}{s^\beta}\right]_0^{2\pi a\lambda}
\\
 &=
 a^{2\beta}-a^{2\beta}\Gamma(\beta+1)\frac{J_{\beta}(2\pi a\lambda)}{(\pi a\lambda)^\beta}
\\
 &=
 a^{2\beta}-P_{\beta,a}^{[2]}
 \cos\left(2\pi a\lambda-\frac{2\beta+1}{4}\pi\right)
 \lambda^{-\beta-1/2}
 +O(\lambda^{-\beta-3/2}).
\end{align*}
Then the proof is complete.
\end{proof}

\begin{proof}[Proof of Lemma~\ref{lem:Pin d>2}]
The equation \eqref{U[d] reduction} follows from
Collorary~\ref{cor:Bessel51}
and the definition of $\uU_{\beta,a,\lambda}^{[d]}(t)$ (see \eqref{U[d]}). 
If $t\ne0$, then from \eqref{Bessel-asymp-infty} it follows that
\begin{equation*}
 J_{\frac d2-\ell+\beta}(2\pi a\lambda)J_{\frac d2-\ell}(2\pi t\lambda)
 =O(\lambda^{-1}),
 \quad \ell=1,\dots,\dsharp-1.
\end{equation*}
Then we have $ {\cP}_{\beta,a,\lambda}^{[d]}(t)=O(\lambda^{-\beta-1})$.
If $t=0$, then by \eqref{Bessel-at-zero} we regard
\begin{equation*}
 \frac{J_{\frac d2-\ell+\beta}(2\pi a\lambda)
       J_{\frac d2-\ell}(2\pi t\lambda)}
      {\left(\frac ta\right)^{\frac d2-\ell}}
 =
 \frac{(\pi a\lambda)^{\frac d2-\ell}}
      {\Gamma({\frac d2-\ell}+1)}
 J_{\frac d2-\ell+\beta}(2\pi a\lambda),
\quad \ell=1,\dots,\dsharp-1.
\end{equation*}
Then, using \eqref{Bessel-asymp-infty}, we have
\begin{align*}
 &\cP_{\beta,a,\lambda}^{[d]}(0)
\\
 &=
 -\frac{\Gamma(\beta+1)a^{\beta}}{(\pi\lambda)^{\beta}}
 \sum_{\ell=1}^{\dsharp-1}
 \frac{(\pi a\lambda)^{\frac d2-\ell}}
      {\Gamma({\frac d2-\ell}+1)}
 J_{\frac d2-\ell+\beta}(2\pi a\lambda)
\\
 &=
 -\frac{\Gamma(\beta+1)a^{\beta}}{(\pi\lambda)^{\beta}}
\frac{(\pi a\lambda)^{\frac d2-1}}
     {\Gamma({\frac d2})}
 J_{\frac d2-1+\beta}(2\pi a\lambda)
 +O(\lambda^{\frac{d-5}2-\beta})
\\ 
 &=
 \left(
 -\frac{\Gamma(\beta+1)}{\Gamma({\frac d2})}
   a^{\frac{d-3}2+\beta} \pi^{\frac{d-4}2-\beta}
 \right)
 \lambda^{\frac{d-3}2-\beta}
\cos\left(2\pi a\lambda-\frac{d-1+2\beta}{4}\pi\right)
\\
 &\phantom{************************************}
 +O(\lambda^{\frac{d-5}2-\beta}).
\end{align*}
This is the conclusion.
\end{proof}

\subsection{Proof of Theorem~\ref{thm:FI} (ii)}\label{ssec (ii)}
For $x\in\tG_a$, by \eqref{U[d]} and \eqref{U=Ud} we have
\begin{equation*}
 \sigma_{\lambda}(U_{\beta,a})(x)
 =
 \uU_{\beta,a,\lambda}^{[d]}(a)
 =
 2^\beta\Gamma(\beta+1)a^{2\beta}
 \int_0^{2\pi a\lambda}
 {\frac{J_{{\frac d2}+\beta}(s)J_{{\frac d2}-1}(s)}{s^\beta}}
 \,ds.
\end{equation*}
Then, using Corollary~\ref{cor:Bessel44} with $\mu=\frac d2-1$,
we have
\begin{equation}\label{P(ii)}
 \uU_{\beta,a,\lambda}^{[d]}(a)
 =
 L_{\beta,a}\lambda^{-\beta}
 +
 \begin{cases}
  O(\lambda^{-\beta-1}), & \beta\ge0, \\
  O(1), & -1<\beta<0,
 \end{cases}
 \quad\text{as $\lambda\to\infty$},
\end{equation}
which shows the conclusion.

\subsection{Proof of Theorem~\ref{thm:FI} (iii)}\label{ssec (iii)}
Let $x\in\tE_a$ and $|x|=t$.
Then $0<t<a$ or $t>a$.
In this case, by \eqref{U[d]}, \eqref{U=Ud} and Corollary~\ref{cor:Bessel1}, 
we have
\begin{align*}
 \sigma_{\lambda}(U_{\beta,a})(x)-U_{\beta,a}(x) 
 &=
 \uU_{\beta,a,\lambda}^{[d]}(t)-\phi_{\beta,a}(t) \\ 
 &=
 -2^\beta\Gamma(\beta+1)a^{2\beta} 
 \int_{2\pi a\lambda}^\infty
 \frac{J_{{\frac d2}-1}(\frac ta s)J_{{\frac d2}+\beta}(s)}
      {\left(\frac ta\right)^{\frac d2-1}s^\beta}
 \,ds.
\end{align*}
Using Lemma~\ref{lem:Bessel3} with $\mu=\frac d2-1$,
we have
\begin{multline}\label{P(iii)}
 \uU_{\beta,a,\lambda}^{[d]}(t)-\phi_{\beta,a}(t)
\\
 =
 -\frac{2^\beta\Gamma(\beta+1)a^{\beta}}{\pi}
 \left(\frac at\right)^{\frac{d-1}2}
 \psi_{\beta}(\lambda,a-t)
  + 
   \left(\frac at\right)^{\frac{d-1}2}
   \left(1+{\frac at}\right)
   O(\lambda^{-\beta-1}).
\end{multline}
Since $\psi_{\beta}(\lambda,a-t)=|a-t|^{-1}O(\lambda^{-\beta-1})$
by Lemma~\ref{lem:Bessel3}~(i),
\begin{equation*}
 \uU_{\beta,a,\lambda}^{[d]}(t)-\phi_{\beta,a}(t)
 =
 O(\lambda^{-\beta-1})
\end{equation*}
uniformly on any closed interval in $(0,a)\cup(a,\infty)$.
This shows the conclusion.

\subsection{Proof of Theorem~\ref{thm:FI} (iv)}\label{ssec (iv)}
Let $\beta>0$.
By \eqref{P(ii)} we have
\begin{equation*}
 \uU_{\beta,a,\lambda}^{[d]}(a)-\phi_{\beta,a}(a)
 =
 \uU_{\beta,a,\lambda}^{[d]}(a)
 =
 L_{\beta,a}\lambda^{-\beta}+O(\lambda^{-\beta-1}).
\end{equation*}
If $0<t<a$ or $t>a$, then by \eqref{P(iii)} and Lemma~\ref{lem:psi}~(ii) we have
\begin{equation*}
 \uU_{\beta,a,\lambda}^{[d]}(t)-\phi_{\beta,a}(t)
 =
 \left(\frac at\right)^{\frac{d-1}2} O(\lambda^{-\beta})
   + \left(\frac at\right)^{\frac{d-1}2}
     \left(1+{\frac at}\right)O(\lambda^{-\beta-1}),
\end{equation*}
where the terms $O(\lambda^{-\beta})$ and $O(\lambda^{-\beta-1})$ 
are uniform with respect to $t$.
Hence,
$\uU_{\beta,a,\lambda}^{[d]}(t)$ converges to $\phi_{\beta,a}(t)$
uniformly on any closed interval in $(0,\infty)$.
That is, 
$\sigma_{\lambda}(U_{\beta,a})(x)$ converges to $U_{\beta,a}(x)$ 
uniformly on any compact subset of $\tE_a\cup\tG_a=\R^d\setminus\{0\}$,
and consequently, 
it does not reveal any phenomenon like the Gibbs-Wilbraham phenomenon.

\subsection{Proof of Theorem~\ref{thm:FI} (v)}\label{ssec (v)}
Let $-1<\beta\le0$.
The first assertion is already shown by (iii).
Next we show the Gibbs-Wilbraham phenomenon.
Here, we recall that
\begin{equation*}
 G_{\beta,a}^{\pm}
 =
 \mp\frac{\Gamma(\beta+1)a^{\beta}(2\pm\beta)^{\beta}}{\pi 2^\beta}
  \int_{\pi}^{\infty} \frac{\sin s}{(s\pm\frac\beta 2\pi)^{\beta+1}}\,ds.
\end{equation*}
By \eqref{P(iii)} we have
\begin{multline*}
 \sigma_{\lambda}(U_{\beta,a})(x_{\lambda}^+)-U_{\beta,a}(x_{\lambda}^+)
 =
 \uU_{\beta,a,\lambda}^{[d]}(|x_{\lambda}^+|)-\phi_{\beta,a}(|x_{\lambda}^+|)
\\
 =
 -\frac{2^\beta\Gamma(\beta+1)a^{\beta}}{\pi}
 \left(\frac{a}{|x_{\lambda}^+|}\right)^{\frac{d-1}2}
 \psi_{\beta}(\lambda,a-|x_{\lambda}^+|)
  + 
   \left(\frac{a}{|x_{\lambda}^+|}\right)^{\frac{d-1}2}
   \left(1+{\frac{a}{|x_{\lambda}^+|}}\right)
   O(\lambda^{-\beta-1}).
\end{multline*}
Since $a-|x_{\lambda}^+|=\frac{2+\beta}{4\lambda}$,
using Lemma~\ref{lem:psi} (iii),
we have
\begin{equation*}
 -\psi_{\beta}(\lambda,a-|x_{\lambda}^+|)
 =
 \left(\frac{2+\beta}{4\lambda}\right)^\beta
 \int_\pi^\infty\frac{-\sin s}{(s+\frac\beta 2 \pi)^{\beta+1}}ds.
\end{equation*}
Observing
\begin{equation*}
 \left(\frac{a}{|x_{\lambda}^+|}\right)^{\frac{d-1}2}
 =\left(1-\frac{2+\beta}{4\lambda a}\right)^{-\frac{d-1}2}=
 1+O(\lambda^{-1}),
\end{equation*}
we conclude that
\begin{align*}
 &\frac{\sigma_{\lambda}(U_{\beta,a})(x_{\lambda}^+)-U_{\beta,a}(x_{\lambda}^+)}
      {\lambda^{-\beta}}
\\
 &=
 -\frac{2^\beta\Gamma(\beta+1)a^{\beta}}{\pi}
   \left(\frac{2+\beta}{4}\right)^\beta
   \int_\pi^\infty\frac{\sin s}{(s+\frac\beta 2 \pi)^{\beta+1}}ds \,
   \big(1+O(\lambda^{-1})\big)
 +O(\lambda^{-1})
\\
 &\to
 G_{\beta,a}^{+}
 \quad\text{as}\quad \lambda\to\infty.
\end{align*}
In a similar way we also have
\begin{equation*}
 \frac{\sigma_{\lambda}(U_{\beta,a})(x_{\lambda}^-)-U_{\beta,a}(x_{\lambda}^-)}
      {\lambda^{-\beta}}
 \to
 G_{\beta,a}^{-}
 \quad\text{as}\quad \lambda\to\infty.
\end{equation*}


\section{Results related to lattice point problems}\label{sec:LPP}

The terms $\Delta_j(s:x)$, $j=0,1,\cdots$, are closely related to lattice point problems
which have been studied by
Landau, Jarn\'ik, Szeg\"o, Nov\'ak and others,
see \cite{Fricker1982, Ivic2006, Kratzel1988, Kratzel2000,Landau1927, Landau1962, Walfisz1957}.
Recall that 
\begin{equation*}
 \Delta_\alpha(s:x)
 =
 D_\alpha(s:x)-{\cD}_\alpha(s:x),
 \quad \alpha>-1,\ s\ge0,\ x\in\R^d,
\end{equation*}
where 
\begin{align*}
 D_{\alpha}(s:x)
 &=
 \begin{cases}\displaystyle
 {\frac1{\Gamma(\alpha+1)}}\sum_{|m|^2<s}(s-|m|^2)^\alpha e^{2\pi imx}, 
  & s>0,\\
 0, & s=0,
 \end{cases}
 \quad x\in\R^d,
\\
 {\cD}_{\alpha}(s:x)
 &=
 \begin{cases}\displaystyle
 {\frac1{\Gamma(\alpha+1)}}\int_{|\xi|^2<s}(s-|\xi|^2)^\alpha e^{2\pi i\xi x}\,d\xi,
  & s>0,\\
 0, & s=0,
 \end{cases}
 \quad x\in\R^d.
\end{align*}

In this section we consider the behavior of $\Delta_\alpha(s:x)$ as $s\to\infty$. In Stein~\cite{Stein1961}, 
the estimation of $\Delta_{\frac{d-1}2}(s:x)$ was treated.
We shall prove the following four lemmas:
In the following $f(s)=\Omega(g(s))$ means $f(s)\ne o(g(s))$.

\begin{lem}\label{lem:LPP1}
Let $d\ge1$.
Then, as $s\to\infty$,
\begin{equation}\label{LPP1}
 \Delta_{\alpha}(s:x)
 =
 \begin{cases}
 O(s^{\frac d2-\frac{d}{d+1}}), 
 & \text{if $\alpha=0$},\\
 O(s^{\frac d2-\frac{d}{d+1}+\frac{\alpha}{d+1}+\ve}) \ \text{for every $\ve>0$}, 
 & \text{if $0<\alpha\le\frac{d-1}2$},\\
 O(s^{\frac{d-1}4+\frac{\alpha}2}), 
 & \text{if } \alpha>\frac{d-1}2,
 \end{cases}
\end{equation}
uniformly with respected to $x\in \T^d$.
\end{lem}

\begin{lem}\label{lem:LPP2}
Let $d\ge5$.
If $0\le\alpha<(d-4)/2$,
then, as $s\to\infty$, 
\begin{equation}\label{LPP2a}
 \Delta_{\alpha}(s:x)
 =
 \begin{cases}
 O(s^{\frac d2-1}),\ \Omega(s^{\frac d2-1}) 
 & \text{for $x\in\T^d\cap\Q^d$},\\
 o(s^{\frac d2-1}) 
 & \text{for $x\in\T^d\setminus\Q^d$}.
 \end{cases}
\end{equation}
If $(d-4)/2\le\alpha\le(d-1)/2$,
then, for every $\varepsilon>0$, as $s\to\infty$,
\begin{equation}\label{LPP2b}
 \Delta_{\alpha}(s:x)
 =
 \begin{cases}
 O(s^{\frac{d-1+\alpha}3+\ve})
 & \text{for $x\in\T^d\cap\Q^d$},\\
 o(s^{\frac{d-1+\alpha}3+\ve}) 
 & \text{for $x\in\T^d\setminus\Q^d$}.
 \end{cases}
\end{equation}
\end{lem}

The following lemma gives more precise information 
on the estimate \eqref{LPP2a} for $x\in\T^d\cap\Q^d$.

\begin{lem}\label{lem:LPP2S}
Let $d\ge5$, $\alpha=0$ and $\beta>-1$.
For $k\in\N=\{1,2,\cdots\}$, let
\begin{align}\label{ell_k}
 \ell_k&=\ell_k(d,\beta)=\frac1{2a}\left(k+\frac{d+2\beta+1}4-\frac14\right),
\\
 m_k&=m_k(d,\beta)=\frac1{2a}\left(k+\frac{d+2\beta+1}4+\frac14\right). 
 \label{m_k}
\end{align}
Then, for large $k\in\N$, there exists $s_k\in[\ell_k^2, m_k^2]$ such that 
\begin{equation}\label{LPP2s}
 |\Delta_{0}(s_k:x)|
 \ge
 C(x)s_k^{\frac d2-1},
 \quad x\in\T^d\cap\Q^d,
\end{equation}
where $C(x)$ is a positive constant dependent on $x$, 
but independent of $s_k$ for large $k$.
\end{lem}

\begin{lem}\label{lem:LPP3}
Let $d\ge4$ and $0\le\alpha\le(d-1)/2$.
Then, for every $\ve>0$, as $s\to\infty$,
\begin{equation}\label{LPP3}
 \Delta_{\alpha}(s:x)
 =
 O(s^{\frac d4+\frac{d-2}{2(d-1)}\alpha+\ve})
 \quad\text{for a.e.\,$x\in\T^d$}.
\end{equation}
\end{lem}

To prove above four lemmas we state known results
(Theorems~\ref{thm:Landau}--\ref{thm:Riesz}):
See also Jarn\'ik \cite{Jarnik1969} for Theorem~\ref{thm:Landau}.
For $\alpha\ge0$, let 
\begin{equation}\label{P_alpha}
 P_{\alpha}(s:x)
 =
 D_{\alpha}(s:x)
 -
 \frac{\pi^{\frac d2}s^{\frac d2+\alpha}}
      {\Gamma(\frac d2+\alpha+1)}
 \delta(x),
 \quad x\in\R^n,\ s\ge0,
\end{equation}
where
$\delta(x)$ is the indicator function of $\Z^d$.

\begin{thm}[Landau~\cite{Landau1915,Landau1962}]\label{thm:Landau}
Let $d\ge2$. Then, for $x\in\R^d$,
\begin{equation}\label{Landau}
 P_{\alpha}(s:x)
 =
 \begin{cases}
 O(s^{\frac d2+\alpha-\frac{d}{d+1-2\alpha}}), 
  & \text{if } 0\le\alpha<\frac{d-1}2,\\
 O(s^{\frac{d-1}2}\log s), 
  & \text{if } \alpha=\frac{d-1}2,\\
 O(s^{\frac{d-1}4+\frac{\alpha}2}), 
  & \text{if } \alpha>\frac{d-1}2.
 \end{cases}
\end{equation}
\end{thm}


\begin{thm}[Nov\'ak~\cite{Novak1972}]\label{thm:Novak}
Let $d\ge5$, and let $0\le\alpha<(d-4)/2$.
Then
\begin{equation}\label{Novak}
 P_{\alpha}(s:x)
 =
 \begin{cases}
 O(s^{\frac d2-1}),\ \Omega(s^{\frac d2-1}) & \text{for $x\in\Q^d$},\\
 o(s^{\frac d2-1}) & \text{for $x\notin\Q^d$},\\
 O(s^{\frac{d}4+\frac{\alpha}2}\log^{\tau}s) & \text{for a.e.\,$x\in\R^d$},
 \end{cases}
\end{equation}
where $\tau=3d$ if $\alpha=0$ and $\tau=3d-1$ if $\alpha>0$.
\end{thm}

\begin{thm}[Nov\'ak~\cite{Novak1969}]\label{thm:Novak-M}
Let $d\ge3$.
Then, for all $x\in\Q^d$, 
there exists a positive constant $K_d(x)$ such that
\begin{equation}\label{Novak-M}
 \int_0^{s} \left|P_{0}(t:x)\right|^2\,dt
 =
 \begin{cases}
 K_d(x)s^{2}\log s+O(s^{2}\log^{1/2}s), & \text{if $d=3$},\\
 K_d(x)s^{3}+O(s^{5/2}\log s), & \text{if $d=4$},\\
 K_d(x)s^{4}+O(s^{3}\log^2s), & \text{if $d=5$},\\
 K_d(x)s^{d-1}+O(s^{d-2}), & \text{if $d\ge6$}.
 \end{cases}
\end{equation}
\end{thm}

\begin{rem}\label{rem:Novak-M}
In Theorem~\ref{thm:Novak-M} 
the positive constant $K_d(x)$ is given explicitly for each $x\in\Q^d$, 
see \cite{Kuratsubo2010}.
For example, if $d\ge 4$ and $x=0$, then
\begin{equation*}
 K_d(0)=\frac{\pi^d(2^d+8)\zeta(d-2)}{12(d-1)(2^d-1)\zeta(d)\Gamma^2(d/2)},
\end{equation*}
where $\zeta$ is the Riemann's zeta function, see \cite{Walfisz1957}.
\end{rem}

\begin{thm}[Kuratsubo~\cite{Kuratsubo1982}]\label{thm:Kura}
Let $d\ge2$ . 
Then, for every $\tau>3/2$,
\begin{equation*}
 P_{0}(s:x)
 =
 O(s^{\frac d4}\log^\tau s)
 \quad\text{for a.e.\,$x$}. 
\end{equation*}
\end{thm}

\begin{rem}\label{rem:Delta-P}
Theorems~\ref{thm:Landau}--\ref{thm:Kura} valid
for $\Delta_{\alpha}(s:x)$ instead of $P_{\alpha}(s:x)$, if $x\in\T^d$.
Actually,
\begin{equation*}
 \Delta_{\alpha}(s:x)-P_{\alpha}(s:x)
 =
 \begin{cases}
 0, & \text{if $x=0$}, \\
 -\cD_{\alpha}(s:x)=O(s^{\frac{d-1}4+\frac{\alpha}2}), 
 & \text{if $x\in\T^d\setminus\{0\}$},
 \end{cases}
\end{equation*}
since (see \eqref{cD} and \eqref{Bessel-asymp-infty}) 
\begin{equation*}
 \cD_{\alpha}(s,x)
 =
 \begin{cases}\displaystyle
 \frac{\pi^{\frac d2}s^{\frac d2+\alpha}}{\Gamma(\frac d2+\alpha+1)}, & x=0, \\[2ex]
 \displaystyle
 \frac{s^{\frac d2+\alpha}}{\pi^{\alpha}}
 \frac{J_{{\frac d2}+\alpha}(2\pi\sqrt{s}|x|)}{(\sqrt{s}|x|)^{{\frac d2}+\alpha}}
 =
 O(s^{\frac{d-1}4+\frac{\alpha}2}),
 & x\ne0.
 \end{cases}
\end{equation*}
Therefore, 
\begin{equation*}
 |\Delta_{\alpha}(s:x)|
 =
 |P_{\alpha}(s:x)|+O(s^{\frac{d-1}4+\frac{\alpha}2}),
 \quad\text{if $x\in\T^d$}.
\end{equation*}
\end{rem}

\begin{rem}\label{rem:Landau}
Let $d=1$. 
Then 
\begin{equation}\label{Landau 1}
 \Delta_{\alpha}(s:x)
 =
 \begin{cases}
 O(1), & \text{if } \alpha=0,\\
 O(s^{\frac{\alpha}2}), & \text{if } \alpha>0,
 \end{cases}
\end{equation}
uniformly with respect to $x\in\T$.
Actually,
for all $s>0$, choosing $N\in\N$ such that $N<\sqrt{s}\le N+1$,
we have by an elementary calculation
\begin{equation*}
 \Delta_{0}(s:x)
 =
 \frac{\sin(2\pi(N+\frac12)x)}{\sin\pi x}
 -\frac{\sin(2\pi\sqrt{s}x)}{\pi x}
 =O(1).
\end{equation*}
For $\alpha>0$, from \eqref{cD} and \eqref{D} it follows that
\begin{equation*}
 \Delta_{\alpha}(s:x)
 =
 \frac{s^{\frac12+\alpha}}{\pi^{\alpha}}
 \sum_{m\in\Z,\ m\ne0}
 \frac{J_{{\frac12}+\alpha}(2\pi\sqrt{s}|x-m|)}{(\sqrt{s}|x-m|)^{{\frac12}+\alpha}}. 
\end{equation*}
By \eqref{Bessel-asymp-infty} 
we have 
\begin{equation*}
 \frac{|J_{{\frac12}+\alpha}(2\pi\sqrt{s}|x-m|)|}{(\sqrt{s}|x-m|)^{{\frac12}+\alpha}}
 \le
 \frac{C}{(\sqrt{s}|x-m|)^{1+\alpha}},
\end{equation*}
for some positive constant $C$.
Since $|x-m|\ge1/2$ for $x\in\T$ and $m\ne0$,
the sum converges absolutely 
and $\Delta_{\alpha}(s:x)=O(s^{\frac{\alpha}2})$. 
\end{rem}

The following is the Riesz' convexity theorem.
(see \cite[page 285]{Stein-Weiss1971} and \cite[page 13]{CM1952}).

\begin{thm}\label{thm:Riesz}
Let $0\le\alpha_0<\alpha_1<\infty$.
For $x\in\R^d$, let $V_0(s:x)$ and $V_1(s:x)$ be two positive nondecreasing functions with respect to $s>0$.
Assume that
\begin{equation*}
 |\Delta_{\alpha_i}(s:x)|\le V_i(s:x), \quad i=0,1.
\end{equation*}
Then, for $0\le\theta\le1$,
\begin{equation*}
 |\Delta_{(1-\theta)\alpha_0+\theta\alpha_1}(s:x)|
 \le 
 C 
 V_0(s:x)^{1-\theta}V_1(s:x)^{\theta},
\end{equation*}
where $C$ is a positive constant 
dependent on $\alpha_1$, $\alpha_0$ and $\theta$,
and independent of $s$ and $x$.
\end{thm}

\begin{proof}[Proof of Lemma~\ref{lem:LPP1}]\label{proof:LPP1}
The case $d=1$ has been already proven in Remark~\ref{rem:Landau}.
Then we consider the case $d\ge2$.
In general, for a function $f:[0,\infty)\to\R$, 
the difference of $f$ with $h>0$ is defined by
\begin{equation*}
 \delta_h f(s)=\delta_h^1 f(s)=f(s+h)-f(s),
\end{equation*}
and\
\begin{equation*}
 \delta_h^{k+1} f(s)=\delta_h^k f(s+h)-\delta_h^k f(s), \quad k=1,2,\cdots.
\end{equation*}
Then 
\begin{equation*}
 \delta_h^k f(s)=\int_s^{s+h} ds_1\cdots\int_{s_{k-1}}^{s_{k-1}+h}f^{(k)}(s_k)\,ds_k.
\end{equation*}

Let $k=\dsharp$ as in \eqref{kd} 
and use this relation for $f(s)=\Delta_k(s:x)$.
Then, from \eqref{D-int} and \eqref{cD-dif} we see that
\begin{align*}
 &\delta_h^k \Delta_k(s:x)-h^k \Delta_0(s:x)
 \\
 &=
 \int_s^{s+h}ds_1\cdots\int_{s_{k-1}}^{s_{k-1}+h} (\Delta_0(s_k:x)-\Delta_0(s:x))\,ds_k
 \\
 &=
 \int_s^{s+h}ds_1\cdots\int_{s_{k-1}}^{s_{k-1}+h}
 \bigg((D_0(s_k:x)-D_0(s:x))-({\cD}_0(s_k:x)-{\cD}_0(s:x))\bigg)\,ds_k
 \\
 &=
 \int_s^{s+h}ds_1\cdots\int_{s_{k-1}}^{s_{k-1}+h}
 \left(\sum_{s\le |m|^2<s_k}e^{2\pi imx}-\int_{s\le|\xi|^2<s_k}e^{2\pi i\xi x}\right)
 \,ds_k.
\end{align*}
Using $s_k\le s+kh$ inside the integration,
we have
\begin{align*}
 &\left|\sum_{s\le |m|^2<s_k}e^{2\pi imx}-\int_{s\le|\xi|^2<s_k}e^{2\pi i\xi x}\right|
\\
 &\le
 \sum_{s\le |m|^2<s_k}1+\int_{s\le|\xi|^2<s_k}1
\\
 &\le
 \sum_{s\le |m|^2<s+kh}1+\int_{s\le|\xi|^2<s+kh}1
\\
 &=
 D_0(s+kh:0)-D_0(s:0)
  +v_d\left((s+kh)^{\frac d2}-s^{\frac d2}\right)
\\
 &\le
 |\Delta_0(s+kh:0)|+|\Delta_0(s:0)|
  +2v_d\left((s+kh)^{\frac d2}-s^{\frac d2}\right)
\\
 &\le
 |\Delta_0(s+kh:0)|+|\Delta_0(s:0)|
  +dkh v_d(s+kh)^{\frac d2-1},
\end{align*}
where $v_d$ is the volume of the $d$-dimensional unit ball.
Hence
\begin{align*}
 &\left| \delta_h^k \Delta_k(s:x) - h^k\Delta_0(s:x) \right|
\\
 &\le
 \int_s^{s+h}ds_1\cdots\int_{s_{k-1}}^{s_{k-1}+h}
 \left|\sum_{s\le |m|^2<s_k}e^{2\pi imx}-\int_{s\le|\xi|^2<s_k}e^{2\pi i\xi x}\right|
 \,ds_k
\\
 &\le
 h^k\left(|\Delta_0(s+kh:0)|+|\Delta_0(s:0)|+dkh v_d(s+kh)^{\frac d2-1}\right).
\end{align*}
By Theorem~\ref{thm:Landau} with $\alpha=0$ and $x=0$, we have
\begin{equation}\label{Delta-1}
 \left| \delta_h^k \Delta_k(s:x) - h^k\Delta_0(s:x) \right|
 \le
 C h^k\left(s^{\frac d2-\frac d{d+1}} + hs^{\frac d2-1}\right),
 \quad\text{if $kh\le s$},
\end{equation}
where $C$ is a positive constant independent of $h$, $s$ and $x$.

Next we estimate $\delta_h^k\Delta_k(s:x)$.
By \eqref{cD} and \eqref{D} we have
\begin{align*}
 \left|\delta_h^k\Delta_k(s:x)\right|
 &=
 \left|
 \delta_h^k
 \left(
  \sum_{m\in\Z^d,\,m\ne0}
   \frac{s^{\frac 12(\frac d2+k)}}{\pi^k}
   \frac{J_{\frac d2+k}(2\pi\sqrt{s}|x-m|)}{|x-m|^{\frac d2+k}}
 \right)
 \right|
\\
 &=
 \left|
  \sum_{m\in\Z^d,\,m\ne0}
 \frac{\delta_h^k
        \left(s^{\frac 12(\frac d2+k)}J_{\frac d2+k}(2\pi\sqrt{s}|x-m|)\right)}
      {\pi^k|x-m|^{\frac d2+k}}
 \right|.
\end{align*}
Here, using a well known inequality on the Bessel function 
(see the equation (37) in \cite[page 472]{Landau1915});
\begin{equation*}
 \left|
 \delta_h^k\left(s^{\frac 12(\frac d2+k)}J_{\frac d2+k}(2\pi\sqrt{st})\right)
 \right|
 \le
 C \frac{s^{\frac{d-1}4}}{t^{\frac 14}} (\min(s,h^2t))^{\frac k2},
\end{equation*}
where $C$ is a positive constant independent of $s$, $t$ and $h$,
we have
\begin{align*}
 \left|\delta_h^k\Delta_k(s:x)\right|
 &\le C
 \left|
  \sum_{m\in\Z^d,\,m\ne0}
   \frac{s^{\frac{d-1}4}(\min(s,h^2|x-m|^2))^{\frac k2}}{|x-m|^{\frac12}}
   \frac{1}{|x-m|^{\frac d2+k}}
 \right|
\\
 &=C
 \left(
 \sum_{|x-m|^2\le{s}/{h^2},\,m\ne0}
 \frac{h^ks^{\frac{d-1}4}}
      {|x-m|^{\frac{d+1}2}}
 +\sum_{|x-m|^2>{s}/{h^2}}
 \frac{s^{\frac{d-1}4+\frac k2}}
      {|x-m|^{\frac{d+1}2+k}}
 \right)
\\
 &\le C
 \left(
  h^ks^{\frac{d-1}4}\left(\frac{s}{h^2}\right)^{\frac{d-1}4}
  + s^{\frac{d-1}4+\frac k2}\left(\frac{s}{h^2}\right)^{\frac{d-1}4-\frac k2}
 \right)
 =C h^k\left(\frac{s}{h}\right)^{\frac{d-1}2}.
\end{align*}
That is
\begin{equation}\label{Delta-2}
 \left|\delta_h^k\Delta_k(s:x)\right|
 \le
 C h^k\left(\frac{s}{h}\right)^{\frac{d-1}2}.
\end{equation}

Therefore, combining \eqref{Delta-1} and \eqref{Delta-2}, we have
\begin{align*}
 |\Delta_0(s:x)|
 &\le
 h^{-k}
 \left(
 \left|\delta_h^k\Delta_k(s:x)-h^k\Delta_0(s:x)\right|
 +\left|\delta_h^k\Delta_k(s:x)\right|
 \right)
\\
 &\le
 C \left(s^{\frac d2-\frac d{d+1}} + hs^{\frac d2-1}\right)
 +C \left(\frac{s}{h}\right)^{\frac{d-1}2},
 \quad\text{if $kh\le s$}.
\end{align*}
Then, setting $h$ as $hs^{\frac d2-1}=(s/h)^\frac{d-1}2$, 
that is, $h=s^\frac1{d+1}$,
we have that
\begin{equation*}
 |\Delta_0(s:x)|\le C s^{\frac d2-\frac{d}{d+1}},
 \quad\text{if $s\ge k^{\frac{d+1}d}$},
\end{equation*} 
where $C$ is a positive constant independent of $x\in\T^d$ and $s\ge k^{\frac{d+1}d}$.

If $\alpha>(d-1)/2$, 
then we have by \eqref{cD} and \eqref{D} that
\begin{equation*}
 \Delta_{\alpha}(s:x)
 =
 \frac{s^{\frac d2+\alpha}}{\pi^{\alpha}}
 \sum_{m\in\Z^d,\,m\ne0}
  \frac{J_{\frac d2+\alpha}(2\pi\sqrt{s}|x-m|)}{(\sqrt{s}|x-m|)^{\frac d2+\alpha}}.
\end{equation*}
Moreover, since $|x-m|\ge1/2$ for $x\in\T^d$ and $m\ne0$,
the sum converges absolutely and 
\begin{equation*}
 |\Delta_{\alpha}(s:x)|
 \le C
 s^{\frac{d-1}4+\frac{\alpha}2},
\end{equation*}
where $C$ is a positive constant independent of $x\in\T^d$ and $s>0$.
Therefore, we have 
\begin{equation*}
 \Delta_{\alpha}(s:x)
 =
 \begin{cases}
 O(s^{\frac d2-\frac{d}{d+1}}), & \text{if } \alpha=0,\\
 O(s^{\frac{d-1}4+\frac{\alpha}2}), & \text{if } \alpha>\frac{d-1}2,
 \end{cases}
\end{equation*}
uniformly with respect to $x\in\T^d$.

Applying Theorem~\ref{thm:Riesz}
as 
\begin{equation*}
 \alpha_0=0,\ \alpha_1=\frac{d-1}2\ \
 \text{and}\ \ \alpha=\theta\alpha_1,
\end{equation*}
we have 
\begin{equation*}
 \Delta_{\alpha}(s:x)
 =
 O(s^{\frac d2-\frac{d}{d+1}+\frac{\alpha}{d+1}+\ve})
 \ \text{for every $\ve>0$, if } 0<\alpha\le\frac{d-1}2,
\end{equation*}
since
\begin{equation*}
 (1-\theta)\left(\frac d2-\frac{d}{d+1}\right)
  +\theta\left(\frac{d-1}4+\frac{\alpha_1}2\right) 
 =\frac d2-\frac{d}{d+1}+\frac{\alpha}{d+1}. 
\end{equation*}
Then the proof is complete.
\end{proof}

\begin{rem}\label{rem:LPP1}
We have the following comparison between Lemma~\ref{lem:LPP1}
and Landau's estimate \eqref{Landau}:
\begin{equation*}
 \frac d2-\frac{d}{d+1}+\frac{\alpha}{d+1}
 <
 \frac d2+\alpha-\frac{d}{d+1-2\alpha},
 \ \text{if}\ 0<\alpha<\frac{d-1}2.
\end{equation*}
\end{rem}

\begin{proof}[Proof of Lemma~\ref{lem:LPP2}]\label{proof:LPP2}
If $0\le\alpha<(d-4)/2$,
then, 
by Theorem~\ref{thm:Novak} and Remark~\ref{rem:Delta-P}
we have \eqref{LPP2a} immediately.
Next we show \eqref{LPP2b}. 
By Theorems~\ref{thm:Landau}, \ref{thm:Novak} and Remark~\ref{rem:Delta-P}
we have
\begin{equation*}
 \Delta_{\alpha}(s:x)
 =
 \begin{cases}
 O(s^{\frac d2-1})
 & \text{for $x\in\T^d\cap\Q^d$},\\
 o(s^{\frac d2-1}) 
 & \text{for $x\in\T^d\setminus\Q^d$},
 \end{cases}
 \quad\text{if $\alpha<\frac{d-4}2$},
\end{equation*}
and
\begin{equation*}
 \Delta_{\alpha}(s:x)
 =
 O(s^{\frac{d-1}2}\log s)
 \ \text{for $x\in\T^d$}, \quad\text{if $\alpha=\frac{d-1}2$}.
\end{equation*}
Applying Theorem~\ref{thm:Riesz} 
as 
\begin{equation*}
 \alpha_0=\frac{d-4}2,\ \alpha_1=\frac{d-1}2\ \
 \text{and}\ \ 
 \alpha=(1-\theta)\alpha_0+\theta\alpha_1,
\end{equation*}
we have 
\begin{equation*}
 \Delta_{\alpha}(s:x)
 =
 O(s^{\frac{d-1+\alpha}3+\ve})
 \ \text{for every $\ve>0$, if } \frac{d-4}2\le\alpha\le\frac{d-1}2,
\end{equation*}
since
\begin{equation*}
 (1-\theta)\left(\frac d2-1\right)+\theta\,\left(\frac{d-1}2\right)
 =\frac{d-1+\alpha}3.
 \qedhere
\end{equation*}
\end{proof}

\begin{proof}[Proof of Lemma~\ref{lem:LPP2S}]\label{proof:LPP2S}
By Theorem~\ref{thm:Novak-M} and Remark~\ref{rem:Delta-P} we have
\begin{equation*}
 \int_0^{s}|\Delta_{0}(t:x)|^2\,dt
 =
 K_d(x)s^{d-1}+O(s^{d-2}\log^{\tau}s),
\end{equation*}
where 
$\tau=2$ if $d=5$, 
$\tau=0$ if $d\ge6$.
Using
\begin{align*}
 m_k^2-\ell_k^2
 &=
 \frac1{(2a)^2}\left(k+\frac{d+2\beta+1}4\right),
\\
 m_k^{2(d-1)}-{\ell}_k^{2(d-1)}
 &=
 \frac{d-1}{(2a)^{2(d-1)}}k^{2d-3}+O(k^{2d-4}),
\end{align*}
we have
\begin{align*}
 \frac1{m_k^2-\ell_k^2}
 \int_{\ell_k^2}^{m_k^2}|\Delta_{0}(t:x)|^2\,dt
 &=
 \frac1{m_k^2-\ell_k^2}
 \left(
 \int_{0}^{m_k^2}|\Delta_{0}(t:x)|^2\,dt
 -\int_{0}^{\ell_k^2}|\Delta_{0}(t:x)|^2\,dt
 \right)
\\
 &=
 K_d(x)\frac{m_k^{2(d-1)}-{\ell}_k^{2(d-1)}}{m_k^2-\ell_k^2}
  +O(k^{2(d-2)-1}\log^{\tau}k)
\\
 &=
 \frac{d-1}{(2a)^{2(d-2)}}K_d(x)\,k^{2d-4}
  +O(k^{2(d-2)-1}\log^{\tau}k).
\end{align*}
Hence we can take $\tilde{K}_d(x)$ such that
\begin{equation*}
 \frac1{m_k^2-\ell_k^2}
 \int_{\ell_k^2}^{m_k^2}|\Delta_{0}(t:x)|^2\,dt
 \ge
 \tilde{K}_d(x)^2m_k^{2d-4}
 \quad\text{for large $k$}.
\end{equation*}
Therefore, for large $k$, there exists $s_k\in[\ell_k^2,m_k^2]$ such that
\begin{equation*}
 |\Delta_{0}(s_k:x)|
 \ge
 \tilde{K}_d(x)m_k^{d-2}
 \ge
 \tilde{K}_d(x)s_k^{\frac d2-1}. \qedhere
\end{equation*}
\end{proof}

\begin{proof}[Proof of Lemma~\ref{lem:LPP3}]\label{proof:LPP3}
By Theorems~\ref{thm:Landau}, \ref{thm:Kura} and Remark~\ref{rem:Delta-P}
we have 
\begin{equation*}
 \Delta_{\alpha}(s,x)
 =
 \begin{cases}
 O(s^{\frac d4}\log^\tau s), & \text{if } \alpha=0,\\
 O(s^{\frac{d-1}2}\log s), & \text{if } \alpha=\frac{d-1}2,
 \end{cases}
 \quad\text{for a.e.\,$x\in\T^d$}. 
\end{equation*}
Applying Theorem~\ref{thm:Riesz} 
as 
\begin{equation*}
 \alpha_0=0,\ \alpha_1=\frac{d-1}2\ \
 \text{and}\ \ \alpha=\theta\alpha_1, 
\end{equation*}
we have 
\begin{equation*}
 \Delta_{\alpha}(s,x)
 =
 O(s^{\frac d4+\frac{d-2}{2(d-1)}\alpha+\ve})
 \ \text{for every $\ve>0$, if } 0<\alpha<\frac{d-1}2,
\end{equation*}
since
\begin{equation*}
 (1-\theta)\frac d4+\theta\frac{d-1}2
 =
 \frac d4+\frac{d-2}{2(d-1)}\alpha. 
 \qedhere
\end{equation*}
\end{proof}

\begin{rem}\label{rem:LPP3}
We have the following comparison between Lemma~\ref{lem:LPP1} and Lemma~\ref{lem:LPP3}:
\begin{equation*}
 \frac d4+\frac{d-2}{2(d-1)}\alpha
 <
 \frac d2-\frac{d}{d+1}+\frac{\alpha}{d+1},
 \ \text{if $d\ge4$ and $0\le\alpha<\frac{d-1}2$}.
\end{equation*}
\end{rem}

\section{Proof of the main theorems}\label{sec:proof}

In this section we prove 
the pointwise convergence 
of the Fourier series of the function $u_{\beta,a}(x)$
described in the main theorems (Theorems~\ref{thm:PinskyPh}--\ref{thm:AEC}).
First we state a generalized Hardy's identity (Theorem~\ref{thm:RT}) 
and three lemmas 
(Lemmas~\ref{lem:K 1-4}--\ref{lem:K a.e.}).
Next, 
using the generalized Hardy's identity and three lemmas, 
we will prove the main theorems
in Subsection~\ref{ss:proof}.
The proofs of Theorem~\ref{thm:RT} 
and Lemmas~\ref{lem:K 1-4}--\ref{lem:K a.e.}
are in 
Subsections~\ref{ss:proof-thm:RT},
\ref{ss:proof-lem:K 1-4},
\ref{ss:proof-lem:K 5-} and
\ref{ss:proof-lem:K a.e.}, respectively.

\subsection{Generalized Hardy's identity and three lemmas}\label{ss:GHI}

Recall that 
$u_{\beta,a}(x)$ is the periodization 
of $U_{\beta,a}(x)=\phi_{\beta,a}(|x|)$ with
\begin{equation*}
 \phi_{\beta,a}(t)
 =
 \begin{cases}
  (a^2-t^2)^{\beta}, & 0\le t<a, \\
  0, & t\ge a.
 \end{cases}
\end{equation*}
That is, 
\begin{equation}\label{uba}
 u_{\beta,a}(x)
 =
 \sum_{m\in{\Z}^d}U_{\beta,a}(x+m)=\sum_{|x+m|<a}(a^2-|x+m|^2)^\beta
,
 \quad x\in\T^d.
\end{equation}
Let $\Delta_j$, $A_{\beta,a}^{(j)}$, $\dsharp$ and $\psi_{\beta}$ be as 
in \eqref{Delta}, \eqref{A_beta,a}, \eqref{kd} and \eqref{psi},
respectively,
and let
\begin{equation}\label{K}
 \cK_{\beta,a}(s:x)
 =
 \sum_{j=0}^{\dsharp}(-1)^j\Delta_j(s:x)A_{\beta,a}^{(j)}(s),
\end{equation}
and
\begin{equation}\label{Gx}
 G_{\beta,a}(\lambda,x)
 =
 \frac{\Gamma(\beta+1)(2a)^\beta}\pi
 \sum_{m\in\Z^d\setminus\{0\},\ |x-m|\ne a}
 \left(\frac{a}{|x-m|}\right)^{\frac{d+1}2+\dsharp}
 \psi_{\beta}(\lambda,|a-|x-m||).
\end{equation}
First we show the convergence of the infinite sum in \eqref{Gx}.
If $x\in\T^d$ and $m\in\Z^d\setminus\{0\}$,
then $|x-m|>1/2$ and then
\begin{equation}\label{sup sum}
 \sup_{x\in\T^d}
 \left\{
 \sum_{m\in\Z^d\setminus\{0\}}
  \frac1{|x-m|^{\frac{d+1}2+\dsharp}}
 \right\}
 <\infty,
\end{equation}
since $\frac{d+1}2+\dsharp>d$.
By Lemma~\ref{lem:psi}~(i) we have that
\begin{equation*}
 \psi_{\beta}(\lambda,|a-|x-m||)
 =
 |a-|x-m||^{-1}
 O(\lambda^{-\beta-1})
 \quad\text{as}\quad \lambda\to\infty.
\end{equation*}
Letting
\begin{equation*} 
 M_a(x)=\max_{m\in\Z^d\setminus\{0\},\ |x-m|\ne a} |a-|x-m||^{-1},
\end{equation*}
we have
\begin{multline}\label{G converge}
 \sum_{m\in\Z^d\setminus\{0\},\ |x-m|\ne a}
  \left(\frac{a}{|x-m|}\right)^{\frac{d+1}2+\dsharp}
  \big|\psi_{\beta}(\lambda,|a-|x-m||)\big| \\
 \le
 C
 \left\{
 \sum_{m\in\Z^d\setminus\{0\}}
  \frac1{|x-m|^{\frac{d+1}2+\dsharp}}
 \right\}
 M_a(x)\lambda^{-\beta-1}<\infty.
\end{multline}
Hence, 
\begin{equation}\label{G -b-1}
 G_{\beta,a}(\lambda,x)
 =
 M_a(x)
 O(\lambda^{-\beta-1})
 \quad\text{as}\quad \lambda\to\infty.
\end{equation}
Next, recall that
\begin{equation*}
 r_d(a:x)=\sum_{m\in\Z^d, \ |x-m|=a}1,
 \qquad
 x\in\T^d,
\end{equation*}
and let
\begin{equation}\label{tr}
 \tilde{r}_d(a:x)=\sum_{m\in\Z^d\setminus\{0\},\ |x-m|=a}1,
 \qquad x\in\T^d.
\end{equation}
Then
\begin{equation}\label{rd trd}
 \tilde{r}_d(a:0)
 =
 r_d(a:0)
 \quad\text{and}\quad
 \tilde{r}_d(a:x)
 =
 r_d(a:x)=0
 \ \text{for} \ 
 x\in E_a.
\end{equation}

\begin{thm}[Generalized Hardy's identity]\label{thm:RT}
Let $d\ge 1$, $\beta>-1$ and $a>0$.
Then 
\begin{multline}\label{GHI}
 S_\lambda(u_{\beta,a})(x)
 =
 u_{\beta,a}(x)
 +\big(\sigma_\lambda(U_{\beta,a})(x)-U_{\beta,a}(x)\big)
 \\
 +\bigg(G_{\beta,a}(\lambda:x)
 +\tilde{r}_d(a:x)\Big(L_{\beta,a}+o(1)\Big)\lambda^{-\beta}\bigg)
\\
 + \cK_{\beta,a}(\lambda^2:x) 
  +O(\lambda^{-\beta-1})
 \quad\text{as}\quad \lambda\to\infty
\end{multline}
for all $x\in\T^d$.
\end{thm}

If $0<a<1/2$,
then $\tilde{r}_d(a:x)=0$ and
$u_{\beta,a}(x)=U_{\beta,a}(x)$
in $x\in\T^d$.
Combining these and \eqref{G -b-1},
we have the following corollary
of Theorem~\ref{thm:RT}.

\begin{cor}\label{cor:GHIa<1/2}
Let $d\ge 1, \beta>-1$ and $0<a<1/2$. 
Then
\begin{equation}\label{GHIa<1/2}
 S_\lambda(u_{\beta,a})(x)
 =
 \sigma_\lambda (U_{\beta,a})(x)
 + \cK_{\beta,a}(\lambda^2:x) 
  +O(\lambda^{-\beta-1})
 \quad\text{as}\quad \lambda\to\infty
\end{equation}
for all $x\in\T^d$.
\end{cor}

If $d=2$, $\beta=0$ and $a>0$, 
then
$U_{0,a}(x)$ is the indicator function 
of the ball $\{x\in\R^2: |x|<a\}$.
In this case,
from Theorem~\ref{thm:FI} (i)~(a), (ii) and (iii)
it follows that
\begin{equation*}
 \lim_{\lambda\to\infty}\big(\sigma_{\lambda}(U_{0,a})(x)-U_{0,a}(x)\big)
 =
 \begin{cases}
  L_{0,a}=\dfrac12, & |x|=a, \\[1ex]
  0, & |x|\ne a,
 \end{cases}
 \quad
 x\in\R^2.
\end{equation*}
Then we conclude that
\begin{equation*}
 \lim_{\lambda\to\infty}\big(\sigma_{\lambda}(U_{0,a})(x)-U_{0,a}(x)\big)
 +\tilde{r}_2(a:x)L_{0,a}
 =
 \frac12 r_2(a:x)
 =
 \frac12 \sum_{|x+m|=a}1,
 \quad
 x\in\T^2.
\end{equation*}
By \eqref{G -b-1} and Lemma~\ref{lem:K 1-4} bellow
we also have
\begin{equation*}
 G_{0,a}(\lambda:x)+\cK_{0,a}(\lambda^2:x)\to0
 \quad\text{as} \ \lambda\to\infty,
 \quad
 x\in\T^2.
\end{equation*}
Since
$u_{0,a}(x)=\sum_{|x+m|<a}1$ (see \eqref{uba}),
we have the following corollary.

\begin{cor}[{\cite[Theorem]{Kuratsubo1996} and \cite[page 446]{Brandolini-Colzani1999}}]\label{cor:HI x}
Let $d=2$, $\beta=0$ and $a>0$.
Then
\begin{equation*}
 \lim_{\lambda\to\infty} S_{\lambda}(u_{0,a})(x)
 =
 \sum_{|x+m|<a}1 + \frac12 \sum_{|x+m|=a}1,
\end{equation*}
for all $x\in\T^2$.
\end{cor}

\begin{rem}\label{rem:Hardy}
Let $d=2$, $\beta=0$ and $a>0$.
By \eqref{Poisson} and \eqref{hat U} together with \eqref{Bessel-at-zero} 
we have
$\hat u_{0,a}(0)=\pi a^2$ and
$\hat u_{0,a}(m)=a{J_1(2\pi a|m|)}/{|m|}$ for $m\ne0$.
Then
\begin{equation*}
 S_\lambda(u_{0,a})(x)
 =
 \pi a^2 
 + a\sum_{0<|m|<\lambda} \frac{J_1(2\pi a|m|)}{|m|} e^{2\pi imx}.
\end{equation*}
In particular, if $x=0$, then
Corollary~\ref{cor:HI x} shows Hardy's identity that
\begin{equation}\label{HI}
 \lim_{\lambda\to\infty}
 \left( \pi a^2 + a\sum_{0<|m|<\lambda} 
       \frac{J_1(2\pi a|m|)}{|m|} \right)
 =
 \sum_{|m|<a}1 + \frac12 \sum_{|m|=a}1,
\end{equation}
which was studied by {Vorono\"\i}~\cite{Voronoi1905} (1905), 
Hardy~\cite{Hardy1915} (1915)
and Hardy and Landau~\cite{Hardy-Landau1924} (1924).
Therefore, we shall call the identity \eqref{GHI} the generalized Hardy's identity.
\end{rem}

\begin{rem}[The "serendipitous phenomenon" by M.~Taylor]\label{rem:Taylor}
The observation of Taylor~\cite{Taylor-preprint-web,Taylor-preprint} 
can be stated as follows under our notation:
For $d=2$, $\beta=-1/2$ and $0<a<1/2$,
Corollary~\ref{cor:GHIa<1/2} implies that 
\begin{equation}\label{serendipitous}
 S_\lambda (u_{-\frac12,a})(x)
 =
 \sigma_\lambda (U_{-\frac12,a})(x)
  + \frac{\Delta_0(\lambda^2: x) }{\lambda}\sin(2\pi a\lambda)
  + O(\lambda^{-\frac12}).
\end{equation}
He found a certain choppiness of the graphs of $S_\lambda (u_{-\frac12,a})(x)$
outside the disk $\{|x|<a\}$.
This choppiness is absent for the analogous partial Fourier inversion 
$\sigma_\lambda (U_{-\frac12,a})(x)$,
since the graphs of $\sigma_\lambda (U_{-\frac12,a})(x)$ are rotational symmetry. 
He also found a further surprise that,
for certain discrete values of $\lambda$, 
this choppiness magically clears up, 
and $S_\lambda (u_{-\frac12,a})(x)$ behaves about as nicely
on the torus as does $\sigma_\lambda (U_{-\frac12,a})(x)$ on the Euclidean space.
Now, from \eqref{serendipitous}
we know that the set of these discrete values of $\lambda$ 
is 
\begin{equation*}
 \{\lambda>0:\sin(2\pi a\lambda)=0\}.
\end{equation*}
For example, if $a=3/(4\pi)$,
then 
$\lambda=({2\pi}/3)n$ for $n=1,2,\dots$,
which are
correspondent with 
Figure 7A ($n=6$) and Figure 7F ($n=7$) 
in  \cite{Taylor-preprint-web}.
For the coefficient of $\sin(2\pi a\lambda)$,
we can calculate by Lemma~\ref{lem:LPP1}
as
\begin{equation*}
 \frac{\Delta_0(\lambda^2:x)}{\lambda}=O(\lambda^{-\frac13}).
\end{equation*}
\end{rem}

In the rest of this subsection we state three lemmas.
Recall that 
\begin{equation*}
 c(d)
 =d-\frac{2d}{d+1}-\frac{d+1}2
 =\frac{d-5}2+\frac2{d+1}
 =\frac{d(d-4)-1}{2(d+1)}
 =\frac{d-3}2-\frac{d-1}{d+1}.
\end{equation*}

\begin{lem}\label{lem:K 1-4}
Let $d\ge1$, $\beta>-1$ and $a>0$. 
Then 
\begin{equation*}
 \cK_{\beta,a}(\lambda^2:x)
 =
 O(\lambda^{c(d)-\beta})
 \quad\text{uniformly on $\T^d$}.
\end{equation*}
Consequently, 
if $1\le d\le4$ and $\beta>c(d)$, then
\begin{equation*}
 \lim_{\lambda\to\infty}\cK_{\beta,a}(\lambda^2:x)=0 
 \quad\text{uniformly on $\T^d$},
\end{equation*}
and, 
if $d\ge2$, then
\begin{equation*}
 \lim_{\lambda\to\infty}
  \frac{\left|\cK_{\beta,a}(\lambda^2:x) \right|}
             {\lambda^{\frac{d-3}2-\beta}}
  =0 
 \quad\text{uniformly on $\T^d$}.
\end{equation*}
\end{lem}

\begin{lem}\label{lem:K 5-}
Let $d\ge5$, $\beta>-1$ and $a>0$. 
Then
\begin{alignat*}{2}
  &\lim_{\lambda\to\infty}
  \frac{\left|\cK_{\beta,a}(\lambda^2:x) \right|}
             {\lambda^{\frac{d-5}2-\beta}}
  =0, 
  &\quad\text{if $x\in\T^d\setminus\Q^d$},
 \\
  0<
  &\limsup_{\lambda\to\infty}
  \frac{\left|\cK_{\beta,a}(\lambda^2:x) \right|}
             {\lambda^{\frac{d-5}2-\beta}}
  <\infty, 
  &\quad\text{if $x\in\T^d\cap\Q^d$}.
\end{alignat*}
\end{lem}

\begin{lem}\label{lem:K a.e.}
Let $d\ge4$, $\beta>-1/2$ and $a>0$.
Then 
\begin{equation*}
 \lim_{\lambda\to\infty}\cK_{\beta,a}(\lambda^2:x)=0,
 \quad\text{a.e.}\, x\in\T^d.
\end{equation*}
\end{lem}

\subsection{Proof of Theorems~\ref{thm:PinskyPh}--\ref{thm:AEC}}\label{ss:proof}

In this section, 
using Theorem~\ref{thm:RT} and Lemmas~\ref{lem:K 1-4}--\ref{lem:K a.e.}, 
we prove the main theorems.

\subsubsection{Proof of Theorem~\ref{thm:PinskyPh}}
For all $\beta>-1$, 
from \eqref{G -b-1} 
and Lemma~\ref{lem:K 1-4} 
it follows that 
\begin{equation*}
 G_{\beta,a}(\lambda:0)=O(\lambda^{-\beta-1})
 \quad\text{and}\quad
 \cK_{\beta,a}(\lambda^2:0)=O(\lambda^{\frac{d-3}2-\beta}), 
\end{equation*}
respectively.
Since
$\tilde{r}_d(a,0)=r_d(a,0)$ as in \eqref{rd trd},
by Theorem~\ref{thm:RT} we have
\begin{multline}\label{PinskyPh r}
 S_{\lambda}(u_{\beta,a})(0)=u_{\beta,a}(0)
 +
 \big(\sigma_{\lambda}(U_{\beta,a})(0)-U_{\beta,a}(0)\big)
 \\
 +
 r_d(a:0)
 \left( L_{\beta,a}+o(1)\right)\lambda^{-\beta}
 +O(\lambda^{\frac{d-3}2-\beta}).
\end{multline}

If $\beta>(d-3)/2$, 
then, using Theorem~\ref{thm:FI}~(i)~(a), we have
\begin{equation*}
 \sigma_{\lambda}(U_{\beta,a})(0)-U_{\beta,a}(0)=O(\lambda^{\frac{d-3}2-\beta}).
\end{equation*}
Combining this with \eqref{PinskyPh r}, we have
\begin{equation*}
 S_{\lambda}(u_{\beta,a})(0)
 =
 u_{\beta,a}(0)
 +
 r_d(a:0)
 \left( L_{\beta,a}+o(1)\right)\lambda^{-\beta}
 +O(\lambda^{\frac{d-3}2-\beta}),
\end{equation*}
which shows (i).

If $d\ge2$ and $-1<\beta\le(d-3)/2$, 
then, using \eqref{FI 0} in Theorem~\ref{thm:FI}~(i), 
we have 
\begin{equation*}
 \sigma_{\lambda}(U_{\beta,a})(0)
 -
 U_{\beta,a}(0)
 =-P_{\beta,a}^{[d]}
  \cos\left(2\pi a\lambda-\frac{d-1+2\beta}4\pi\right)
  \lambda^{\frac{d-3}2-\beta}
 +O(\lambda^{\frac{d-5}2-\beta}),
\end{equation*}
which shows (ii).

\subsubsection{Proof of Theorem~\ref{thm:|x|=a and PWC}}
By Theorem~\ref{thm:FI}~(ii) and (iii) 
we have
\begin{equation*}
 \lim_{\lambda\to\infty}\frac{\sigma_{\lambda}(U_{\beta,a})(x)-U_{\beta,a}(x)}{\lambda^{-\beta}}
 =
 \begin{cases}
 L_{\beta,a}, & x\in \tG_a, \\
 0, & x\in \tE_a,
 \end{cases}
\end{equation*}
which shows
\begin{equation}\label{sU-U+r}
 \lim_{\lambda\to\infty}\frac{\sigma_{\lambda}(U_{\beta,a})(x)-U_{\beta,a}(x)}{\lambda^{-\beta}}
 +\tilde{r}_d(a:x)L_{\beta,a}
 =
 r_d(a:x)L_{\beta,a},
 \quad
 x\in E_a\cup G_a.
\end{equation}

\noindent
{\bf Proof of (i) (a)} \ 
Let $1\le d\le4$.
By \eqref{G -b-1} and Lemma~\ref{lem:K 1-4} we have
\begin{equation*}
 \frac{G_{\beta,a}(\lambda:x)}{\lambda^{-\beta}}
 =O(\lambda^{-1}),
 \quad
 \frac{\cK_{\beta,a}(\lambda^2:x)}{\lambda^{-\beta}}
 =O(\lambda^{\frac{d-5}2+\frac 2{d+1}})=o(1).
\end{equation*}
By Theorem~\ref{thm:RT} and \eqref{sU-U+r} we have
\begin{align*}
 &\frac{S_\lambda(u_{\beta,a})(x)-u_{\beta,a}(x)}{\lambda^{-\beta}}
\\
 &=
 \frac{\sigma_{\lambda}(U_{\beta,a})(x)-U_{\beta,a}(x)}{\lambda^{-\beta}}
 + \frac{G_{\beta,a}(\lambda:x)}{\lambda^{-\beta}}
 +\tilde{r}_d(a:x)\Big(L_{\beta,a}+o(1)\Big)
\\
 &\phantom{****************************}
 + \frac{\cK_{\beta,a}(\lambda^2:x)}{\lambda^{-\beta}} 
  +O(\lambda^{-1})
\\
 &\to
 r_d(a:x)L_{\beta,a}
 \quad\text{as}\quad \lambda\to\infty,
\end{align*}
which shows the conclusion.

\noindent
{\bf Proof of (i) (b)} \ 
Let $\beta>c(d)=\frac{d-5}2+\frac 2{d+1}$.
Then by Lemma~\ref{lem:K 1-4} we have
\begin{equation*}
 \cK_{\beta,a}(\lambda^2:x)
 =O(\lambda^{\frac{d-5}2+\frac 2{d+1}-\beta})=o(1)
 \quad\text{uniformly on $\T^d$}.
\end{equation*}
If $x\in E_a$,
then $ \tilde{r}_d(a:x)=0$ as \eqref{rd trd}.
By Theorem~\ref{thm:FI} (iii) and \eqref{G -b-1} we have
\begin{equation*}
 \lim_{\lambda\to\infty}
 (\sigma_\lambda(U_{\beta,a})(x)-U_{\beta,a}(x))
 =0,
\quad
 \lim_{\lambda\to\infty}
 G_{\beta,a}(\lambda: x)
 =0
\end{equation*}
uniformly on any compact set in $E_a$.
Hence, by Theorem~\ref{thm:RT} we have
\begin{align*}
 &S_\lambda(u_{\beta,a})(x)
 -
 u_{\beta,a}(x)
\\
 &=
 (\sigma_\lambda(U_{\beta,a})(x)
 -U_{\beta,a}(x))
 +G_{\beta,a}(\lambda: x)
 +\cK_{\beta,a}(\lambda^2:x) 
  +O(\lambda^{-\beta-1})
\\
 &\to0
 \quad\text{as}\quad \lambda\to\infty
\end{align*}
uniformly on any compact set in $E_a$,
which shows the conclusion.

\noindent
{\bf Proof of (ii)} \ 
Let $d\ge5$ and $x\in E_a\cup G_a$.
By Theorem~\ref{thm:RT} we have
\begin{align*}
 &
 \frac1{\lambda^{\frac{d-5}2}}
 \left(
  \frac{S_\lambda(u_{\beta,a})(x)-u_{\beta,a}(x)}{\lambda^{-\beta}}
  -r_d(a:x)L_{\beta,a}
 \right)
\\
 &=
 \frac1{\lambda^{\frac{d-5}2}}
 \left(
 \frac{\sigma_{\lambda}(U_{\beta,a})(x)-U_{\beta,a}(x)}{\lambda^{-\beta}}
 +\tilde{r}_d(a:x)\Big(L_{\beta,a}+o(1)\Big)
  -r_d(a:x)L_{\beta,a}
 \right)
\\
 &\phantom{**}
 +\frac{G_{\beta,a}(\lambda:x)}{\lambda^{\frac{d-5}2-\beta}}
 + \frac{\cK_{\beta,a}(\lambda^2:x)}{\lambda^{\frac{d-5}2-\beta}} 
  +O(\lambda^{-\frac{d-5}2-1}).
\end{align*}
By \eqref{sU-U+r}, \eqref{G -b-1} and Lemma~\ref{lem:K 5-} we have
\begin{equation*}
 \lim_{\lambda\to\infty}
 \frac1{\lambda^{\frac{d-5}2}}
 \left(
  \frac{S_\lambda(u_{\beta,a})(x)-u_{\beta,a}(x)}{\lambda^{-\beta}}
  -r_d(a:x)L_{\beta,a}
 \right)
 =
 \lim_{\lambda\to\infty}
 \frac{\cK_{\beta,a}(\lambda^2:x)}{\lambda^{\frac{d-5}2-\beta}}
 =0
\end{equation*}
for all $x\in(E_a\cup G_a)\setminus\Q^d$,
and
\begin{multline*}
 0<
 \limsup_{\lambda\to\infty}
 \frac1{\lambda^{\frac{d-5}2}}
 \left|
  \frac{S_\lambda(u_{\beta,a})(x)-u_{\beta,a}(x)}{\lambda^{-\beta}}
  -r_d(a:x)L_{\beta,a}
 \right|
\\
 =
 \limsup_{\lambda\to\infty}
 \frac{\left|\cK_{\beta,a}(\lambda^2:x)\right|}{\lambda^{\frac{d-5}2-\beta}}
 <\infty
\end{multline*}
for all $x\in(E_a\cup G_a)\cap\Q^d$.
The proof is complete.

\subsubsection{Proof of Theorem~\ref{thm:GibbsPh}}

Let $1\le d\le 4$, $c(d)<\beta\le 0$ and $0<a<1/2$. 
Then by Lemma~\ref{lem:K 1-4} we have
\begin{equation*}
 \frac{\cK_{\beta,a}(\lambda^2:x_{\lambda}^{\pm})}{\lambda^{-\beta}}
 =o(1).
\end{equation*}
If $m\ne0$, then $|m-x_0|\ge1-a$, since $|x_0|=a$.
Then,  for large $\lambda$, we have
\begin{equation*}
 |a-|x_{\lambda}^{\pm}-m||
 \ge
 |m-x_0|-|x_{\lambda}^{\pm}-x_0|-a
 \ge
 1-2a-|x_{\lambda}^{\pm}-x_0|
 \ge
 \frac{1-2a}2>0.
\end{equation*}
Hence, by \eqref{G -b-1} we have
\begin{equation*}
 \frac{G_{\beta,a}(\lambda:x_{\lambda}^{\pm})}
      {\lambda^{-\beta}}
 =O(\lambda^{-1}).
\end{equation*}
Since $x_{\lambda}^{\pm}\in E_a$, 
$\tilde{r}_d(a:x_{\lambda}^{\pm})=0$ as in \eqref{rd trd}.
By Theorem~\ref{thm:RT} and Theorem~\ref{thm:FI}~(v) we have
\begin{align*}
 &
 \frac{S_{\lambda}(u_{\beta,a})(x_{\lambda}^{\pm})-u_{\beta,a}(x_{\lambda}^{\pm})}
      {\lambda^{-\beta}}
\\
 &=
 \left(
  \frac{\sigma_{\lambda}(U_{\beta,a})(x_{\lambda}^{\pm})-U_{\beta,a}(x_{\lambda}^{\pm})}
       {\lambda^{-\beta}}
 \right)
 +
 \frac{G_{\beta,a}(\lambda:x_{\lambda}^{\pm})}
      {\lambda^{-\beta}}
 +
 \frac{\cK_{\beta,a}(\lambda^2:x_{\lambda}^{\pm})}{\lambda^{-\beta}}
 +O(\lambda^{-1})
\\
 &\to
 G_{\beta,a}^{\pm}
 \quad\text{as}\quad \lambda\to\infty,
\end{align*}
which is the conclusion.

\subsubsection{Proofs of Theorem~\ref{thm:AEC}}
Let $d\ge 4$, $\beta>-1/2$ and $a>0$. 
Then by lemma~\ref{lem:K a.e.} we have
\begin{equation*}
 \lim_{\lambda\to\infty}\cK_{\beta,a}(\lambda^2:x)=0,
 \quad\text{a.e.}\, x\in\T^d.
\end{equation*}
Let $x\in E_a$.
Then $\tilde{r}_d(a:x)=0$ as in \eqref{rd trd}.
By Theorem~\ref{thm:FI} (iii) and \eqref{G -b-1} we have
\begin{equation*}
 \sigma_{\lambda}(U_{\beta,a})(x)-U_{\beta,a}(x)
 =O(\lambda^{-\beta-1})
 \quad\text{and}\quad
 G_{\beta,a}(\lambda:x)
 =O(\lambda^{-\beta-1}).
\end{equation*}
Then by Theorem~\ref{thm:RT} 
we have, for a.e.\,$x\in E_a$,
\begin{align*}
 &
 S_\lambda(u_{\beta,a})(x)
 -
 u_{\beta,a}(x)
\\
 &=
 \big(\sigma_\lambda(U_{\beta,a})(x)-U_{\beta,a}(x)\big)
 +G_{\beta,a}(\lambda:x)
 + \cK_{\beta,a}(\lambda^2:x) 
  +O(\lambda^{-\beta-1})
\\
 &\to0 
 \quad\text{as}\quad \lambda\to\infty,
\end{align*}
which shows the conclusion,
since the measure of $\T^d\setminus E_a$ is zero.

\subsection{Proof of Theorem~\ref{thm:RT}}\label{ss:proof-thm:RT}

By Corollary~\ref{cor:GHI} and \eqref{K} we have that, 
for all $x\in\T^d$,
\begin{multline*}
 S_\lambda(u_{\beta,a})(x)
 =
 \sigma_{\lambda}(U_{\beta,a})(x) 
 +\cK_{\beta,a}(\lambda^2:x)
\\ 
 +(-1)^{\dsharp+1}\sum_{m\in\Z^d\setminus\{0\}}
 \int_0^{\lambda^2}{\cD}_{\dsharp}(s:x-m)A_{\beta,a}^{(\dsharp+1)}(s)\,ds.
\end{multline*}
Using \eqref{cD} and \eqref{A_beta,a} 
we have
\begin{align*}
 &{\cD}_{\dsharp}(s:x-m)A_{\beta,a}^{(\dsharp+1)}(s) \\
 &=
 \frac{s^{\frac d2+\dsharp}}{\pi^{\dsharp}}
 \frac{J_{{\frac d2}+\dsharp}(2\pi\sqrt{s}|x-m|)}{(\sqrt{s}|x-m|)^{{\frac d2}+\dsharp}}
 (-1)^{\dsharp+1}{\frac{\Gamma(\beta+1)}{\pi^{\beta-(\dsharp+1)}}}a^{{\frac d2}+\beta+\dsharp+1}
 \frac{J_{{\frac d2}+\beta+\dsharp+1}(2\pi a\sqrt{s})}{s^{{\frac d2}+\beta+\dsharp+1}} \\
 &=
 (-1)^{\dsharp+1} 2^{\beta}\Gamma(\beta+1) a^{2\beta}
 \frac{J_{{\frac d2}+\dsharp}(2\pi\sqrt{s}|x-m|) J_{{\frac d2}+\beta+\dsharp+1}(2\pi a\sqrt{s})}
      {(\frac{|x-m|}{a})^{\frac{d}2+\dsharp}(2\pi a\sqrt{s})^{\beta}}
  \frac{\pi a}{\sqrt{s}}.
\end{align*}
Then, letting $u=2\pi a\sqrt{s}$, 
we have
\begin{multline}\label{DA kd}
 (-1)^{\dsharp+1}
 \int_0^{\lambda^2}\cD_{\dsharp}(s:x-m)A_{\beta,a}^{(\dsharp+1)}(s)\,ds
\\
 =
 2^{\beta}\Gamma(\beta+1) a^{2\beta}
 \int_0^{2\pi a\lambda}
 \dfrac{J_{{\frac d2}+\dsharp}(\frac{|x-m|}{a}u)\,J_{{\frac d2}+\beta+\dsharp+1}(u)}
 {(\frac{|x-m|}{a})^{\frac{d}2+\dsharp}\,u^{\beta}}
 \,du.
\end{multline}
If $|x-m|\ne a$, then 
Corollary~\ref{cor:Bessel1} shows that
\begin{multline}\label{JJ kd}
 2^{\beta}\Gamma(\beta+1) a^{2\beta}
 \int_0^{\infty}
 \dfrac{J_{{\frac d2}+\dsharp}(\frac{|x-m|}{a}u)\,J_{{\frac d2}+\beta+\dsharp+1}(u)}
 {(\frac{|x-m|}{a})^{\frac{d}2+\dsharp}\,u^{\beta}}
 \,du
\\
 =
 U_{\beta,a}(x-m)
 =
 \begin{cases}
 0, & \text{if}\ \ a<|x-m|,\\
 (a^2-|x-m|^2)^{\beta}, & \text{if}\ \ 0<|x-m|<a.
 \end{cases}
\end{multline}
If $m\ne0$ and $|x-m|\ne a$, then 
Lemma~\ref{lem:Bessel3} shows that
\begin{multline}\label{JJ kd 2}
 2^{\beta}\Gamma(\beta+1) a^{2\beta}
 \int_{2\pi a\lambda}^{\infty} 
 \dfrac{J_{{\frac d2}+\dsharp}(\frac{|x-m|}{a}u)\,J_{{\frac d2}+\beta+\dsharp+1}(u)}
 {(\frac{|x-m|}{a})^{\frac{d}2+\dsharp}\,u^{\beta}}
 \,du \\
 =
 \frac{2^{\beta}\Gamma(\beta+1)a^{\beta}}{\pi}
 \left(\frac a{|x-m|}\right)^{\frac{d+1}2+\dsharp}
 \psi_{\beta}(\lambda,a-|x-m|)
 +\frac1{|x-m|^{\frac{d+1}2+\dsharp}}
 O(\lambda^{-\beta-1}),
\end{multline}
since $|x-m|\ge1/2$ for $m\ne0$ and $x\in\T^d$.
Combining \eqref{DA kd}, \eqref{JJ kd} and \eqref{JJ kd 2},
and observing \eqref{Gx}--\eqref{G converge},
we have
\begin{align*}
 &(-1)^{\dsharp+1}
 \sum_{m\in\Z^d\setminus\{0\},\ |x-m|\ne a}
 \int_0^{\lambda^2}{\cD}_{\dsharp}(s:x-m)A_{\beta,a}^{(\dsharp+1)}(s)\,ds
\\
 &=
 \left(\sum_{m\in\Z^d\setminus\{0\},\ |x-m|\ne a}
 U_{\beta,a}(x-m)\right)
 -G_{\beta,a}(\lambda:x)
 +O(\lambda^{-\beta-1}).
\\
 &=
 u_{\beta,a}(x)-U_{\beta,a}(x)
 -G_{\beta,a}(\lambda:x)
 +O(\lambda^{-\beta-1}).
\end{align*}
On the other hand,
if $m\ne0$ and $|x-m|=a$,
then, using \eqref{DA kd} and 
Corollary~\ref{cor:Bessel44}, 
we have
\begin{align*}
 &(-1)^{\dsharp+1}
 \int_0^{\lambda^2}
 \cD_{\dsharp}(s:x-m)A_{\beta,a}^{(\dsharp+1)}(s)
 \,ds
\\
 &= 
 2^{\beta}\Gamma(\beta+1) a^{2\beta}
 \int_0^{2\pi a\lambda}
 \dfrac{J_{{\frac d2}+\dsharp}(u)\,J_{{\frac d2}+\beta+\dsharp+1}(u)}
 {\,u^{\beta}}
 \,du
\\
 &= 
 L_{\beta,a}
 \lambda^{-\beta} 
 +
 \begin{cases}
  O(\lambda^{-\beta-1}), & \beta\ge0, \\
  O(1), & -1<\beta<0
 \end{cases}
\\
&=\left(L_{\beta,a}+o(1)\right)\lambda^{-\beta}.
\end{align*}
Therefore, 
we have the conclusion.

\subsection{Proof of Lemma~\ref{lem:K 1-4}}\label{ss:proof-lem:K 1-4}
To estimate 
\begin{equation*}
 \cK_{\beta,a}(\lambda^2:x) 
 =
 \sum_{j=0}^{\dsharp}(-1)^j\Delta_j(\lambda^2:x)A_{\beta,a}^{(j)}(\lambda^2),
\end{equation*}
we combine the estimates of $\Delta_j(\lambda^2:x)$ 
and $A_{\beta,a}^{(j)}(\lambda^2)$.
Firstly, by \eqref{Bessel-asymp-infty}
we see that
\begin{multline}\label{Aj-asymp}
 A_{\beta,a}^{(j)}(s)
 =
 (-1)^{j}{\frac{\Gamma(\beta+1)}{\pi^{\beta-j}}}a^{{\frac d2}+\beta+j}
 \frac{J_{{\frac d2}+\beta+j}(2\pi a\sqrt{s})}
      {s^{\frac12({\frac d2}+\beta+j)}}.
\\
 =
 (-1)^{j}{\frac{\Gamma(\beta+1)}{\pi^{\beta-j+1}}}
 \frac{a^{{\frac d2}+\beta+j-\frac12}}
      {s^{\frac12({\frac d2}+\beta+j+\frac12)}}
 \cos\left(2\pi a\sqrt{s}-\frac{d+2\beta+2j+1}4\pi\right)
 \\
 +O(s^{-\frac12({\frac d2}+\beta+j+\frac32)})
 \quad\text{as $s\to\infty$}.
\end{multline}
That is, for some positive constant $C$,
\begin{equation}\label{Aaaj}
 |A_{\beta,a}^{(j)}(\lambda^2)|\le C\lambda^{-(\frac d2+\beta+j+\frac12)},
 \quad \lambda\ge1,
\end{equation}
For the terms $\Delta_j(s:x)$,
we use Lemma~\ref{lem:LPP1}, that is, 
\begin{equation*}\label{Dj}
 \Delta_{j}(\lambda^2:x) 
 =
 \begin{cases}
 O(\lambda^{d-\frac{2d}{d+1}}), 
 & \text{if $j=0$},\\
 O(\lambda^{d-\frac{2d}{d+1}+\frac{2j}{d+1}+\ve}) \ \text{for every $\ve>0$}, 
 & \text{if $0<j\le\frac{d-1}2$},\\
 O(\lambda^{\frac{d-1}2+j}), 
 & \text{if } j>\frac{d-1}2.
 \end{cases}
\end{equation*}
If $j=0$, then 
\begin{equation}\label{DeltaA 0}
 |\Delta_0(\lambda^2:x)A_{\beta,a}(\lambda^2)|
 \le
 C\lambda^{d-\frac{2d}{d+1}}\lambda^{-(\frac d2+\beta+\frac12)}
 =
 C\lambda^{\frac{d-5}2+\frac{2}{(d+1)}-\beta}.
\end{equation}
If $0<j<\dsharp$, 
then $0<j\le(d-1)/2$
and then, 
for any small $\ve>0$,
\begin{multline}\label{DeltaA j}
 |\Delta_j(\lambda^2:x)A_{\beta,a}^{(j)}(\lambda^2)|
 \le
 C\lambda^{d-\frac{2d}{d+1}+\frac{2j}{d+1}+\ve}\lambda^{-(\frac d2+\beta+j+\frac12)}
 =
 C\lambda^{\frac{d-5}2+\frac{2-j(d-1)}{(d+1)}-\beta+\ve}.
\end{multline}
If $j=\dsharp$, then $j>(d-1)/2$ and then
\begin{equation}\label{DeltaA kd}
 |\Delta_{\dsharp}(\lambda^2:x)A_{\beta,a}^{(\dsharp)}(\lambda^2)|
 \le
 C\lambda^{\frac{d-1}2+\dsharp} \lambda^{-(\frac d2+\beta+\dsharp+\frac12)}
 =
 C\lambda^{-1-\beta}.
\end{equation}
Comparing \eqref{DeltaA 0}, \eqref{DeltaA j} and \eqref{DeltaA kd},
we have the conclusion.

\subsection{Proof of Lemma~\ref{lem:K 5-}}\label{ss:proof-lem:K 5-}
By Lemma~\ref{lem:LPP2}
we have that, as $\lambda\to\infty$,
\begin{align*}
 \Delta_j(\lambda^2:x)
 &=
 \begin{cases}
 O(\lambda^{d-2}) 
 & \text{for $x\in\T^d\cap\Q^d$ and if $0\le j<(d-4)/2$}, 
 \\
 o(\lambda^{d-2}) 
 & \text{for $x\in\T^d\setminus\Q^d$ and if $0\le j<(d-4)/2$}, 
 \\
 O(\lambda^{\frac{2(d-1+j)}3+\ve})
 & \text{for $x\in\T^d$ and if $(d-4)/2\le j\le(d-1)/2$}.
 \end{cases}
\end{align*}
Combining this and \eqref{Aaaj}, we have the following estimates:
If $0\le j<(d-4)/2$ and $x\in\T^d\setminus\Q^d$, then,
for some decreasing function $\vp(\lambda)$
which satisfies $\vp(\lambda)\to0$ as $\lambda\to\infty$,
\begin{equation*}
 |\Delta_j(\lambda^2:x)A_{\beta,a}^{(j)}(\lambda^2)|
 \le
 C\lambda^{d-2}\vp(\lambda)\lambda^{-(\frac d2+\beta+j+\frac12)}
 =
 C\lambda^{\frac{d-5}2-\beta-j}\vp(\lambda).
\end{equation*}
If $0\le j<(d-4)/2$ and $x\in\T^d\cap\Q^d$, then 
\begin{equation*}
 |\Delta_j(\lambda^2:x)A_{\beta,a}^{(j)}(\lambda^2)|
 \le
 C\lambda^{d-2}\lambda^{-(\frac d2+\beta+j+\frac12)}
 =
 C\lambda^{\frac{d-5}2-\beta-j}.
\end{equation*}
If $(d-4)/2\le j\le(d-1)/2$, then,
for small $\ve>0$,
\begin{multline*}
 |\Delta_j(\lambda^2:x)A_{\beta,a}^{(j)}(\lambda^2)|
 \le
 C\lambda^{\frac{2(d-1+j)}3+\ve}\lambda^{-(\frac d2+\beta+j+\frac12)}
 =
 C\lambda^{\frac{d-7}6-\beta-\frac j3+\ve} 
 \le 
 C\lambda^{-\frac12-\beta+\ve}.
\end{multline*}
If $j=\dsharp$, then we have \eqref{DeltaA kd}.
Comparing these estimates, we have
\begin{equation*}
 |\cK_{\beta,a}(\lambda^2:x)|
 \le
 \begin{cases}
 C\lambda^{\frac{d-5}2-\beta}\vp(\lambda)
 & \text{for $x\in\T^d\setminus\Q^d$},
 \\
 C\lambda^{\frac{d-5}2-\beta}
 & \text{for $x\in\T^d\cap\Q^d$},
 \end{cases}
\end{equation*}
which shows
\begin{alignat*}{2}
  &\lim_{\lambda\to\infty}
  \frac{\left|\cK_{\beta,a}(\lambda^2:x) \right|}
             {\lambda^{\frac{d-5}2-\beta}}
  =0, 
  &\quad\text{if $x\in\T^d\setminus\Q^d$},
 \\
  &\limsup_{\lambda\to\infty}
  \frac{\left|\cK_{\beta,a}(\lambda^2:x) \right|}
             {\lambda^{\frac{d-5}2-\beta}}
  <\infty, 
  &\quad\text{if $x\in\T^d\cap\Q^d$}.
\end{alignat*}
Next we shall prove 
\begin{equation*}
 0<
 \limsup_{\lambda\to\infty}
  \frac{\left|\cK_{\beta,a}(\lambda^2:x) \right|}
             {\lambda^{\frac{d-5}2-\beta}},
  \quad\text{if $x\in\T^d\cap\Q^d$}.
\end{equation*}
If $j=0$ and $x\in\T^d\cap\Q^d$, 
then, using Lemma~\ref{lem:LPP2S}, we have 
\begin{equation*}
 |\Delta_0(\lambda_k^2:x)|
 \ge
 C(x)\lambda_k^{d-2},
 \quad
 \lambda_k\in[\ell_k,m_k],\ k\in\N.
\end{equation*}
On the other hand, by \eqref{Aj-asymp},
\begin{multline*}
 A_{\beta,a}(\lambda^2)
 =
 {\frac{\Gamma(\beta+1)}{\pi^{\beta+1}}}
 \frac{a^{{\frac d2}+\beta-\frac12}}
      {\lambda^{{\frac d2}+\beta+\frac12}}
 \cos\left(2\pi a\lambda-\frac{d+2\beta+1}4\pi\right) \\
 +O(\lambda^{-({\frac d2}+\beta+\frac32)})
 \quad\text{as $\lambda\to\infty$}.
\end{multline*}
If $\lambda\in[\ell_k,m_k]$, then
\begin{equation*}
 k\pi-\frac14\pi
 \le
 2\pi a\lambda-\frac{d+2\beta+1}4\pi
 \le
 k\pi+\frac14\pi,
\end{equation*}
that is,
\begin{equation*}
 \left|\cos\left(2\pi a\lambda-\frac{d+2\beta+1}4\pi\right)\right|
 \ge
 \frac1{\sqrt 2}.
\end{equation*}
Therefore, if $x\in\T^d\cap\Q^d$,
then, for $\lambda_k\in[\ell_k,m_k]$, 
\begin{equation*}
 |\Delta_0(\lambda_k^2:x)A_{\beta,a}(\lambda_k^2)|
 \ge
 C_{\beta,a}C(x)\lambda_k^{d-2}\lambda_k^{-({\frac d2}+\beta+\frac12)}
 =
 C_{\beta,a}C(x)\lambda_k^{\frac{d-5}2-\beta},
\end{equation*}
that is,
\begin{equation*}
 0<\limsup_{\lambda\to\infty}
 \frac{|\Delta_0(\lambda^2:x)A_{\beta,a}(\lambda^2)|}
      {\lambda^{\frac{d-5}2-\beta}}.
\end{equation*}
The proof is complete.

\subsection{Proof of Lemma~\ref{lem:K a.e.}}\label{ss:proof-lem:K a.e.}

If $j=\dsharp$, then we have \eqref{DeltaA kd}.
If $0\le j\le(d-1)/2$, then, using Lemma~\ref{lem:LPP3} and \eqref{Aaaj}, 
we have, for small $\ve>0$,
\begin{multline*}
 |\Delta_j(\lambda^2:x)A_{\beta,a}^{(j)}(\lambda^2)|
 \le
 C\lambda^{\frac{d}2+\frac{d-2}{d-1}j+\ve}\lambda^{-(\frac d2+\beta+j+\frac12)}
 =
 C\lambda^{-\frac12-\beta-\frac j{d-1}+\ve}
 \le
 C\lambda^{-\frac12-\beta+\ve},
 \\ \lambda\ge1,\ \text{a.e.}\, x\in\T^d.
\end{multline*}
This shows the conclusion.

\section{A relation between multiple Fourier series and lattice point problems}\label{sec:relation}

Let $d\ge1$, $\beta>-1$ and $0<a<1/2$.
Then Corollary~\ref{cor:GHIa<1/2} has shown that
\begin{equation*}
 S_\lambda(u_{\beta,a})(x)
 =
 \sigma_\lambda (U_{\beta,a})(x)
 + \cK_{\beta,a}(\lambda^2:x) 
  +O(\lambda^{-\beta-1})
 \quad\text{as}\quad \lambda\to\infty
\end{equation*}
for all $x\in\T^d$.
The behavior of $\sigma_\lambda(U_{\beta,a})(x)$ 
has been clarified by Theorem~\ref{thm:FI}.
Therefore, to clarify the behavior of $S_\lambda(u_{\beta,a})(x)$
we need to investigate the term 
\begin{equation*}
 \cK_{\beta,a}(\lambda^2:x) 
 =
 \sum_{j=0}^{\dsharp}(-1)^j\Delta_j(\lambda^2:x)A_{\beta,a}^{(j)}(\lambda^2).
\end{equation*}
Since the estimates for the terms $A_{\beta,a}^{(j)}(s)$ 
are gotten as \eqref{Aaaj},
the convergence problem depends only on the terms $\Delta_j(s:x)$,
which are connected with the lattice point problem.
Especially, the estimate for $\Delta_0(s:x)$ is very important 
and difficult.

First we state the following theorem, which will be proven later:

\begin{thm}\label{thm:conj d=23}
Let $d=2$ or $3$, $0<a<1/2$ and $x_0\in\T^d$. 
Then, 
\begin{equation*}
 S_\lambda(u_{\beta,a})(x_0)-\sigma_\lambda(U_{\beta,a})(x_0)=o(1)
 \ \text{as\ $\lambda\to\infty$ for all $\beta>-1$},
\end{equation*}
if and only if 
\begin{equation}\label{conj}
 \Delta_0(s:x_0)=O(s^{\frac{d-1}4+\varepsilon})
 \ \text{as $s\to\infty$ for all $\varepsilon>0$}.
\end{equation}
\end{thm}

For $d=1$ the convergence property of 
$S_\lambda(u_{\beta,a})(x)$ and $\sigma_\lambda(U_{\beta,a})(x)$ are well known
and the estimate \eqref{conj} holds.
Actually, we have $\Delta_0(s:x)=O(1)$ for all $x\in\T$, 
see Remark~\ref{rem:Landau}.
However, for $d=2$ and $d=3$, the estimate \eqref{conj} is an open problem,
see Remarks~\ref{rem:conj 2} and~\ref{rem:conj 3}, respectively.
For $d\ge4$, see Remarks~\ref{rem:conj 4} and~\ref{rem:conj-aec}.

\begin{rem}\label{rem:conj 2}
If $d=2$, then
the following fact is well known 
(see \cite[Hauptsatz 3]{Novak1969} and Remark~\ref{rem:Delta-P}):
\begin{equation}\label{C2D2}
 C_2(x)\,t^{\frac14}
 <
 \left(\frac1t\int_0^t|\Delta_0(s:x)|^2\,ds\right)^\frac12
 <
 D_2(x)\,t^{\frac14} \ \text{for all $x\in\T^2$},
\end{equation}
where $C_2(x)$ and $D_2(x)$ are positive constants depending on $x\in\T^2$.
Therefore it is natural to conjecture that 
$\Delta_0(s:x)=O(s^{\frac14+\ve})$ for all $x\in\T^2$.
However,
\begin{equation*}
 \text{``$\Delta_0(s:0)=O(s^{\frac14+\varepsilon})$ for all $\ve>0$''}
\end{equation*}
is an open problem as famous Hardy's conjecture on Gauss's circle problem,
since
\begin{equation*}
 \Delta_0(s:0)
 =
 D_0(s:0)-\cD_0(s:0)
 =
 \sum_{|m|^2<s}1-\int_{|\xi|^2<s}\,d\xi,
\end{equation*}
which is the difference between 
the number of lattice points inside the circle
and 
the area of the circle.
Up to now the best result on this problem is 
$\Delta_0(s:0)=O(s^{131/416}(\log s)^{18637/8320})$ 
by M.~N.~Huxley~\cite{Huxley2003} in 2003.
(Recently, Bourgain and Watt~\cite{Bourgain-Watt_arXiv1709-04340v1} 
gave $\theta=517/1648=0.31371\dots$ in arXiv, 2017.)
By Theorem~\ref{thm:conj d=23}, 
it turns out that
Hardy's conjecture on Gauss's circle problem
is equivalent to
\begin{equation*}
 S_\lambda(u_{\beta,a})(0)-\sigma_\lambda(U_{\beta,a})(0)=o(1)
 \ \text{as\ $\lambda\to\infty$ for all $\beta>-1$}.
\end{equation*}
\end{rem}

\begin{rem}\label{rem:conj 3}
If $d=3$ and  for all $ x\in\T^3$ then 
(see \cite[Hauptsatz 3]{Novak1969} and Remark~\ref{rem:Delta-P})
\begin{equation}\label{C3D3}
 C_3(x)\,t^{\frac12}
 <
 \left(\frac1t\int_0^t|\Delta_0(s:x)|^2\,ds\right)^\frac12
 <
 D_3(x)\,t^{\frac12}\log^{\frac12}t \ \text{for all $x\in\T^3$} ,
\end{equation}
where $C_3(x)$ and $D_3(x)$ are positive constants depending on $x\in\T^3$.
Therefore it is natural to conjecture that 
$\Delta_0(s:x)=O(s^{\frac12+\varepsilon})$ for all $x\in\T^3$.
However, this problem is also very hard. Up to now the best result on this problem is 
$\Delta_0(s:0)=O(s^{21/32+\varepsilon})$ 
by D.~R.~Heath-Brown~\cite{HeathBrown1999} in 1999.
\end{rem}

For $d\ge4$ we have the following theorem. The proof will be given later:

\begin{thm}\label{thm:conj d=4}
Let $d\ge 4$, $\beta>-1$ and $0<a<1/2$. 
For $x_0\in\T^d$,
if \eqref{conj} holds,
then
\begin{equation*}
 S_\lambda(u_{\beta,a})(x_0)-\sigma_\lambda(U_{\beta,a})(x_0)=o(1)
 \ \text{as $\lambda\to\infty$}.
\end{equation*}
\end{thm}

\begin{rem}\label{rem:conj 4}
If $d\ge4$ and $x\in\T^d\cap\Q^d$, then 
(see \cite[Hauptsatz 1]{Novak1969} and Remark~\ref{rem:Delta-P})
\begin{equation*}
 \left(\frac1t\int_0^t|\Delta_0(s:x)|^2\,ds\right)^\frac12
 =
 M_d(x)\,t^{\frac d2-1}+O(t^{\frac d2-\frac32+\ve}),
\end{equation*}
where $M_d(x)$ is a positive constant depending on $x$.
Therefore, the estimate \eqref{conj} fails at $x\in\T^d\cap\Q^d$ for $d\ge4$.
\end{rem}

By Propositions~\ref{thm:conj d=23} and~\ref{thm:conj d=4}
we have the following corollary immediately:

\begin{cor}\label{cor:conj-aec}
Let $d\ge2$, $\beta>-1$ and $0<a<1/2$. 
Assume that, for all $\ve>0$,
\begin{equation}\label{conj-ae}
 \Delta_0(s:x)
 =
 O(s^{\frac{d-1}4+\ve})
 \ \text{as $s\to\infty$ for a.e. $x\in\T^d$}.
\end{equation}
Then
\begin{equation*}
 S_\lambda(u_{\beta,a})(x)- \sigma_\lambda(U_{\beta,a})(x)=o(1)
 \ \text{as $\lambda\to\infty$ for a.e. $x\in\T^d$}.
\end{equation*}
Consequently,
\begin{equation*}
 \lim_{\lambda\to\infty}S_\lambda(u_{\beta,a})(x)
 =
 u_{\beta,a}(x)
 \quad\text{for a.e. $x\in\T^d$}.
\end{equation*}
\end{cor}

\begin{rem}\label{rem:conj-aec}
If $d=2$ or $d=3$, then \eqref{conj-ae} is conjectured naturally by \eqref{C2D2}
and \eqref{C3D3}.
If $d\ge4$, then the following is known 
(see \cite[Theorem 5]{Novak1970} and Remark~\ref{rem:Delta-P})
\begin{equation*}
 \left(\frac1t\int_0^t |\Delta_0(s:x)|^2 \,ds\right)^{1/2}
 =
 O(t^\frac{d-1}4)\ \ \text{for a.e. $x$}.
\end{equation*}
Therefore \eqref{conj-ae} is conjectured naturally for $d\ge4$ also.
\end{rem}

\begin{proof}[Proof of Theorem~\ref{thm:conj d=23}]\label{proof:conj d=23}
First observe that Theorem~\ref{thm:RT} implies that
\begin{equation*}
 S_\lambda(u_{\beta,a})(x_0)=\sigma_\lambda(U_{\beta,a})(x_0)+o(1)
 \quad\text{as}\quad \lambda\to\infty,
\end{equation*}
if and only if
\begin{equation}\label{K=o(1)}
 \cK_{\beta,a}(\lambda^2:x_0) 
 =
 \sum_{j=0}^{\dsharp}(-1)^j\Delta_j(\lambda^2:x_0)A_{\beta,a}^{(j)}(\lambda^2)
 =
 o(1).
\end{equation}

If $j=\dsharp$, then by \eqref{DeltaA kd} we have
\begin{equation*}
 \Delta_{\dsharp}(\lambda^2:x_0)A_{\beta,a}^{(\dsharp)}(\lambda^2)
 =
 O(\lambda^{-1-\beta}).
\end{equation*}
Note that $\dsharp=1$ if $d=2$ and $\dsharp=2$ if $d=3$.
For the case $d=3$ and $j=1$, by \eqref{DeltaA j},
we have, for small $\ve_0>0$,
\begin{equation*}
 \Delta_{1}(\lambda^2:x_0)A_{\beta,a}^{(1)}(\lambda^2)
 =
 O(\lambda^{-1-\beta+\ve_0}).
\end{equation*}
Hence
\begin{equation*}\label{K=DeltaA+O}
 \cK_{\beta,a}(\lambda^2:x_0) 
 =
 \begin{cases}
 \Delta_{0}(\lambda^2:x_0)A_{\beta,a}(\lambda^2)
 +
 O(\lambda^{-1-\beta}),
 & d=2, \\
 \Delta_{0}(\lambda^2:x_0)A_{\beta,a}(\lambda^2)
 +
 O(\lambda^{-1-\beta+\ve_0}),
 & d=3.
 \end{cases}
\end{equation*}
Since we can take $\ve_0>0$ as $-1-\beta+\ve_0<0$ in the case $d=3$,
we see that \eqref{K=o(1)} is equivalent to
\begin{equation}\label{DA=o(1)}
 \Delta_{0}(\lambda^2:x_0)A_{\beta,a}(\lambda^2)=o(1).
\end{equation}

\rm{(i)} 
Assume that the estimate \eqref{conj} holds.
Then,
for all $\beta>-1$, 
we can take $\ve>0$ as $-1-\beta+\ve<0$ and   
\begin{equation*}
 |\Delta_{0}(\lambda^2:x_0)A_{\beta,a}(\lambda^2)|
 \le
 C\lambda^{\frac{d-1}2+\ve}\lambda^{-(\frac d2+\beta+\frac12)}
 =
 C\lambda^{-1-\beta+\ve},
\end{equation*}
where we use \eqref{conj} and \eqref{Aaaj}.
This shows \eqref{DA=o(1)}.

\rm{(ii)} 
Conversely, assume that the estimate \eqref{DA=o(1)} holds
for all $\beta>-1$.
Then by \eqref{Aj-asymp} we have that
\begin{equation*}
 \frac{\Delta_{0}(\lambda^2:x)}{\lambda^{\frac d2+\beta+\frac12}}
 \cos\left(2\pi a\lambda-\frac{d+2\beta+1}4\pi\right)
 =o(1)
\end{equation*}
for all $\beta>-1$.
Now, for all $\ve>0$,
take $\beta(1)$ and $\beta(2)$ such that 
$-1<\beta(1)<-1+\epsilon=\beta(2)<\beta(1)+1$. 
Then 
\begin{equation}\label{DeltaCos}
 \frac{\Delta_{0}(\lambda^2:x)}{\lambda^{\frac d2+\beta(i)+\frac12}}
 \cos\left(2\pi a\lambda-\frac{d+2\beta(i)+1}4\pi\right)
 =o(1)
 \quad (i=1,2).
\end{equation}
Let $\theta_0=(\beta(2)-\beta(1))\pi/4$.
Then $0<\theta_0<\pi/4$
and 
\begin{equation*}
 \min_{\theta\in\R}
 \left(\max\left\{
  |\cos\theta|,|\cos(\theta-2\theta_0)|\right\}\right)
 =
 \min_{-\pi\le\theta\le\pi}
\left( \max\left\{
  |\sin\theta|,|\sin(\theta-2\theta_0)|\right\}\right)
 \ge
 \sin\theta_0.
\end{equation*}
Therefore,
\begin{equation*}
 \min_{\lambda>0}
\left( \max\left\{
 \left|
 \cos\left(2\pi a\lambda-\frac{d+2\beta(i)+1}4\pi\right)
 \right|: i=1,2
 \right\}\right)
 \ge
 \sin\theta_0.
\end{equation*}
Combining this and \eqref{DeltaCos}, 
and observing that $\beta(1)<\beta(2)$, 
we conclude that
\begin{equation*}
 \frac{\Delta_{0}(\lambda^2:x_0)}{\lambda^{\frac d2+\beta(2)+\frac12}}
 =
 \frac{\Delta_{0}(\lambda^2:x_0)}{\lambda^{\frac{d-1}2+\epsilon}}
 =
 o(1).
\end{equation*}
This shows the conclusion.
\end{proof}

\begin{proof}[Proof of Theorem~\ref{thm:conj d=4}]\label{proof:conj d=4}
From Theorem~\ref{thm:RT} it is enough to prove that
\begin{equation*} 
 \cK_{\beta,a}(\lambda^2:x_0) 
 =
 \sum_{j=0}^{\dsharp}(-1)^j\Delta_j(\lambda^2:x_0)A_{\beta,a}^{(j)}(\lambda^2)
 =
 o(1).
\end{equation*}
By the assumption and Lemma~\ref{lem:LPP1},
for all $\ve>0$,
\begin{equation*}
 \Delta_{\alpha}(s:x_0)
 =
 \begin{cases}
 O(s^{\frac{d-1}4+\ve}), & \text{if $\alpha=0$}, \\
 O(s^{\frac{d-1}2+\ve}), & \text{if $\alpha=\frac{d-1}2$}.
\end{cases}
\end{equation*}
Applying Theorem~\ref{thm:Riesz} 
as 
\begin{equation*}
 \alpha_0=0,\ \alpha_1=\frac{d-1}2\ \
 \text{and}\ \ \alpha=\theta\alpha_1, 
\end{equation*}
we have 
\begin{equation*}
 \Delta_{\alpha}(s:x_0)
 =
 O(s^{\frac{d-1}4+\frac{\alpha}2+\ve}),
 \quad\text{if}\ \ 0\le\alpha\le\frac{d-1}2,
\end{equation*}
since
\begin{equation*}
 (1-\theta)\frac{d-1}4+\theta\frac{d-1}2
 =\frac{d-1}4+\frac{\alpha}2. 
\end{equation*}
Take $\ve>0$ as $-1-\beta+\ve<0$. 
Then we have by \eqref{Aaaj} and \eqref{DeltaA kd}
\begin{align*}
 |\cK_{\beta,a}(\lambda^2:x_0)| 
 &\le
 \sum_{j=0}^{\dsharp-1}|\Delta_j(\lambda^2:x_0)A_{\beta,a}^{(j)}(\lambda^2)|
 +
 O(\lambda^{-1-\beta})
\\
 &\le C
 \sum_{j=0}^{\dsharp-1} 
  \lambda^{\frac{d-1}2+j+\ve}\lambda^{-(\frac d2+\beta+j+\frac12)}
 +
 O(\lambda^{-1-\beta})
\\
 &=O(\lambda^{-1-\beta+\ve})\\
 &=o(1).
\end{align*} 
The proof is complete.
\end{proof}


\vskip 5mm

\begin{flushright}
\begin{minipage}{80mm}
Shigehiko Kuratsubo \\
Department of Mathematical Sciences \\
Hirosaki University \\
Hirosaki 036-8561, Japan \\
E-mail: kuratubo@hirosaki-u.ac.jp \\
\end{minipage}
\begin{minipage}{80mm}
Eiichi Nakai \\
Department of Mathematics  \\
Ibaraki University  \\
Mito, Ibaraki 310-8512, Japan \\
E-mail: eiichi.nakai.math@vc.ibaraki.ac.jp 
\end{minipage}
\end{flushright}

\end{document}